\documentclass[11pt]{amsart}

\usepackage{mathtools}
\usepackage{amsmath, amsthm, amssymb, amsfonts, enumerate}
\usepackage{enumitem}
\usepackage{dsfont}
\usepackage{color}
\usepackage{geometry}
\usepackage{todonotes}
\usepackage{graphicx}
\usepackage{caption}
\usepackage{float}
\usepackage{soul}

\usepackage[unicode,colorlinks=true,linkcolor=blue,urlcolor=blue, citecolor=blue]{hyperref}
\usepackage[euler]{textgreek}

\mathtoolsset{showonlyrefs}

\def \cA{\mathcal{A}}
\def \cB{\mathcal{B}}
\def \cC{\mathcal{C}}
\def \cD{\mathcal{D}}

\def \cF{\mathcal{F}}
\def \cG{\mathcal{G}}
\def \cH{\mathcal{H}}

\def \cJ{\mathcal{J}}
\def \cK{\mathcal{K}}
\def \cL{\mathcal{L}}
\def \cM{\mathcal{M}}
\def \cN{\mathcal{N}}
\def \cO{\mathcal{O}}
\def \cP{\mathcal P}

\def \cS{\mathcal{S}}

\def \cZ{\mathcal{Z}}
\def \P{\mathsf P}

\def \E{\mathsf E}

\def \N{\mathbb{N}}
\def \R{\mathbb{R}}
\def \F{\mathbb{F}}
\def \G{\mathbb{G}}
\def \H{\mathbb{H}}
\def \bQ{\mathbb{Q}}

\def \ud{\mathrm{d}}
\def \e{\mathrm{e}}

\newcommand{\eps}{\varepsilon}

\newtheorem{theorem}{Theorem}[section]
\newtheorem{lemma}[theorem]{Lemma}
\newtheorem{corollary}[theorem]{Corollary}
\newtheorem{proposition}[theorem]{Proposition}

\newtheorem{definition}[theorem]{Definition}
\newtheorem{remark}[theorem]{Remark}

\theoremstyle{definition}

\DeclareMathOperator*{\argmax}{arg\,max}

\usepackage{xcolor}

\makeatletter
\@namedef{subjclassname@2020}{\textup{2020}Mathematics Subject Classification}
\makeatother

\title[Dynamic contracting with limited liability]{A continuous-time dynamic contracting problem \\ with limited liability and finite horizon}
\author[Bovo]{Andrea Bovo}
\author[De Angelis]{Tiziano De Angelis}
\author[Villeneuve]{Stephane Villeneuve}
\address{A.\ Bovo: Dept.\ Economics and Management, University of Brescia, Via San Faustino 74/B, 25122, Brescia, Italy.}
\email{\href{mailto:andrea.bovo@unito.it}{andrea.bovo@unibs.it}}
\address{T.\ De Angelis: School of Management and Economics, Dept.\ ESOMAS, University of Torino, Corso Unione Sovietica, 218 Bis, 10134, Torino, Italy; Collegio Carlo Alberto, Piazza Arbarello 8, 10122, Torino, Italy.}
\email{\href{mailto:tiziano.deangelis@unito.it}{tiziano.deangelis@unito.it}}
\address{S.\ Villeneuve: Toulouse School of Economics, University of Toulouse Capitole, 31000 Toulouse, France.}
\email{\href{mailto:stephane.villeneuve@tse-fr.eu}{stephane.villeneuve@tse-fr.eu}}
\date{\today}

\numberwithin{equation}{section}

\begin{document}

\begin{abstract}
We perform a detailed study of a principal--agent problem in a continuous time version of the celebrated Holmstr\"om--Milgrom model (Econometrica 55 (2), 1987) where we add limited liability for the Agent. We develop a probabilistic methodology to prove that the Principal's value function is the unique bounded classical solution to a fully nonlinear and fully degenerate partial differential equation (PDE) with Cauchy--Dirichlet boundary conditions on $[0,T]\times[0,\infty)$. Indeed, we also prove infinite continuous differentiability of the solution in the interior of the domain. The strength of our regularity result is such that we can ensure existence of optimal controls in strong form---a rare occurrence in dynamic contracting---and we obtain fine properties of the optimal control map, including a characterisation via a further nonlinear degenerate PDE.
\end{abstract}

\maketitle
\tableofcontents

\section{Introduction}

In dynamic contracting a Principal has the financial means to run a project but lacks the expertise. She can hire an Agent who is capable of running the project but requires a suitable compensation package in return. Mathematically the problem can be formulated as a nonzero-sum Stackelberg game in which first the Principal offers a contract, then the Agent determines her optimal level of effort (given the contract) and, finally, the Principal chooses the contract that strikes the optimal balance between expected output and associated costs. 

The output produced by the Agent is the result of a combination of the Agent's own effort and of random fluctuations generated by external factors. This creates the so-called {\em moral hazard} problem whereby the Agent may exert lower effort than what the Principal believes is actually exerted. To avoid this issue the Principal must restrict the class of contracts she is willing to offer to the class of so-called Incentive Compatible (IC) contracts. Such restriction effectively determines the Agent's optimal actions and reduces the whole Stackelberg game to a problem for the Principal of selecting an optimal IC contract.

Within this framework we consider the combination of finite-time horizon and limited liability for the Agent. The former means that the contract is resolved at a fixed maturity $T\in(0,\infty)$ and the latter means that the Agent does not accept contracts that may incur them a financial loss at the terminal time. The limited liability acts as a constraint in the contract making the finite-horizon problem very challenging and not well-understood. Recent work by Kr\v{s}ek and Possama\"i \cite{krvsek2026randomisation} proves existence of an optimal contract in a general framework that can accommodate the limited liability constraint as a special case. The existence result is obtained within a weak formulation in which the Principal and the Agent may use relaxed (randomised) controls. The methods rely on a special class of BSDEs driven by a martingale measure and compactness arguments. This approach is remarkable in its generality but it offers little insight into the form of the optimal contract. As the authors of \cite{krvsek2026randomisation} explain in their Introduction, solving the Principal's problem is essentially equivalent to finding a {\em classical} solution to a Hamilton-Jacobi-Bellman (HJB) equation for which ``{\em regularity results remain in general elusive [...] and one can at most hope for some form of Sobolev regularity, which is usually too weak for our purpose}''. 

Our paper tackles the latter issue. We investigate in detail the solvability of the HJB equation arising from a continuous-time dynamic contracting model with finite-time horizon and limited liability. Although the limited liability is a constraint on the terminal value of the contract, in fact it creates a dynamic constraint on the whole dynamics of the contract's value, which is bound to stay non-negative at all times. As a result, the Principal solves a stochastic control problem on the horizon $[0,T]$, in which she controls drift and diffusion of an arithmetic Brownian motion and the process is absorbed upon hitting zero for the first time. As a consequence the Principal's HJB equation is posed on the half-line with Cauchy condition at the terminal time $T$ and Dirichlet condition at zero. We prove that the Principal's value function is the unique bounded classical solution to the following fully nonlinear Cauchy-Dirichlet problem
\begin{equation}\label{eq:nonlinPDE}
\begin{cases}
&-\frac{1}{4\partial_{t}\psi(t,x)}+\tfrac{1}{2}\sigma^2\partial_{xx}\psi(t,x)+c_A\partial_{x}\psi(t,x)=c_A,\quad (t,x)\in[0,T)\times(0,\infty),\\
&\psi(t,0)=0,\qquad\ \, t\in[0,T),\\
&\psi(T,x)=0,\qquad x\in[0,\infty),
\end{cases}
\end{equation}
where $\sigma$ and $c_A$ are parameters specified in the model.
We also prove that such solution belongs to the class $C^\infty([0,T)\times(0,\infty))\cap C([0,T]\times[0,\infty))$ of continuous functions on $[0,T]\times[0,\infty)$ which are infinitely continuously differentiable in $[0,T)\times(0,\infty)$. We are then able to construct the dynamics of the optimal contract in strong formulation (i.e., adapted to the filtration observed by the Principal) and we prove fine properties of the Agent's optimal effort level as a function of time and wealth. Finally, we also obtain a probabilistic representation of the Agent's optimal effort map and we prove that it is a classical solution of another nonlinear PDE. 

The whole analysis is performed developing new probabilistic ideas for the direct study of the Principal's value function, setting our work apart from the majority of the existing literature on dynamic contracting. As we discuss later in this Introduction, our results also shed new light on the economic impact of the interplay between finite maturity of the contract and limited liability. The two features have not yet been considered simultaneously in the economic literature. 

Before delving deeper into our contributions to the literature, we would like to emphasise that our model starts from the continuous-time version of the classical model by Holmstr\"om and Milgrom \cite{holmstrom1987aggregation} and it adds to it the limited liability feature. This change alone produces a wealth of new challenges in the mathematical analysis and strong nonlinearities in the structure of the resulting optimal contracts. We should also mention that our Agent is risk-averse (CARA utility maximiser) with a quadratic cost of effort and the Principal is risk-neutral.

\subsection{Our mathematical contribution} 
In broad generality the modern approach to dynamic contracting problems was rigorously established by Cvitani\'c, Possama{\"\i} and Touzi \cite{cvitanic2018dynamic}, who show that all IC contracts can be written in terms of second-order BSDEs describing the dynamics over time of the contract's value. Their results formalise in rigorous mathematics and expand ideas that were proposed in Sannikov \cite{sannikov2008continuous} for an infinite-time horizon problem. The mathematical community has recently adopted this methodology in a variety of situations. Lin et al.\ \cite{lin2022random} and Possama{\"\i} and Touzi \cite{possamai2025golden} study contracts with random-horizon.
Elie and Possama{\"\i} \cite{ elie2019tale} consider principal--agent problems with many agents and mean-field
interactions. Mastrolia and Possama{\"\i} \cite{mastrolia2018moral} include ambiguity in principal--agent problems. A{\"\i}d, Possama{\"\i} and Touzi \cite{aid2022optimal} and Elie et al.\ \cite{elie2021mean} use dynamic contracting theory to design electricity demand response mechanisms.
The regulation of market making is considered in El Euch et al.\ \cite{eleuch2021optimal} and, finally, Gaussian outputs with memory are studied in
Abi Jaber and Villeneuve \cite{abijaber2025gaussian}.

It is well-known that in the Markovian setting, BSDEs arising from stochastic control are linked to viscosity solutions of HJB equations. For the particular HJB equations arising in dynamic contracting it is often difficult to obtain solutions that are sufficiently regular to justify a rigorous construction of the optimal contract, because the Agent's optimal effort should be in Markovian form as a function of the observed output. The function itself would normally involve the gradient and Hessian of the Principal's value function, which however may not enjoy the required regularity.

As mentioned above, in our model the Principal controls drift and volatility of an arithmetic Brownian motion. Because of the structure of the economic model, at each time $t$ the control exerted in the volatility corresponds to the square root of the control exerted in the drift. On the one hand this provides a good scaling property of the controlled dynamics, which we use in the solution of the problem. On the other hand, the problem is completely degenerate, in the sense that the second order differential operator associated to the controlled diffusion process does not satisfy the usual ellipticity condition required from the theory of linear and nonlinear PDEs (cf., \cite[Ch.\ 6]{evans10partial} and \cite[Ch.\ 10 and 12]{krylov2018sobolev}). Taking a leap of faith, at the level of the HJB we can calculate the maximiser in the Hamiltonian and plug it back into the equation. This yields \eqref{eq:nonlinPDE}, where the diffusive part of the operator is non-degenerate but a nasty nonlinearity appears for the time-derivative of the solution. 

In summary, the HJB is a fully nonlinear and fully degenerate PDE with Cauchy and Dirichlet boundary conditions. 
The form of the nonlinearity combined with the degeneracy of the diffusive operator in the HJB place 
our problem outside the scope of the existing PDE theory for classical solutions and $L^p$-viscosity solutions (as in, e.g., \cite{crandall1992user}). 
To make this point clearer, we notice that even after we make substantial effort to prove Sobolev regularity of the Principal's value function via direct probabilistic estimates (Proposition \ref{prop:W12p}), we cannot continue the analysis leveraging the theory of $L^p$-viscosity solutions for fully nonlinear PDEs (cf.\ \cite[Def.\ 2.1]{caffarelli1996viscosity}), in order to conclude that our HJB equation admits a strong solution (in the sense of Sobolev spaces). Ellipticity assumptions needed in the $L^p$-viscosity theory and expressed via the Pucci extremal operators are not satisfied in our setting (see \cite[Eq.\ (SC) in Sec.\ 2]{caffarelli1996viscosity} and \cite[Eq.\ (SC) in Sec.\ 1]{crandall2000Lptheory}). 

Other general results on viscosity solutions for non-linear PDEs do not apply either (cf.\ \cite{krylov2018sobolev} and \cite{taylor2023nonlinear}). By treating the $t$-variable as a spatial variable we can interpret our problem \eqref{eq:nonlinPDE} as a semi-linear elliptic problem which is completely degenerate in the $t$-variable (due to the absence of a term of the form $\partial_{tt}\psi$). In that case, the theory from \cite[Ch.\ 10]{krylov2018sobolev} does not seem to be applicable because, for example, Assumptions 10.1.6 and 10.1.11 concerning ellipticity and Lipschitz continuity of the nonlinearity (and necessary for Theorem 10.1.14) fail. Likewise the ellipticity property required in \cite[Ch.\ 14.3]{taylor2023nonlinear} fails (see also \cite[Eq.\ (0.3)]{crandall1999existence}). 
If instead we consider the $t$-variable as time, we rewrite \eqref{eq:nonlinPDE} as the fully non-linear parabolic PDE
\begin{equation*}
\partial_t\psi(t,x)-\frac{1}{4[\frac{\sigma^2}{2}\partial_{xx}\psi(t,x)+c_A\partial_{x}\psi(t,x)-c_A]}=0.
\end{equation*} 
In this case we are unable to use the results from \cite[Ch.\ 12]{krylov2018sobolev} because, for example, Assumptions 12.1.4 and 12.1.8 fail (these are the analogue in the parabolic case of Assumptions 10.1.6 and 10.1.11 mentioned earlier).

Given the difficulties illustrated above we propose a novel methodology, based entirely on the direct probabilistic analysis of the value function of the control problem. First we reformulate the original problem in terms of a family of auxiliary problems obtained via a time-change and parametrised by a constraint on the class of admissible controls (Section \ref{sec:aux}). Initially, this reformulation is heuristic but it will be rigorously established after some preliminary regularity results. In the new formulation the Brownian dynamics is uncontrolled and instead the control is somewhat transferred to the time-horizon of the problem. We perform several probabilistic bounds on the value function of the auxiliary problem that eventually lead us to showing its regularity in the Sobolev class $W^{1,2;p}$ of functions with weak derivatives in $L^p$ (one in time and two in space; cf.\ Proposition \ref{prop:W12p}). 

Later, using Proposition \ref{prop:W12p} and Theorem \ref{thm:DPP_time} we show that the value function of the auxiliary control problem solves \eqref{eq:nonlinPDE}, first in the Sobolev sense (Theorem \ref{thm:HJB}) and then in the classical sense (Theorem \ref{thm:smooth}). This is the step in our analysis where we realised that a direct use of the theory for $L^p$-viscosity solutions from \cite[Thm.\ 2.10]{caffarelli1996viscosity} or \cite[Prop.\ 2.10]{crandall2000Lptheory} does not seem legitimate under our assumptions. It is worth noticing that the proof of Theorem \ref{thm:smooth} is the only one based on PDE methods.
Finally, in Section \ref{sec:proofthm} we use a verification argument to rigorously establish the equivalence between the value functions of the original and of the auxiliary problem, and to construct the optimal contract in strong formulation. 

After completing the study of the Principal's value function we turn our attention to the specific properties of the optimal contract and of the Agent's optimal effort map $\beta^*$ (also referred to as pay-performance sensitivity (PPS)). In Section \ref{sec:econ} we show monotonicity and asymptotic properties of the optimal effort as function of time and wealth (Proposition \ref{prop:beta}). We prove that the function $\beta^*$ solves another nonlinear PDE (Proposition \ref{prop:betaP}) and we find an explicit probabilistic representation for it (Proposition \ref{prop:betaexact}). This analysis allows us to show that (Theorem \ref{thm:stop}) with positive probability the Agent stops exerting effort prior to the maturity of the contract---hence benefiting from her outside option---and that the Principal optimally pays an initial signing bonus to the Agent (Theorem \ref{thm:GH}). The economic implications of these results are discussed next.

\subsection{Our contribution to economic modelling}
Models of moral hazard and its applications have been recognised as very important in
understanding sharecropping contracts, delegated portfolio management and executive
compensation to cite just a few examples. 
Mirrlees \cite{mirrlees1976optimal} has shown that, when the Principal can use arbitrarily large
carrots and sticks as incentive tools, the optimal contract prevents the Agent from
shirking. In practice, limited liability laws curtail the size of penalties in
response to poor performance---in real life the Agent receives a minimum legal wage regardless of the performance. As a consequence, the
Principal is forced to give the Agent additional rents to provide incentive to make an
effort, thus increasing the size of bonuses while
limiting that of penalties.

This paper develops a tractable continuous-time principal--agent model that studies the
implications of a limited liability constraint when the contract has a fixed maturity
$T$. Work by Sappington \cite{sappington1983limited} and
Innes \cite{innes1990limited} pioneered the study of the impact of limited liability constraints
on optimal contracts. Focusing on limited liability, both papers have assumed that the
Principal and the Agent are risk-neutral and considered a single period model. We
consider here a continuous-time model, which enables a characterisation of the
optimal contract using partial differential equations and it can be used to compare
theoretical findings to corporate finance data that are inherently dynamic. 

Our model adds limited liability for the Agent in the framework of Holmstr\"om and Milgrom \cite{holmstrom1987aggregation}. In \cite{holmstrom1987aggregation} optimal
contracts are linear in the output process and the Agent is therefore exposed to
unbounded losses with a Gaussian distribution. Limited liability is a standard ingredient
of continuous-time contracting models, in which it usually takes the form of a
non-negativity constraint on the flow of payments to the Agent (see, e.g.,
 De Marzo and Sannikov \cite{demarzo2006optimal}, Sannikov \cite{sannikov2008continuous},
Biais et al.\ \cite{biais2010large}). However, so far there is no study concerning how
such a constraint interacts with a deadline (the finite maturity of a contract).

We show that combining limited liability with a finite horizon changes the picture in a
way that neither one produces alone. Early in the contract, an adverse shock can still
push the Agent's continuation utility down to the point of inefficient termination, so
the Principal must guard against this risk. As the deadline approaches, there is simply
less time left for such a shock to occur, and this precautionary concern fades on its
own. The horizon itself, rather than any decision of the Principal, is what relaxes the
constraint. This single mechanism, worked out in a finite-horizon model with CARA
preferences and limited liability, is what generates the paper's three economic results, 
illustrated in mathematical formalism in Section \ref{sec:econ}.

First, the optimal pay-performance sensitivity (PPS) rises monotonically over the life of the
contract and converges at maturity to the Holmstr\"om--Milgrom (HM) optimal level, for
any positive continuation value of the Agent (Proposition \ref{prop:beta}-(i)). Second, away from
maturity, PPS is increasing and concave in the Agent's continuation value (Proposition \ref{prop:beta}-(ii)). Because all contracts converge to the same terminal benchmark, the gap
between the PPS of the wealthier and less wealthy agents shrinks to zero at maturity.
This pattern connects to a broader concern in the empirical literature that PPS, as conventionally measured, does not distinguish incentives meant to operate over different horizons (see, e.g., Edmans and Gabaix \cite{edmans2016executive} on the divergence between wealth- and percentage-based PPS measures and practitioner discussions of vesting-schedule horizon). Rising incentives near the end of a contract are usually attributed to career concerns---the loss of reputational discipline as retirement approaches---which the Principal must offset with stronger contractual incentives (Gibbons and Murphy \cite{gibbons1992optimal}) or to short-termism, where managers close to departure have a high incentive to manipulate short-term performance measures (Marinovic and Varas \cite{marinovic2019ceo}). Our result identifies a third, independent mechanism: the same qualitative pattern arises purely as an optimal contracting response to the declining relevance of the liability boundary as the horizon shortens, with no role for reputation or performance manipulation.

Our third result concerns the start of the contract (Theorem \ref{thm:GH}): the Principal optimally grants the Agent an initial promised utility above the outside option, a signing premium that buffers against inefficient early termination. We name it the Golden Hello, and we show that it grows with the length of the contract. This rules out an explanation of the Golden Hello based solely on bargaining power, according to which its amount would be determined only by the Agent's external option and independently of the time horizon. Here, the premium corresponds to the capital committed from the outset to move away from 0 for the entire duration of the commitment. Thus, a longer commitment requires proportionally more capital.

\subsection{Structure of the paper}
The paper is organised as follows. In Section \ref{sec:formulation} we formulate the economic problem of dynamic contracting with limited liability and characterise the class of IC contracts using the methods from \cite{cvitanic2018dynamic}. In Section \ref{sec:problem} we specify the stochastic control problem for the Principal, we state the main results concerning the Principal's value function and the dynamics of the optimal contract in strong formulation; we also provide an informal outline of the main ideas that motivate the introduction of a family of auxiliary problems in the subsequent Section \ref{sec:aux}. In Sections \ref{sec:lipschitz} and \ref{sec:phixx} we perform a detailed probabilistic analysis of the value function for the auxiliary problems and we obtain an initial Sobolev regularity. The latter is then employed in Section \ref{sec:HJB} to solve the HJB equation first in the strong sense (a.e.) and then, thanks to a lift of the regularity, in the classical sense. In Section \ref{sec:proofthm} we prove that the solution of the HJB equation is indeed the value function of the original problem and derive the dynamics of the optimal contract. In Section \ref{sec:econ} we derive precise regularity results concerning the Agent's optimal effort map, which we characterise as a bounded solution of another PDE, and we obtain a probabilistic representation thereof. The paper is completed by a technical appendix.


\section{Formulation of the dynamic contracting problem}\label{sec:formulation}

Let $(\Omega,\mathcal{F},\P_0)$ be a probability space equipped with a Brownian motion $\bar B\coloneqq(\bar B_{t})_{{t \in [0,\infty)}}$. We denote the Brownian filtration completed with $\P_0$-null sets by $\bar \F=(\bar \cF_t)_{t\in[0,\infty)}$. Fix $T\in (0,\infty)$. The dynamics under $\P_0$ of the output process $(Y_t)_{t\in[0,T]}$ reads as
\begin{equation*}
Y_t = y +\sigma \bar B_t, \quad t\in [0,T],
\end{equation*}
for $y\in\R$. Both the Principal and the Agent observe the output process and therefore they have access to the filtration $\bar \F$.

In keeping with the existing literature (see, e.g. Cvitani\'c, Possama{\"\i} and Touzi \cite{cvitanic2018dynamic}) we formulate the problem in weak form: the Agent's actions induce a change of probability measure that adds a drift to the dynamics of the output. However, we point out immediately that we will find a solution of the problem in strong form. 

 The class of Agent's admissible actions is denoted by $\cA$ and it contains $\bar\F$-predictable, processes $(a_t)_{t\in[0,T]}$ such that 
\begin{equation*}
\E_0\Big[\exp\Big( \int_0^T \frac{a_s}{\sigma}\,\ud \bar B_s-\frac{1}{2}\int_0^T \Big\vert\frac{a_s}{\sigma}\Big\vert^2\,\ud s\Big)\Big]=1,
\end{equation*}
where $\E_0[\,\cdot\,]$ is the expectation under the measure $\P_0$.
Clearly $\cA\neq \varnothing$ because bounded processes are admissible. Associated to the class $\cA$ we have a family of equivalent probability measures $\{\P_a,\,a\in\cA\}$ given by
\begin{equation*}
\frac{\ud\P_a}{\ud\P_0}\Big|_{\bar\cF_T}=\exp\Big( \int_0^T \frac{a_s}{\sigma}\,\ud \bar B_s-\frac{1}{2}\int_0^T \Big\vert\frac{a_s}{\sigma}\Big\vert^2\,\ud s\Big).
\end{equation*}
Under $\P_a$, the process $B_t^a=\bar B_t-\int_0^t \sigma^{-1}a_s\,\ud s$ is a Brownian motion and $Y_t$ evolves as 
\begin{equation}\label{eq:dynX:Pa}
Y_t=y+\int_0^t a_s\,\ud s +\sigma B_t^a,\quad t\in[0,T].
\end{equation}

The Principal offers the Agent a contract based on the observed output. A contract is described mathematically by a $\bar\cF_T$-measurable function $\xi$. Because we are assuming the Agent's limited liability, we further restrict the class of admissible contracts to non-negative functions $\xi\ge 0$. In summary, the class of admissible contracts $\cC$ is given by
\begin{equation*}
\cC\coloneqq\big\{\xi:\Omega\to[0,\infty)\text{ s.t.\ $\xi$ is $\bar\cF_T$-measurable}\}.
\end{equation*}
Notice that there is asymmetry of information between Principal and Agent, because the former does not directly observe the latter's effort. That is, the Principal only observes the trajectory of the output process $(Y_t)_{t\in[0,T]}$ but does not observe the decomposition \eqref{eq:dynX:Pa}. In Economics, this asymmetry of information is known as {\em moral hazard}, because the Agent may ``pretend'' to exert more effort than she actually does in order to obtain a larger payoff.
 
The Principal-Agent problem is formulated as a Stackelberg game in which both players are utility maximisers. The Principal acts as the leader by choosing $\xi\in\cC$ and then the Agent picks a best response (if it exists) by choosing an optimal level of effort $(a^*_t(\xi))_{t\in[0,T]}$. That induces a probability measure $\P_{a^*}$. Finally, the Principal maximises her expected utility by choosing an optimal contract $\xi^*$. A contract $\xi\in\cC$ is said to be {\em implementable} if there is a best response $a^*(\xi)\in\cA$ for the Agent. 

For the Agent we follow the Holmstr\"om and Milgrom's model \cite{holmstrom1987aggregation}. The Agent adopts a utility function $U_A(x)$ with Constant Absolute Risk Aversion (CARA) and with risk-aversion coefficient $\gamma_A>0$. That is,
\begin{equation*}
U_A(x)\coloneqq-\exp(-\gamma_Ax),\quad \forall x\in\R.
\end{equation*}
Exerting effort is costly for the Agent. The cost of effort is modelled via a strictly convex, non-negative, twice continuously differentiable function $k:\R\to[0,\infty)$ satisfying $k(0)=0$. 

Given a contract $\xi\in\cC$, the Agent's promised utility at time zero is defined as
\begin{equation*}
V_0^A(\xi)=\sup_{a\in\cA} \E_a \Big[ U_A\Big(\xi - \int_0^T k(a_s) \,\ud s \Big)\Big],
\end{equation*}
where $\E_a[\,\cdot\,]$ is the expectation under the measure $\P_a$. If a contract $\xi$ is {\it implementable}, there exists $a^*(\xi)\in\cA$ such that 
\begin{equation*}
V_0^A(\xi)=\E_{a^*(\xi)}\Big[ U_A\Big(\xi - \int_0^T k(a^*_s(\xi)) \,\ud s \Big)\Big].
\end{equation*}
The optimal effort level is a mapping $\cC\ni\xi\mapsto a^\ast(\xi)\in\cA$ and indeed we will show in Section \ref{sec:IC} that strict convexity of the cost of effort guarantees that $a^*(\xi)$ exists and it is unique for each $\xi\in\cC$.

The Agent has a so-called {\em participation constraint}. That is, she only accepts contracts for which $V_0^A(\xi)\ge -R_0$, where $R_0 > 0$, and $-R_0$ is the Agent's reservation utility. A contract $\xi\in\cC$ is said to be {\em incentive compatible} if it is implementable and it satisfies the participation constraint. The class of incentive compatible contracts is denoted $\cC_{IC}\subset\cC$ and we will show in Section \ref{sec:IC} that $\cC_{IC}\neq\varnothing$. In fact, following \cite{cvitanic2018dynamic} we are going to {\em characterise all contracts} $\xi\in\cC_{IC}$.

The Principal is risk-neutral, i.e., she adopts a linear utility function $U_P(x)=x$, for $x\in\R$. Even though the Principal does not directly observe the Agent's effort, she can compute the mapping $\cC\ni\xi\mapsto a^\ast(\xi)\in\cA$. The Agent has no profitable deviations from $a^*(\cdot)$ and therefore the Principal anticipates the Agent's actions. Then, the Principal's problem reduces to solving
\begin{equation*}
V_0^P\coloneqq\sup_{\xi\in\cC_{IC}} \E_{a^*(\xi)}\big[Y_T-\xi \big],
\end{equation*}
and identifying an optimal contract $\xi^*\in\cC_{IC}$.

\subsection{Incentive-compatible contracts}\label{sec:IC}

Following Sannikov (see \cite{sannikov2008continuous}), we now characterise implementable contracts by means of the martingale optimality principle. 
\begin{lemma} \label{martingaleoptimality} 
Given $\xi\in\cC$, assume there is a family of stochastic processes $\{U^a(\xi),\,a\in\cA\}$ with $U^a(\xi)\coloneqq(U_t^a)_{t\in [0,T]}$ such that
\begin{itemize}
\item[i)] $U_T^a = U_A\big(\xi - \int_0^T k(a_s) \ud s\big)$, $\forall a \in \cA$,
\item[ii)] $(U_{t}^a)_{t\in[0,T]}$ is a $((\bar \cF_t)_{t\in [0,T]},\P_a)$-supermartingale $\forall a\in\cA$,
\item[iii)] $U_0^a=U_0\ge -R_0$, $\forall a\in\cA$,
\item[iv)] there exists $a^*\in\cA$, such that $(U^{a^*}_t)_{t\in[0,T]}$ is a $((\bar \cF_t)_{t\in [0,T]},\P_{a^*})$-martingale.
\end{itemize}
Then, $\xi\in\cC_{IC}$ and $a^*$ is an optimal effort for the Agent.
\end{lemma}

\begin{proof}
For any $a\in\cA$, conditions $i)-iv)$ immediately imply 
\begin{equation*}
\begin{aligned}
\E_a\Big[U_A\Big(\xi-\int_0^{T} k(a_s) \ud s\Big)\Big]&\overset{i)}{=}\E_a[U_T^a] \overset{ii)}{\leq} U_0^a \overset{iii)}{=} U_0^{a^*} \overset{iv)}{=} \E_{a^*}[U_T^{a^*}]\overset{i)}{=} \E_{a^*}\Big[U_A\Big(\xi-\int_0^{T} k(a_s^*) \ud s\Big)\Big].
\end{aligned}
\end{equation*}
This implies 
\begin{equation*}
V_0^A(\xi)=\E_{a^*}\Big[U_A\Big(\xi-\int_0^{T} k(a_s^*) \ud s\Big)\Big]=U_0\ge -R_0,
\end{equation*}
which shows optimality of $a^*$ and the participation constraint holds. 
\end{proof} 

Given a contract $\xi$, the process $(U_t^{a^\ast})_{t\in [0,T]}$ describes the Agent's utility along the optimal effort policy $a^*(\xi)$. Next we are going to construct a family $\{U^a(\xi),\,a\in\cA\}$ and an optimal effort policy $a^*(\xi)$ using the theory of BSDEs as in \cite{cvitanic2018dynamic}.

\begin{proposition}\label{prop:IC}
For $\xi\in\cC$ suppose that the pair $(\bar W,\bar Z)$ is solution to the following BSDE:
\begin{equation}
\label{BSDE:agent}
\bar W_t = \xi - \int_{t}^T f^*(\bar Z_s) \ud s -\int_t^T \bar Z_s \sigma \ud \bar B_s, \quad t\in [0,T], 
 \end{equation}
with $\bar W_0\ge -1/\gamma_A\log R_0$ and 
\begin{equation}\label{BSDEdrift}
 f^*(z)\coloneqq \frac{\gamma_A}{2} \sigma^2 |z|^2 + \inf_{a \in \R} \big[k(a)-a z\big].
 \end{equation}
Let $a^*(z)=(k^\prime)^{-1}(z)$ and suppose $a^*_t\coloneqq a^*(\bar Z_t)\in\cA$ and
\begin{equation}\label{eq:martZ}
\E_{a^*}\Big[\exp\Big(-\sigma\gamma_A\int_0^T \bar Z_s\ud B^{a^*}_s-\tfrac12\sigma^2\gamma^2_A\int_0^T \bar Z^2_s\ud s\Big)\Big]=1.
\end{equation}

Then $\xi\in\cC_{IC}$ and $a^*_t$ is the Agent's unique optimal action.
\end{proposition}
\begin{proof}
Leveraging the CARA assumption for the Agent's utility function, for any $a\in\cA$ we guess 
\begin{equation}\label{eq:Rat}
U^{a}_t\coloneqq-\exp\Big(-\gamma_A \Big(\bar W_t - \int_0^t k(a_s) \ud s \Big)\Big), \quad t\in [0,T]. 
\end{equation}
Observe that the pair $(\bar W_t,\bar Z_t)$ does not depend on $a\in\cA$ and thus $U_0^a=-\exp(-\gamma_A\bar W_0)\ge -R_0$ is independent of $a\in\cA$ as required in iii) of Lemma \ref{martingaleoptimality}. The process $(U^a_t)_{t\in[0,T]}$ is an It\^o process with dynamics 
\begin{equation*}
\ud U_t^a=-\gamma_A U_t^a \bar Z_t \sigma \ud B_t^a+\gamma_A U_t^a\Big( -f^*(Z_t)+k(a_t)+\frac{\gamma_A}{2} \sigma^2 |\bar Z_t|^2-a_t\bar Z_t\Big)\ud t,
\end{equation*}
where we recall that $B_t^a=\bar B_t-\int_0^t \sigma^{-1}a_s\,\ud s$ is a Brownian motion under $\P_a$.
Noticing that $U^a_t\le 0$ we deduce that $(U_t^a)_{t\in [0,T]}$ is a supermartingale for every $a \in \cA$ by definition of $f^*$.
 
Because $a\mapsto k(a)$ is strictly convex, there is a unique minimiser in \eqref{BSDEdrift}, denoted $a^*(z)=(k^\prime)^{-1}(z)$. Taking $a^*_t=a^*(\bar Z_t)$, by assumption we have $a^*\in\cA$ and 
\begin{equation*}
\ud U_t^{a^*}=-\gamma_A U_t^{a^*} \bar Z_t \sigma \ud B_t^{a^*},\quad t\in[0,T],
\end{equation*}
is a martingale under $\P_{a^*}$, thanks to \eqref{eq:martZ}. The family of processes $\{(U^a_t)_{t\in[0,T]},\,a\in\cA\}$ defined via \eqref{eq:Rat} satisfies (i)---(iv) of Lemma \ref{martingaleoptimality}.
Then $\xi\in\cC_{IC}$ and $a^*$ is the Agent's unique optimal effort, by uniqueness of the minimiser in \eqref{BSDEdrift}. 
\end{proof}

The above proposition shows that finding $\xi\in\cC_{IC}$ boils down to solving \eqref{BSDE:agent} and proving $a^*(\bar Z_t)\in\cA$ with \eqref{eq:martZ}. A tractable situation arises when the Agent's cost of effort is quadratic, i.e., $k(a)=a^2/2$. In that case the optimal effort reads $a^*(z)=z$ and plugging $f^\ast(z)=( \frac{\gamma_A\sigma^2-1}{2})z^2$ into \eqref{BSDE:agent}, we get
\begin{equation}\label{eq:BSDE}
\bar W_t(\xi)=\xi-\int_t^T \frac{\eta}{2}\bar Z^2_s(\xi)\,\ud s-\int_t^T \bar Z_s (\xi) \sigma \ud \bar B_s,
\end{equation}
with $\eta=\gamma_A\sigma^2-1$. When $\eta\neq 0$ we have a quadratic BSDE which is amenable to explicit solution arguing as in \cite[Thm.\ 3.1]{briand2007one} and assuming that
\medskip
 
\begin{itemize}
\item[\bf(H1)]\label{H1} If $\eta > 0$ then $\E_0[\xi^2]<\infty$, and if $\eta<0$ then $\E_0[\exp(-\eta/\sigma^2 \xi)]<\infty$. 
\end{itemize}
\medskip

\noindent We provide full details in the appendix (cf.\ Proposition \ref{prop:bsde}) and here we discuss further implications. Throughout we are going to use that $\E_0[\int_0^T\bar Z^2_t(\xi)\ud t]<\infty$, as shown in Proposition \ref{prop:bsde}.

Setting 
\begin{equation}\label{eq:defwidehatW}
\widehat W_t(\xi)\coloneqq \E_0\Big[ \e^{-\frac{\eta}{\sigma^2} \xi} \Big| \bar \cF_t\Big],\quad t\in[0,T],
\end{equation}
the integrability assumptions {\bf(H1)} imply that $(\widehat W_t(\xi))_{t\in[0,T]}$ is a martingale under $\P_0$. Moreover, it is shown in the proof of Proposition \ref{prop:bsde} that 
\begin{equation}\label{eq:CH}
\widehat W_t(\xi)=\exp\Big(-\frac{\eta}{\sigma^2}\bar W_t(\xi)\Big)
\end{equation}
and so, by It\^o's formula,
\begin{equation}\label{eq:dWhat}
\ud\widehat{W}_t(\xi)=-\frac{\eta}{\sigma} \widehat{W}_t(\xi)\bar Z_t(\xi)\ud\bar B_t.
\end{equation}
We notice that, because $a_t^*=\bar Z_t(\xi)$, condition \eqref{eq:martZ} reads equivalently as
\begin{equation*}
\E_{0}\Big[\exp\Big(-\frac\eta\sigma\int_0^T \bar Z_s(\xi)\ud \bar B_s-\frac12\frac{\eta^2}{\sigma^2}\int_0^T \bar Z^2_s(\xi)\ud s\Big)\Big]=1.
\end{equation*}
Then, \eqref{eq:martZ} is satisfied because
\begin{equation*}
\widehat W_t(\xi)=\widehat W_0(\xi)\exp\Big(-\frac\eta\sigma\int_0^t\bar Z_s(\xi)\ud \bar B_s-\frac12\frac{\eta^2}{\sigma^2}\int_0^t \bar Z^2_s(\xi)\ud s\Big),
\end{equation*}
is a $\P_0$-martingale by construction. Finally, we find the explicit form of $\bar W_t(\xi)$ as
\begin{equation}\label{eq:formulaforY}
\bar W_t(\xi)=-\frac{\sigma^2}{\eta} \log\Big( \E_0 \Big[ \e^{-\frac{\eta}{\sigma^2} \xi} \Big| \bar \cF_t\Big]\Big),\quad t\in[0,T].
\end{equation}

Such an explicit formula allows us to make sense more concretely of the incentive compatibility constraints.
Given $\xi\in\cC_{IC}$ satisfying {\bf(H1)}, the Agent's optimal expected utility and optimal actions are fully determined by the pair $(\bar W(\xi),\bar Z(\xi))$.
As a consequence, a contract $\xi\in\cC_{IC}$ satisfies the participation constraint (i.e., condition (iii) in Lemma \ref{martingaleoptimality}), if 
\begin{equation*}
U^a_0=-\exp\big(-\gamma_A \bar W_0(\xi)\big)=-\Big(\E_0 \Big[ \e^{-\frac{\eta}{\sigma^2} \xi} \Big]\Big)^\frac{\gamma_A\sigma^2}{\gamma_A\sigma^2-1}\ge -R_0.
\end{equation*} 
Equivalently we can write $\bar W_0(\xi) \ge (-1/\gamma_A)\log R_0$. Moreover, limited liability $\xi\ge 0$ is equivalent to $\bar W_t(\xi) \ge 0$, for every
$t \in [0,T]$. It is interesting to note that the limited liability constraint, which is a constraint at the date of payment of the contract, induces a dynamic constraint on the Agent's expected wage.

Finally, it is also useful to observe that 
the dynamics of the promised utility $\bar W_t(\xi)$ can be written as a forward SDE 
\begin{equation*}
\bar W_t(\xi)= \bar W_0(\xi)+\int_0^t \frac{\eta}{2}\bar Z^2_s(\xi)\,\ud s+\int_0^t \bar Z_s (\xi) \sigma \ud \bar B_s,
\end{equation*}
where $\bar W_0(\xi)=-\frac{\sigma^2}{\eta} \log( \E_0 [ \exp(-(\eta/\sigma^2) \xi)])$ is deterministic.
Letting $\bar \tau(\xi)\coloneqq\inf\{t\ge 0:\bar W_t(\xi)=0\}$ with $\inf\varnothing=\infty$ we notice 
\begin{equation*}
\mathds{1}_{\{\bar \tau(\xi)<T\}}\bar W_{\bar \tau(\xi)}(\xi)=0\iff \mathds{1}_{\{\bar \tau(\xi)<T\}}\widehat W_{\bar \tau(\xi)}(\xi)=\mathds{1}_{\{\bar \tau(\xi)<T\}}.
\end{equation*}
Therefore, on the event $\{\bar \tau(\xi)<T\}$ it must be 
$1=\widehat W_{\bar \tau(\xi)}(\xi)=\E_0\big[\e^{-\frac{\eta}{\sigma^2}\xi}\big|\cF_{\bar \tau(\xi)}\big]$, by definition of $\widehat W_{t}(\xi)$. The latter is equivalent to $\xi=0$, $\P_0$-a.s.\ on $\{\bar \tau(\xi)<T\}$. Hence, on the event $\{\bar \tau(\xi)<T\}$ we have $\widehat W_{s}(\xi)=1$, for all\footnote{For any pair of random times $\rho\le \tau$ we denote closed/open random intervals by $[\![\rho,\tau]\!]$ and $(\!(\rho,\tau)\!)$, respectively. Notice that in the general theory of stochastic processes $[\![\rho,\tau]\!]=\{(t,\omega):\rho(\omega)\le t\le \tau(\omega)\}$. Thus, when we write $t\in[\![\rho,\tau]\!]$ we are slightly abusing notation as we should write $(t,\omega)\in[\![\rho,\tau]\!]$, although the meaning does not change.}
 $s\in[\![\bar\tau(\xi),T]\!]$. Then, the quadratic variation of the process $(\widehat W_{s}(\xi))_{s\in [\bar\tau(\xi),T]}$ vanishes. That is, from \eqref{eq:dWhat} we deduce
\begin{equation*}
\mathds{1}_{\{\bar\tau(\xi)<T\}}\int_{\bar\tau(\xi)}^T\big(\widehat{W}_t(\xi)\big)^2 \bar Z^2_t(\xi)\ud t=0,\quad\P_0\text{-a.s.}
\end{equation*} 
Since $\widehat{W}_t(\xi)>0$ for all $t\in[0,T]$, the above equation implies that $\P_0$-a.s.
\begin{equation}\label{eq:Z0}
\bar Z_t(\xi)=0,\quad\text{for a.e.\ $t\in[\![\bar\tau(\xi),T]\!]$ on the event $\{\bar \tau(\xi)<T\}$.}
\end{equation}

If $\eta=0$ in \eqref{eq:BSDE} then the BSDE is linear and $\bar W_t(\xi)=\E_0[\xi|\bar\cF_t]$. It is not hard to see that also in this case $\bar W_t\ge 0$ for $t\in[0,T]$ and \eqref{eq:Z0} continues to hold, provided that $\xi\ge 0$ a.s. The case $\eta=0$ is less interesting and can be treated with the same methods developed in this paper. To simplify the exposition, from now on we only consider $\eta\neq 0$.

So far we have determined necessary conditions for a contract $\xi\in\cC$ to be incentive-compatible via properties of the dynamics of the pair $(\bar W(\xi),\bar Z(\xi))$. It turns out that those conditions are indeed also sufficient. 
\begin{proposition}\label{prop:allIC}
Suppose $k(a)=a^2/2$ and $\eta\neq 0$. We have $\xi\in\cC_{IC}$ satisfying {\em \bf(H1)} if and only if there is a pair of $\bar \F$-adapted processes $(\bar W^\beta_t,\beta_t)_{t\in[0,T]}$ such that $\bar W^\beta_T=\xi$ satisfies {\em \bf(H1)} and, for $x \ge \max[0,(-1/\gamma_A)\log R_0]$,
\begin{equation}\label{eq:dyna}
\bar W^\beta_t=x+\int_0^t \frac{\eta}{2}\beta^2_s \ud s+\int_0^t\beta_s \sigma\ud \bar B_s,\quad t\in[0,T],
\end{equation}
with $\beta_s=0$ for $s\in[\![\bar \tau_\beta,\infty)\!)$
where $\bar \tau_\beta\coloneqq\inf\{t\ge 0: \bar W_t^{\beta}\le 0\}$ with $\inf\varnothing=\infty$.
\end{proposition}
\begin{proof}
For the only if part of the statement we construct $\widehat W(\xi)$ as in \eqref{eq:defwidehatW} and we obtain the solution of the BSDE from a formula analogous to \eqref{eq:CH} (cf.\ proof of Proposition \ref{prop:bsde}). The discussion following \eqref{eq:formulaforY} guarantees the properties of the $\beta$-process in \eqref{eq:dyna}. 

For the if part, we start from the pair $(\bar W^\beta,\beta)$ that solves the SDE \eqref{eq:dyna}.
We set $\xi=\bar W^\beta_T$ and we observe that $\bar W_T^{\beta}$ is almost surely positive and even more, 
$\bar W_t^{\beta}\geq 0$, for $t\in [0,T]$, $\P_0$-a.s.
Then, $\xi\in\cC$. Moreover, $\bar W^\beta_T=\xi$ satisfies {\bf (H1)} by assumption and therefore the process $\widehat W$ defined in \eqref{eq:CH} with $\bar Z_s=\beta_s$ is a true martingale (cf.\ again the construction in the proof of Proposition \ref{prop:bsde}). As a result \eqref{eq:martZ} holds. Then Proposition \ref{prop:IC} guarantees $\xi\in\cC_{IC}$.
\end{proof}

\section{The Principal's problem: statement and main results}\label{sec:problem}

In this section we formulate the Principal's problem following the insight of Proposition \ref{prop:allIC}. Throughout the section we assume $k(a)=a^2/2$ so that all results from the previous sections hold. We recall $Y_t=y+\sigma\bar B_t$ and define the following class of admissible controls for the Principal.
\begin{definition}\label{def:B}
The class $\bar \cB$ contains $\bar \F$-progressively measurable processes $(\beta_t)_{t\in[0,T]}$ such that 
\begin{itemize}
\item[(i)] the integrability conditions hold
\begin{equation*}
\E_0\Big[\exp\Big( \int_0^T \frac{\beta_s}{\sigma}\,\ud \bar B_s-\frac{1}{2}\int_0^T \frac{\beta_s^2}{\sigma^2}\,\ud s\Big)\Big]=1\quad\text{and}\quad\E_\beta\Big[\int_0^T\beta^2_s\ud s\Big]<\infty,
\end{equation*}
\item[(ii)] given the It\^o process, for $x\ge 0$,
\begin{equation*}
\begin{aligned}
\bar W^\beta_t&=x+\int_0^t \frac{\eta}{2}\beta_s^2\mathrm{d}s+\int_0^t\beta_s \sigma\mathrm{d} \bar B_s,\quad t\in[0,T],
\end{aligned}
\end{equation*}
and the stopping time $\bar \tau_\beta\coloneqq\inf\{t\ge 0: \bar W_t^{\beta}\le 0\}$ with $\inf\varnothing=\infty$,
we have $\beta_s=0$ for $s\in[\![\bar \tau_\beta,T]\!]$ on the event $\{\bar\tau_\beta<T\}$. 
\end{itemize}
\end{definition}
In the above definition, for simplicity we do not include integrability conditions for $\bar W^\beta_T$ that mirror {\bf (H1)} for $\xi$. However, it turns out in our Theorem \ref{thm:main} that the optimal control $(\beta^*_t)_{t\in[0,T]}$ is bounded so that the optimal contract $\bar W^{\beta^*}_T=\xi$ satisfies {\bf (H1)}.

Given $\beta\in\bar \cB$, we define the measure $\P_\beta$ via its Radon-Nikodym derivative
\begin{equation}\label{eq:Pbeta}
\frac{\ud \P_\beta}{\ud \P_0}\Big|_{\bar\cF_T}\coloneqq\exp\Big( \int_0^T \frac{\beta_s}{\sigma}\,\ud \bar B_s-\frac{1}{2}\int_0^T \Big(\frac{\beta_s}{\sigma}\Big)^2\,\ud s\Big).
\end{equation}
Under $\P_\beta$ the process $B^\beta_t=\bar B_t-\int_0^t\frac{\beta_s}{\sigma}\ud s$ is a Brownian motion and, setting $Y=Y^\beta$, the previous dynamics read for $t\in[0,T]$
\begin{equation}\label{eq:SDEYW}
Y_t^{\beta}=y+\int_0^t\beta_s\, \ud s+ \sigma B^\beta_{t},\quad
\bar W_t^{\beta}=x+\int_0^t c_A \beta_s^2\,\ud s+\int_0^t\beta_s\sigma \,\ud B^\beta_s,
\end{equation}
where we set $c_A=c_A(\sigma)=(\gamma_A\sigma^2+1)/2$, to simplify the notation. Sometimes we denote $Y^{y;\beta}$, $\bar W^{x;\beta}$ and $\bar \tau_\beta(x)$ to account explicitly for the initial position of the processes $\bar W^\beta$ and $Y^\beta$. 

In order to embed the Principal's problem in a dynamic framework, given $t\in[0,T]$, the remaining time horizon of the Principal is $T-t$. The Principal's payoff from using a control $\beta\in\bar \cB$, evaluated taking expectation $\E_\beta[\,\cdot\,]$ under the measure $\P_\beta$, reads as 
\begin{equation*}
\begin{aligned}
\cJ_{t,x,y}( \beta)\coloneqq& \E_\beta\Big[Y_{T-t}^{y;\beta}- \bar W_{T-t}^{x;\beta}\Big]=\E_{\beta}\Big[ Y_{T-t}^{y;\beta}-\bar W_{(T-t)\wedge\bar \tau_{\beta}}^{x;\beta}\Big]\\
=&\E_{\beta}\Big[Y_{(T-t)\wedge\bar \tau_{\beta}}^{y;\beta}+ \sigma\big(B^\beta_{T-t}- B^\beta_{(T-t)\wedge\bar \tau_{\beta}}\big)-\bar W_{(T-t)\wedge\bar \tau_{\beta}}^{x;\beta}\Big]=\E_{\beta}\Big[ Y^{y;\beta}_{(T-t)\wedge\bar \tau_{\beta}}-\bar W_{(T-t)\wedge\bar \tau_{\beta}}^{x;\beta}\Big],
\end{aligned}
\end{equation*}
where the first and second equalities hold because $\beta_s=0$ for $s\in[\![\bar \tau_\beta,\infty)\!)$ and the third one holds by optional sampling. Finally, using the explicit dynamics of $Y^\beta$ and $\bar W^\beta$ the payoff reduces to 
\begin{equation*}
\cJ_{t,x,y}(\beta)=y-x+\E_\beta\Big[\int_0^{(T-t)\wedge\bar \tau_{\beta}}\beta_s\big(1-c_A\beta_s\big)\ud s\Big],
\end{equation*}
where the quantity under expectation depends on $x$ via the stopping time $\bar\tau_\beta=\bar\tau_\beta(x)$, it depends explicitly on the remaining time-horizon $T-t$ and it is independent of the initial output level $y$.
Then, the value function of the Principal's problem reads as
\begin{equation*}
u(t,x,y)\coloneqq\sup_{\beta\in\bar{\cB}}\cJ_{t,x,y}(\beta)=y-x+\bar\varphi(t,x),
\end{equation*}
with 
\begin{equation}\label{eq:orig_value}
\bar{\varphi}(t,x)\coloneqq \sup_{\beta\in\bar{\cB}}\E_\beta\Big[\int_0^{(T-t)\wedge\bar \tau_\beta(x)}\beta_s\big(1-c_A\beta_s\big)\,\ud s\Big].
\end{equation}
\begin{remark}\label{rem:barphi}
The notation in \eqref{eq:orig_value} is mildly redundant. In particular, by definition of the class $\bar\cB$ the dynamics $\bar W^{\beta}$ is absorbed at zero and therefore 
\begin{equation*}
\int_0^{(T-t)\wedge\bar \tau_\beta}\beta_s\big(1-c_A\beta_s\big)\,\ud s=\int_0^{T-t}\beta_s\big(1-c_A\beta_s\big)\,\ud s.
\end{equation*}
Then, if we keep the notation with the stopping time $\bar\tau_\beta$ in \eqref{eq:orig_value}, we can also remove the constraint $\beta_s=0$ for $s\in[\![\bar \tau_\beta,\infty)\!)$ from the definition of the class $\bar \cB$.
\end{remark}
\begin{remark}\label{rem:timehomo}
It is useful in some parts of our analysis to keep track of the maturity of the contract and denote the function in \eqref{eq:orig_value} as $\bar\varphi(t,x;T)$. Because the Brownian motion is time-homogeneous and the set of admissible controls is time-invariant, it is easy to deduce from \eqref{eq:orig_value} that $\bar\varphi(t,x;T)=\bar\varphi(0,x;T-t)$. This will be used in the proof of Theorem \ref{thm:GH}.
\end{remark}

One of our key results is a characterisation of the Principal's value function and of the Agent's optimal effort. These results are stated in the next theorem, whose proof is given in Section \ref{sec:proofthm} after several steps are performed in the next two sections. Later, in Section \ref{sec:econ} we analyse in further detail the structure of the Agent's optimal effort, in order to gain deeper economic insight. 
\begin{theorem}\label{thm:main}
The function $\bar{\varphi}$ belongs to $C^{\infty}((0,T)\times (0,\infty))\cap C([0,T]\times[0,\infty))$. It is concave and decreasing in $t$, with $\partial_t\bar\varphi$ taking values in $[-1/(4c_A),0]$, it is concave and increasing in $x$ and it is the unique bounded classical solution of:
\begin{equation*}
\begin{cases}\displaystyle\partial_{t}\psi+\sup_{\beta\in\R}\Big[\tfrac{1}{2}\sigma^2\beta^2\partial_{xx}\psi+c_A\beta^2\partial_{x}\psi+(1-c_A\beta)\beta\Big]=0,& \text{on }[0,T)\times(0,\infty),\\
\psi(t,0)=0,& \forall t\in[0,T],\\
\psi(T,x)=0,& \forall x\in[0,\infty).
\end{cases}
\end{equation*}

The dynamics $\bar W^*=\bar W^{*;t,x}$ of the optimal contract (starting at $(t,x)$) is the unique $\bar \F$-adapted solution of 
\begin{equation}\label{eq:W*tx}
\bar W_s^{*}=x+4c_A\int_0^s\mathds{1}_{\{u<\bar \tau_*\}}\big(\partial_{t}\bar{\varphi}(t+u,\bar W_u^{*})\big)^2 \,\ud u-2\int_0^s\mathds{1}_{\{u<\bar \tau_*\}}\partial_{t}\bar{\varphi}(t+u,\bar W_u^{*})\sigma \,\ud B^*_u,
\end{equation}
for $s\in[0,T-t]$, where\footnote{Notice that indeed $\bar \tau_*=\bar\tau^{t,x}_*$ depends on $(t,x)$.} 
$\bar\tau_*\coloneqq\bar \tau_{\beta^*}=\inf\big\{s\in[0,\infty):\bar W^{*}_s\le 0\big\}$, with $\inf\varnothing=+\infty$,
and $B^*_t=B^{\beta^*}_t$ is a Brownian motion under $\P_*=\P_{\beta^*}$, given by 
$B^*_t=\sigma^{-1}(Y_t-Y_0-\int_0^t\beta^*_s\ud s)$. 

Finally, the Agent's optimal effort reads as $\beta^*_{s}=-2\mathds{1}_{\{s<\bar \tau_*\}}\partial_t\bar{\varphi}(t+s,\bar W^*_s)$, hence $\beta^*_s\in [0,1/(2c_A)]$, for $s\in[0,T-t]$.
\end{theorem}

Recalling that the Principal only observes the output process $Y$, we notice that indeed the dynamics of the optimal contract is neatly expressed in terms of such output process:
\begin{equation*}
\bar W_s^{*}=x+4(c_A-1)\int_0^s\mathds{1}_{\{u<\bar \tau^{t,x}_*\}}\big(\partial_{t}\bar{\varphi}(t+u,\bar W_u^{*})\big)^2 \,\ud u-2\int_0^s\mathds{1}_{\{u<\bar \tau^{t,x}_*\}}\partial_{t}\bar{\varphi}(t+u,\bar W_u^{*}) \,\ud Y_u.
\end{equation*}

\subsection{Heuristics of a time-change}

In this section we illustrate ideas that lead us to consider an auxiliary stochastic control problem. The latter turns out to be amenable to solution and is equivalent to the original problem. The first observation is that the only relevant dynamics for the study of \eqref{eq:orig_value} is the dynamics of $\bar W^\beta$. Then we notice that at least formally we can perform a time change using $T_\beta(t)=\int_0^t\beta^2_s\ud s$ to write
\begin{equation*}
\bar W^\beta_t=x+c_A T_\beta(t)+\sigma \widetilde B_{T_\beta(t)}\eqqcolon X_{T_{\beta}(t)},
\end{equation*} 
where $\widetilde B$ is a new Brownian motion and $X_s=x+c_A s+\sigma \widetilde B_s$. Assuming that the inverse $S_\beta(t)$ of $T_\beta(t)$ is well-defined and changing variable of integration in the cost function we obtain
\begin{equation*}
\begin{aligned}
\E_\beta\Big[\int_0^{T\wedge \bar \tau_\beta}\!\!\beta_s\big(1\!-\!c_A \beta_s\big)\ud s\Big]=\E_\beta\Big[\int_0^{T\wedge \bar \tau_\beta}\!\!\Big(\frac{1}{\beta_s}\!-\!c_A \Big)\ud T_\beta(s)\Big]=\E_\beta\Big[\int_0^{T_\beta(T\wedge \bar \tau_\beta)}\!\!\Big(\alpha_u\!-\!c_A \Big)\ud u\Big],
\end{aligned}
\end{equation*}
with $\alpha_u=1/\beta_{S_\beta(u)}$. The advantage of this formulation is that the dynamics of $X$ is uncontrolled and extremely simple. The time-horizon rewrites as $T_\beta(T\wedge \bar \tau_\beta)=T_\beta(T)\wedge T(\bar \tau_\beta)$, where $T(\bar \tau_\beta)=\inf\{t\ge 0:X_t\le 0\}$ and $T_\beta(T)=\int_0^T\alpha^2_s\ud s\eqqcolon Z^\alpha_T$ (not to be confused with the process $\bar Z$ in \eqref{BSDE:agent}). Thus we have a new controlled process $Z^\alpha$ which effectively determines the length of the time horizon in the optimisation via the choice of the control process $\alpha_s$. Intuitively one should not choose $\alpha_s<c_A$ because that generates a net cost. Thus we expect $\alpha_s\ge c_A$. Likewise it makes sense that when $X$ is large, it should be suboptimal to use a large amount of control which would rapidly bring about the maturity. In particular, it will be shown that if we drop the stopping time $T(\bar \tau_\beta)$, then the optimal control is constant and equal to $2c_A$. On the contrary, when $X$ is close to zero, there is an incentive to use a large amount of control because shortening the time-horizon becomes irrelevant in the face of the high probability of $X$ hitting zero. This suggests that the class of admissible controls can be restricted to those taking values in an interval $\alpha_s\in [c_A, f(X_s)]$ where we expect $f(\infty)=2 c_A$ and we allow $f(0)=+\infty$. Choosing appropriately the function $f$ we expect to incur no change in the value of the optimisation. 

It turns out that all the heuristics above can be made rigorous and indeed the time-changed problem is amenable to an extremely detailed probabilistic analysis. On the strength of these considerations, in the next section we introduce a family of auxiliary control problems.


\section{Auxiliary problem}\label{sec:aux}
In this section, we consider a family of auxiliary problems which help us to solve the original one. These problems are cast in {\em strong} formulation in the sense that on a given probability space $(\Omega,\cF, \P)$, we consider a {\em fixed} Brownian motion $B=(B_t)_{t\in[0,\infty)}$ and the filtration $\F\coloneqq(\cF_t)_{t\in[0,T]}$ generated by the Brownian increments. We could in principle choose $(B,\P)=(\sigma^{-1}(Y-Y_0),\P_0)$ but this is not necessary. The filtration is right-continuous and augmented with the $\P$-null sets.

We introduce the class of processes
\begin{equation*}
\cA(\F)=\big\{\alpha\,|\,(\alpha_s)_{s\in[0,\infty)}\text{ is $\F$-progr.\ meas.\ and $\alpha_s\ge c_A$ for all $s\ge0$}\big\},
\end{equation*}
with $c_A>0$ as in \eqref{eq:SDEYW}.
Given $\alpha\in\cA(\F)$ we introduce processes $Z$ and $X$, defined as
\begin{equation*}
Z_s^{t;\alpha}=t+\int_0^s\! \alpha_v^2\,\ud v\quad\text{and}\quad X_s^x=x+c_A s+ \sigma B_s,\qquad s\ge 0,
\end{equation*}
with starting point $(t,x)\in[0,T]\times[0,\infty)$. We also define the stopping times 
\begin{equation}\label{eq:sigmat_rhox}
\sigma_{t}^\alpha\coloneqq\inf\big\{s\ge 0: Z_s^{t;\alpha} = T\big\}\quad\text{and}\quad\rho_x\coloneqq\inf\big\{s\ge0: X_s^x= 0\big\},\quad \inf\varnothing=\infty.
\end{equation}
Let $K:\R\to [0,\frac{1}{4c_A}]$ be continuous, non-decreasing and such that $K(x)>0$ for $x>0$ and $K(x)=0$ for $x\le 0$. For any such function $K$, and adopting the convention $1/0=+\infty$, we define the class of admissible controls $\cA^K_x(\F)\subset\cA(\F)$ as
\begin{equation}\label{eq:AK}
\cA^{K}_x(\F)=\Big\{\alpha\in\cA(\F)\,\Big|\, c_A\le \alpha_s\le \tfrac{1}{2}[K(X^x_s)]^{-1}\quad\text{for all $s\ge0$}\Big\}.
\end{equation}
Notice that for all values of $x\in\R$ the interval $[c_A,1/(2K(x))]\supseteq[c_A,2 c_A]$ and therefore the class $\cA^K_x(\F)$ is nonempty.

Given a function $K(x)$ as above, the objective function of the auxiliary problem reads as
\begin{equation}\label{eq:vphiK}
\varphi^K(t,x)\coloneqq\sup_{\alpha\in\cA^K_x(\F)}\E\Big[\int_0^{\sigma_t^\alpha\wedge \rho_x}\!(\alpha_s-c_A)\ud s\Big], \quad (t,x)\in[0,T]\times[0,\infty).
\end{equation}
For now we allow $K$ to be generic, but after Proposition \ref{prop:d_time} we are going to fix one specific $K(x)=K_*(x)$ and denote the associated value function $\varphi^{K_*}$ simply by $\varphi^*$. Eventually we show that $\varphi^{K_*}$ is equal to the value function $\bar \varphi$ from \eqref{eq:orig_value}. Thus, every result stated throughout the rest of the paper for a generic $\varphi^K$ also holds for the Principal's original value function $\bar \varphi$.

Because we are working with the Brownian filtration $\F$, the Dynamic Programming Principle (DPP) holds for the auxiliary problem (see, e.g., \cite{bouchard2012dynamic} or \cite{bouchard2011weak}) and it will be used throughout the paper. That is, we have
\begin{equation}\label{eq:DPP}
\varphi^K(t,x)=\sup_{\alpha\in\cA^K_x(\F)}\E\Big[\int_0^{\tau\wedge \sigma_t^{\alpha}\wedge\rho_x}(\alpha_u-c_A)\ud u+\varphi^K\big(Z^{t;\alpha}_{\tau\wedge \sigma_t^{\alpha}\wedge\rho_x},X_{\tau\wedge \sigma_t^{\alpha}\wedge\rho_x}^x\big)\Big],
\end{equation}
for any $\F$-stopping time $\tau$. The emphasis on the filtration in the notation $\cA^K_x(\F)$ will only be needed later in the proof of Theorem \ref{thm:DPP_time}. For the time being, we use a leaner notation $\cA^K_x=\cA^K_x(\F)$. We notice that a proof of \eqref{eq:DPP} requires continuity of $\varphi^K$, which we obtain in Proposition \ref{cor:contphi}.

Set $\cK'(x)\coloneqq[c_A,1/(2K(x))]$ for notational convenience and let
\begin{equation}\label{eq:hamil}
\cM_K(x,p)\coloneqq\sup_{\alpha\in\cK'(x)}\big(\alpha^2p+\alpha\big),
\end{equation}
be the Hamiltonian for the problem above. The HJB equation that we expect $\varphi^K$ should solve can be stated as 
\begin{equation*}
\tfrac12\sigma^2\partial_{xx}\psi(t,x)+c_A\partial_x \psi(t,x)+\cM_K\big(x,\partial_t\psi(t,x)\big)=c_A,
\end{equation*}
with $\psi(t,0)=0$ for $t\in[0,T)$ and $\psi(T,x)=0$ for $x\in[0,\infty)$.
The analysis in this section prepares the ground to show that $\varphi^K$ is classical solution of the above equation.

In order to prove some bounds on $\varphi^K$ we need to introduce also a deterministic control problem.
Let $\cA_d$ be the class of measurable functions $[0,\infty)\ni t\mapsto a(t)$ bounded from below by $c_A$. Consider the mapping $[0,\infty)\times[0,T]\times\cA_d\ni(s,t,a)\mapsto \cZ(s,t,a)=t+\int_0^s a^2(u)\ud u$. 
Let us denote
\begin{equation}\label{eq:phi(t)}
\phi(t)\coloneqq\sup_{a\in\cA_d}\int_0^{\eta(t,a)}\!\big(a(s)-c_A\big)\ud s,\qquad t\in[0,T],
\end{equation}
with $\eta(t,a)=\inf\{s\ge 0\,:\,\cZ(s,t,a)=T\}$. It is not hard to show that 
$\phi(t)=\frac{T-t}{4c_A}$ by solving the associated HJB:
\begin{equation*}
\sup_{\alpha\in[c_A,\infty)}\big(\alpha^2\dot{\phi}(t)+\alpha-c_A\big)=0,\quad\text{for $t\in[0,T)$},
\end{equation*}
with terminal condition $\phi(T)=0$.

\subsection{Lipschitz bounds and concavity in time}\label{sec:lipschitz}

We begin our analysis by proving boundedness and continuity of the function $\varphi^K$. After that we are able to use the DPP \eqref{eq:DPP} to prove Lipschitz continuity and concavity in time, along with its time-monotonicity. 

\begin{proposition}\label{cor:contphi}
The function $\varphi^K$ is continuous on $[0,T]\times[0,\infty)$ and bounded with
\begin{equation*}
0\le \varphi^K(t,x)\le \phi(t)=\frac{T-t}{4c_A},\quad (t,x)\in[0,T]\times[0,\infty).
\end{equation*}
\end{proposition}
\begin{proof}
First we prove the bounds.
Fix $(t,x)\in[0,T]\times[0,\infty)$. Any $\alpha\in\cA^K_x$ satisfies $\alpha_s-c_A\ge0$ for all $s\ge0$ and so $\varphi^K\ge 0$ by the arbitrariness of $(t,x)$. Clearly,
\begin{equation*}
\begin{aligned}
\varphi^K(t,x)&\le \sup_{\alpha\in\cA^K_x}\E\Big[\int_0^{\sigma_t^\alpha}\!(\alpha_s-c_A)\ud s\Big]\\
&\le\sup_{\alpha\in\cA}\E\Big[\int_0^{\sigma_t^\alpha}\!(\alpha_s-c_A)\ud s\Big]\le \frac{1}{4c_A}\sup_{\alpha\in\cA}\E\Big[\int_0^{\sigma_t^\alpha}\!(\alpha_s)^2\ud s\Big]= \phi(t),
\end{aligned}
\end{equation*}
where the first inequality holds because $\rho_x\wedge\sigma^\alpha_t\le \sigma^\alpha_t$ and $\alpha_s-c_A\ge0$ for $s\ge0$ for $\alpha\in\cA^K_x$, and for the second inequality we use $\cA^K_x\subset\cA$. The third inequality holds because $1/(4c_A) \alpha^2-\alpha+c_A\ge 0$ for $\alpha\in\R$. The equality at the end holds by definition of $\sigma_t^\alpha$. 
Thus the upper bound on $\varphi^K$ holds. 

Let us now prove continuity. The starting observation is that 
\begin{equation*}
\varphi^K(t,x)=\sup_{\alpha\in\cA}\E\Big[\int_0^{\sigma_t^\alpha\wedge \rho_x}\!\Big(\min\big\{\alpha_s,\tfrac12\big[K(X^x_s)\big]^{-1}\big\}-c_A\Big)\ud s\Big]\eqqcolon \sup_{\alpha\in\cA}\cP_{t,x}(\alpha).
\end{equation*}
This is easily seen because $\cA_x^K\subset\cA$ but for any $\alpha\in\cA$ the process $\tilde \alpha_s\coloneqq\min\{\alpha_s,\frac12[K(X^x_s)]^{-1}\}$ belongs to $\cA_x^K$ and $\sigma^\alpha_t \le \sigma^{\tilde\alpha}_t$. Because $\tilde\alpha_s-c_A\ge 0$ for all $s\ge0$, it follows that
\begin{equation*}
\cP_{t,x}(\alpha)=\E\Big[\int_0^{\sigma^\alpha_t\wedge\rho_x}\!\big(\tilde\alpha_s-c_A\big)\ud s\Big]
\le\E\Big[\int_0^{\sigma^{\tilde\alpha}_t\wedge\rho_x}\!\big(\tilde\alpha_s-c_A\big)\ud s\Big]\le\varphi^K(t,x).
\end{equation*}
We also notice for the later use of dominated convergence that 
\begin{equation}\label{eq:asb}
\int_0^{\infty}\!\mathds{1}_{\{s<\sigma_t^\alpha\wedge \rho_x\}}\min\big\{\alpha_s,\tfrac12\big[K(X^x_s)\big]^{-1}\big\}\ud s\le \frac{1}{4c_A}\int_0^{\sigma_t^\alpha}\big(\alpha_s\big)^2\ud s\le \phi(t).
\end{equation}

Fix $(t,x)\in[0,T]\times[0,\infty)$ and let $(t_n,x_n)_{n\in\N}\subseteq [0,T]\times[0,\infty)$ be any sequence converging to $(t,x)$. 
It is not hard to see that for every $\alpha\in\cA$ 
\begin{equation*}
\begin{aligned}
\lim_{n\to\infty}\cP_{t_n,x_n}(\alpha)=\cP_{t,x}(\alpha),
\end{aligned}
\end{equation*}
thanks to a.s.\ convergence of $(\sigma_{t_n}^\alpha,\rho_{x_n},K(X^{x_n}_s))$ to $(\sigma_{t}^\alpha,\rho_{x},K(X^{x}_s))$ and the bound in \eqref{eq:asb}. Then $(t,x)\mapsto \cP_{t,x}(\alpha)$ is continuous and $\varphi^K$ is lower semi-continuous as supremum of continuous functions. 

Thanks to the explicit form of the stopping times $\sigma^\alpha_t$ we can prove stronger regularity in time for $\cP_{t,x}(\alpha)$. This will be needed later to show upper semi-continuity of $\varphi^K$. For any $s,t\in[0,T]$, $s<t$ and a fixed $\alpha\in\cA$ we have $Z^{s;\alpha}_u<Z^{t;\alpha}_u$ for all $u\in[0,\infty)$. Then
\begin{equation*}
\begin{aligned}
\sigma^\alpha_s&=\sigma^\alpha_t\!+\!\inf\Big\{u\ge 0: Z^{s;\alpha}_{\sigma^\alpha_t}\!+\!\int_0^u\!\big(\alpha_{\sigma^\alpha_t+r}\big)^2\ud r=T\Big\}=\sigma^\alpha_t\!+\!\inf\Big\{u\ge 0: \int_0^u\!\big(\alpha_{\sigma^\alpha_t+r}\big)^2\ud r=t-s\Big\},
\end{aligned}
\end{equation*}
where we used that $Z^{s;\alpha}_{\sigma^\alpha_t}=Z^{t;\alpha}_{\sigma^\alpha_t}-(t-s)=T-(t-s)$. Because $\alpha_u\ge c_A$ for all $u\in[0,\infty)$ we reach the uniform bound for all $\omega\in\Omega$
\begin{equation*}
0\le \sup_{\alpha\in\cA}\big(\sigma^\alpha_s(\omega)-\sigma^\alpha_t(\omega)\big)\le \frac{t-s}{c_A^2}.
\end{equation*}
Let us now fix $x\in[0,\infty)$, $s<t$, $s,t\in[0,T]$ and $\alpha\in\cA_x^K$. We have 
\begin{equation*}
\begin{aligned}
0&\le \int_0^{\rho_x\wedge\sigma^\alpha_s}\big(\alpha_r-c_A\big)\ud r-\int_0^{\rho_x\wedge\sigma^\alpha_t}\big(\alpha_r-c_A\big)\ud r\\
&=\int_{\rho_x\wedge\sigma^\alpha_t}^{\rho_x\wedge\sigma^\alpha_s}\alpha_r\ud r\le \sqrt{\sigma^\alpha_s-\sigma^\alpha_t}\Big(\int_{\rho_x\wedge\sigma^\alpha_t}^{\rho_x\wedge\sigma^\alpha_s}\big|\alpha_r\big|^2\ud r\Big)^{\frac12}\le \sqrt{T}\frac{\sqrt{t-s}}{c_A}.
\end{aligned}
\end{equation*}
That implies
\begin{equation}\label{eq:Phol}
0\le \cP_{s,x}(\alpha)-\cP_{t,x}(\alpha)\le \sqrt{T}\frac{\sqrt{t-s}}{c_A},
\end{equation}
which will be used later below. Notice that the above bound immediately implies that 
\begin{equation*}
\sup_{x\ge 0}\big|\varphi^K(t,x)-\varphi^K(s,x)\big|\le \sqrt{T}\frac{\sqrt{t-s}}{c_A}.
\end{equation*}

In order to prove upper semi-continuity let $\alpha^n\in\cA$ be $\eps$-optimal for the problem with value $\varphi^K(t_n,x_n)$. Notice that because $\alpha^n$ is $\F$-progressively measurable, then there is a measurable mapping $a_n:[0,\infty)\times C([0,\infty))\to \R$ such that 
$\alpha^n_s(\omega)=a_n\big(s,X^{x_n}_{\cdot\wedge s}(\omega)\big)$ (cf.\ \cite[Lem.\ 1.99 and 2.20]{fabbri2017stochastic}). Letting 
$\eta_n\coloneqq\inf\big\{s\ge 0: X^x_s=x_n\big\}$,
we define
\begin{equation*}
\nu^n_s(\omega)\coloneqq c_A\mathds{1}_{\{s\le \eta_n(\omega)\}}+ a_n\big(s-\eta_n,x_n+ \Delta X^{n}_{\cdot \wedge s}(\omega)\big)\mathds{1}_{\{s>\eta_n(\omega)\}}
\end{equation*}
with $\Delta X^n_u= c_A (u-\eta_n)+\sigma (B_{u}-B_{\eta_n})$ for $u\in[\![\eta_n,\infty)\!)$. By construction $\nu^n\in\cA$ and therefore $\varphi^K(t,x)\ge \cP_{t,x}(\nu^n)$. On the set $\{\sigma^{\nu^n}_t\wedge \rho_x >\eta_n\}$ we have 
\begin{equation*}
\begin{aligned}
\rho_x&=\eta_n+\inf\big\{u\ge 0: x_n+\Delta X^n_{\eta_n+u}=0\}\eqqcolon\eta_n+\vartheta_n,\\
\sigma^{\nu^n}_t&=\eta_n+\inf\Big\{u\ge 0: t+c^2_A\eta_{n}+\int_0^u\big|a_n\big(v,x_n+ \Delta X^{n}_{(\eta_n+\cdot) \wedge(\eta_n+v)}\big)\big|^2 \ud v=T \Big\}= \eta_n+\gamma_n(\eta_n),
\end{aligned}
\end{equation*}
where we define for $r\ge 0$
\begin{equation*}
\gamma_n(r)\coloneqq\inf\Big\{u\ge 0: t+c^2_A r+\int_0^u\big|a_n\big(v,x_n+ \Delta X^{n}_{(\eta_n+\cdot) \wedge(\eta_n+v)}\big)\big|^2 \ud v=T \Big\}.
\end{equation*}
Notice that the mapping $r\mapsto \gamma_n(r)$ is measurable and the tuple $((\nu^n_{\eta_n+v},\Delta X^n_{\eta_n+v})_{v\ge 0}, \vartheta_n,\gamma_n(r))$ is independent of $\cF_{\eta_n}$. Moreover, we have the following equivalence in law: 
\begin{equation*}
\mathsf{Law}\Big((\nu^n_{\eta_n+v},x_n+\Delta X^n_{\eta_n+v})_{v\ge 0},\vartheta_n,\gamma_n(r)\Big|\cF_{\eta_n}\Big)=\mathsf{Law}\Big((\alpha^n_v,X^{x_n}_v)_{v\ge 0},\rho_{x_n},\sigma^{\alpha^n}_{t+c^2_A r}\Big).
\end{equation*}

Equipped with these observations we obtain the following expression:
\begin{equation*}
\begin{aligned}
\cP_{t,x}(\nu^n)&=\E\Big[\mathds{1}_{\{\eta_n<\sigma^{\nu^n}_t\wedge \rho_x\}}\int_{\eta_n}^{\sigma^{\nu^n}_t\wedge \rho_x}\Big(\min\big\{\nu^n_s,\tfrac12/K(X^x_s)\big\}-c_A\Big)\ud s\Big]\\
&=\E\Big[\mathds{1}_{\{\eta_n<\sigma^{\nu^n}_t\wedge \rho_x\}}\int_{0}^{\gamma_n(\eta_n)\wedge \vartheta_n}\Big(\min\big\{\nu^n_{\eta_n+s},\tfrac12/K(x_n+\Delta X^n_{\eta_n+s})\big\}-c_A\Big)\ud s\Big].
\end{aligned}
\end{equation*}
Using tower property with $\cF_{\eta_n}$ and using independence (cf.\ \cite[Lem.\ 4.1]{baldi2017stochastic}) and the above equivalence in law we get
\begin{equation*}
\begin{aligned}
\cP_{t,x}(\nu^n)&=\E\Big[\mathds{1}_{\{\eta_n<\sigma^{\nu^n}_t\wedge \rho_x\}}\E\Big[\int_{0}^{\gamma_n(\eta_n)\wedge \vartheta_n}\Big(\min\big\{\nu^n_{\eta_n+s},\tfrac12/K(x_n+\Delta X^n_{\eta_n+s})\big\}-c_A\Big)\ud s\Big|\cF_{\eta_n}\Big]\Big]\\
&=\E\Big[\mathds{1}_{\{\eta_n<\sigma^{\nu^n}_t\wedge \rho_x\}}\E\Big[\int_{0}^{\sigma^{\alpha^n}_{t+c_A^2 r}\wedge \rho_{x_n}}\Big(\min\big\{\alpha^n_{s},\tfrac12/K(X^{x_n}_s)\big\}-c_A\Big)\ud s\Big]\Big|_{r=\eta_n}\Big]\\
&=\E\Big[\mathds{1}_{\{\eta_n<\sigma^{\nu^n}_t\wedge \rho_x\}}\cP_{t+c_A^2\eta_n,x_n}(\alpha^n)\Big].
\end{aligned}
\end{equation*}
Recalling the uniform bound in \eqref{eq:Phol} we deduce
\begin{equation*}
\begin{aligned}
\varphi^K(t,x)&\ge \E\Big[\mathds{1}_{\{\eta_n<\sigma^{\nu^n}_t\wedge \rho_x\}}\cP_{t_n,x_n}(\alpha^n)\Big]-\frac{\sqrt{T}}{c_A}\Big(\sqrt{|t-t_n|}+c_A\E\big[\mathds{1}_{\{\eta_n<\sigma^{\nu^n}_t\wedge \rho_x\}}(\eta_n)^{\frac12}\big]\Big)\\
&\ge\varphi^K(t_n,x_n) \P\big(\eta_n<\sigma^{\nu^n}_t\wedge \rho_x\big)-\eps -\frac{\sqrt{T}}{c_A}\Big(\sqrt{|t-t_n|}+c_A\E\big[(\eta_n\wedge T)^{\frac12}\big]\Big)
\end{aligned}
\end{equation*}
where for the second inequality we used that $\varphi^K(t_n,x_n)\le \cP_{t_n,x_n}(\alpha^n)+\eps$ by definition of $\alpha^n$. 
It follows that 
\begin{equation}\label{eq:limsup}
\begin{aligned}
\varphi^K(t_n,x_n)\!-\!\varphi^K(t,x)&\le \varphi^K(t_n,x_n) \P\big(\eta_n\ge \sigma^{\nu^n}_t\!\wedge\! \rho_x\big)\!+\!\eps\! +\!\frac{\sqrt{T}}{c_A}\Big(\sqrt{|t-t_n|}\!+\!c_A\E\big[(\eta_n\wedge T)^{\frac12}\big]\Big)\\
&\le \frac{T}{4 c_A} \P\big(\eta_n\ge \sigma^{\nu^n}_t\!\wedge\! \rho_x\big)\!+\!\eps \!+\!\frac{\sqrt{T}}{c_A}\Big(\sqrt{|t-t_n|}\!+\!c_A\E\big[(\eta_n\wedge T)^{\frac12}\big]\Big),
\end{aligned}
\end{equation}
where the final inequality uses $0\le \varphi^K(t_n,x_n)\le \phi(t_n)$. We decompose the probability as
\begin{equation*}
\begin{aligned}
\P\big(\eta_n\ge \sigma^{\nu^n}_t\wedge \rho_x\big)&=\P\big(\eta_n\ge \sigma^{\nu^n}_t,\sigma^{\nu^n}_t<\rho_x\big)+\P\big(\eta_n\ge \rho_x, \rho_x\le \sigma^{\nu^n}_t\big).
\end{aligned}
\end{equation*}
For the first term on the right-hand side above we have that 
\begin{equation*}
\{\eta_n\ge \sigma^{\nu^n}_t\}\subseteq\{\nu^n_s=c_A,\, s\in[0,\sigma^{\nu^n}_t]\}\subseteq\{\sigma^{\nu^n}_t=(T-t)/c^2_A\}.
\end{equation*}
That yields, for $t\in[0,T)$,
\begin{equation*}
\P\big(\eta_n\ge \sigma^{\nu^n}_t,\sigma^{\nu^n}_t<\rho_x\big)\le \P\big(\eta_n\ge (T-t)/c^2_A\big)\xrightarrow{n\to\infty} 0,
\end{equation*}
because $\eta_n\to 0$ in probability. For the other term we have that because $\P(\rho_x>0)=1$, then for any $\delta\in(0,1)$ there is $\lambda_\delta>0$ such that $\P(\rho_x>\lambda_\delta)\ge 1-\delta$. That implies, for $x>0$ 
\begin{equation*}
\begin{aligned}
\P\big(\eta_n\ge \rho_x, \rho_x\le \sigma^{\nu^n}_t\big)&\le \P\big(\eta_n\ge \rho_x\big)\\
&\le\P\big(\eta_n\ge \rho_x,\rho_x>\lambda_\delta\big)+\delta\le\P\big(\eta_n\ge \lambda_\delta\big)+\delta\xrightarrow{n\to\infty}\delta,
\end{aligned}
\end{equation*}
where the limit holds again by convergence in probability $\eta_n\to 0$. 

Plugging the bounds above back into \eqref{eq:limsup} and letting $n\to \infty$ it is not hard to verify that for $(t,x)\in[0,T)\times(0,\infty)$ we get
\begin{equation*}
\limsup_{n\to\infty}\varphi^K(t_n,x_n)\le \varphi^K(t,x)+\eps+\frac{T}{4c_A}\delta.
\end{equation*}
By arbitrariness of $\delta,\eps>0$ we deduce that $\varphi^K$ is upper semi-continuous on $[0,T)\times(0,\infty)$.

So far we have proven that $\varphi^K$ is continuous on $[0,T)\times(0,\infty)$ and it is at least $\frac12$-H\"older continuous in time, uniformly with respect to $x$, on $[0,T]\times[0,\infty)$. In order to prove that $\varphi^K\in C([0,T]\times[0,\infty))$ it only remains to show that $x\mapsto \varphi^K(t,x)$ is continuous at $x=0$ for any $t\in[0,T]$.
For any $(t,x)$, we have
\begin{equation*}
\begin{aligned}
0\le \varphi^{K}(t,x)&\le \sup_{\alpha\in\cA^K_x}\E\Big[\int_0^{\sigma_t^\alpha\wedge \rho_x}\! \alpha_s\ud s\Big]\\
&\le \sup_{\alpha\in\cA^K_x}\E\Big[\Big(\int_0^{\sigma_t^\alpha\wedge \rho_x}\! \alpha_s^2\ud s\Big)^{1/2}\big(\sigma_t^\alpha\wedge \rho_x\big)^{1/2}\Big]\\
&\le \sup_{\alpha\in\cA^K_x}\E\Big[\Big(\int_0^{\sigma_t^\alpha}\! \alpha_s^2\ud s\Big)^{1/2}\big( \rho_x\wedge\frac{T-t}{c_A^2}\big)^{1/2}\Big]=\E\Big[\big(T-t\big)^{1/2}\big( \rho_x\wedge\frac{T-t}{c_A^2}\big)^{1/2}\Big],
\end{aligned}
\end{equation*}
where we used Cauchy--Schwarz in the second inequality and the definition of $\sigma_t^\alpha$ (cf.\ \eqref{eq:sigmat_rhox}) in the equality. For the third inequality we used that $\alpha_s\ge c_A$ implies $\sigma^\alpha_t\le (T-t)/c^2_A$. Now, for any $t_0\in [0,T]$ and any sequence $(t_n,x_n)_{n\in\N}\subset[0,T]\times [0,\infty)$ converging to $(t_0,0)$ we have 
\begin{equation*}
0=\varphi^K(t_0,0)\le \varphi^K(t_n,x_n) \le \big(T-t_n\big)^{1/2}\E\Big[\Big( \rho_{x_n}\wedge\frac{T-t_n}{c_A^2}\Big)^{1/2}\Big]\to 0,\quad\text{as $n\to\infty$},
\end{equation*}
because $\rho_{x_n}\to 0$, $\P$-a.s. Therefore $\varphi^K(t_n,x_n)\to \varphi^K(t_0,0)$ as $n\to\infty$ as needed.
\end{proof}

The continuity of $\varphi^K$ allows us to use the DPP \eqref{eq:DPP} to improve the regularity to local Lipschitz continuity.

\begin{proposition}\label{prop:phi_bnd}
The following properties hold:
\begin{itemize}
\item[(i)] For each $x\in[0,\infty)$, the mapping $t\mapsto\varphi^K(t,x)$ is non-increasing with 
\begin{equation}\label{eq:lip_time}
-\frac{(t_2-t_1)}{4c_A}\le \varphi^K(t_2,x)-\varphi^K(t_1,x)\le 0, \quad\text{for $0\le t_1<t_2\le T$}.
\end{equation}
\item[(ii)] For each $t\in[0,T]$ and any fixed $\bar x>0$, 
\begin{equation*}
\sup_{x_1,x_2\in[\bar x,\infty)}\frac{|\varphi^{K}(t,x_2)-\varphi^{K}(t,x_1)|}{|x_2-x_1|}\le \frac{2/\sigma^2}{K(\bar x)}\big(T+2\sigma\sqrt T\big).
\end{equation*}
\end{itemize}
\end{proposition}

\begin{proof}
The proof is divided into two parts.

{\em Proof of (i)}. Monotonicity of the mapping $t\mapsto \varphi^K(t,x)$ is clear because $t\mapsto \sigma^\alpha_t$ is non-increasing and $\alpha_s\ge c_A$ for all $s\ge 0$, for any $\alpha\in\cA^K_x$. Next we prove the Lipschitz bound. Fix $t_1,t_2\in[0,T]$ with $t_2>t_1$ and $x\in[0,\infty)$. Let $\eps>0$ and take $\alpha\in\cA^K_x$ an $\eps$-optimal control for the function $\varphi^K(t_1,x)$. Since $\alpha$ is admissible for $\varphi^K(t_2,x)$, then
\begin{equation*}
\varphi^K(t_2,x)\ge \E\Big[\int_0^{\sigma^\alpha_{t_2}\wedge\rho_x}(\alpha_s-c_A)\ud s\Big].
\end{equation*}
Using $\eps$-optimality of $\alpha$ along with the DPP \eqref{eq:DPP} and $\sigma_{t_2}^{\alpha}\le \sigma_{t_1}^{\alpha}$, $\P$-a.s., we also get
\begin{equation*}
\varphi^K(t_1,x)\le \E\Big[\int_0^{\sigma^\alpha_{t_2}\wedge\rho_x}(\alpha_s-c_A)\ud s+\mathds{1}_{\{\sigma^\alpha_{t_2}<\rho_x\}}\varphi^K\big(Z_{\sigma^\alpha_{t_2}}^{t_1;\alpha},X^x_{\sigma^\alpha_{t_2}}\big)\Big]+\eps,
\end{equation*}
because on $\{\sigma^\alpha_{t_2}\ge \rho_x\}$ we have $\varphi^K(Z^{t_1;\alpha}_{\rho_x},X^x_{\rho_x})=\varphi^K(Z^{t_1;\alpha}_{\rho_x},0)=0$.
Subtracting the two expressions yields 
\begin{equation*}
\begin{aligned}
0&\ge \varphi^K(t_2,x)-\varphi^K(t_1,x)\\
&\ge-\E\Big[\mathds{1}_{\{\sigma^\alpha_{t_2}<\rho_x\}}\varphi^K\big(Z_{\sigma^\alpha_{t_2}}^{t_1;\alpha},X^x_{\sigma^\alpha_{t_2}}\big)\Big]-\eps=-\E\Big[\mathds{1}_{\{\sigma^\alpha_{t_2}<\rho_x\}}\varphi^K\big(T-(t_2-t_1),X^x_{\sigma^\alpha_{t_2}}\big)\Big]-\eps\\
&\ge -\frac{(t_2-t_1)}{4c_A}\P\big(\sigma^\alpha_{t_2}<\rho_x\big)-\eps\ge -\frac{(t_2-t_1)}{4c_A}-\eps,
\end{aligned}
\end{equation*}
where we used $Z_{\sigma^\alpha_{t_2}}^{t_1;\alpha}=T-(t_2-t_1)$ for the equality and the upper bound from Proposition \ref{cor:contphi} in the penultimate inequality. By the arbitrariness of $\eps$, we conclude the proof of $(i)$. 
\smallskip

{\em Proof of (ii)}. 
Take $0\le x_1<x_2<\infty$. Setting $\tau_1=\inf\{s\ge 0: X^{x_2}_s\le x_1\}$ and taking $\alpha_s=c_A$ in the DPP \eqref{eq:DPP} we obtain
\begin{equation*}
\varphi^K(t,x_2)\ge \E\Big[\mathds{1}_{\{\tau_1<\sigma^\alpha_t\}}\varphi^K(t+c^2_A\tau_1,x_1)\Big]=\E\Big[\mathds{1}_{\big\{\tau_1<\frac{T-t}{c^2_A}\big\}}\varphi^K(t+c^2_A\tau_1,x_1)\Big],
\end{equation*}
where the equality holds because $\sigma^\alpha_t=(T-t)/c^2_A$ for the constant control $\alpha_s=c_A$. Notice that we also used $\varphi^K (Z^{t;\alpha}_{\sigma^\alpha_t},X^{x_2}_{\sigma^\alpha_t})=\varphi^K(T,X^{x_2}_{\sigma^\alpha_t})=0$ 
on the event $\{\tau_1\ge \sigma^\alpha_t\}$. Then, subtracting $\varphi^K(t,x_1)$, we obtain a lower bound
\begin{equation*}
\begin{aligned}
\varphi^K(t,x_2)-\varphi^K(t,x_1)&\ge \E\Big[\mathds{1}_{\big\{\tau_1<\frac{T-t}{c^2_A}\big\}}\big(\varphi^K(t+c^2_A\tau_1,x_1)-\varphi^K(t,x_1)\big)\Big]-\varphi^K(t,x_1)\E\Big[\mathds{1}_{\big\{\tau_1\ge \frac{T-t}{c^2_A}\big\}}\Big]\\
&\ge -\tfrac{c_A}{4}\E\Big[\tau_1\mathds{1}_{\big\{\tau_1<\frac{T-t}{c^2_A}\big\}}\Big]-\varphi^K(t,x_1)\E\Big[\mathds{1}_{\big\{\tau_1\ge \frac{T-t}{c^2_A}\big\}}\Big]\\
&\ge -\tfrac{c_A}{4}\E\Big[\tau_1\mathds{1}_{\big\{\tau_1<\frac{T-t}{c^2_A}\big\}}\Big]-\tfrac{T-t}{4c_A}\E\Big[\mathds{1}_{\big\{\tau_1\ge \frac{T-t}{c^2_A}\big\}}\Big]\\
&= -\tfrac{c_A}{4}\E\big[\tau_1\wedge \big(\tfrac{T-t}{c^2_A}\big)\big]
\end{aligned}
\end{equation*}
where in the second inequality we use the Lipschitz bound in time for $\varphi^K$ from $(ii)$ and in the last inequality we use the upper bound on $\varphi^K$ from $(i)$. 

Next we obtain an upper bound for $\varphi^K(t,x_2)-\varphi^K(t,x_1)$.
Let $\alpha\in\cA^{K}_{x_2}$ be $\eps$-optimal for $\varphi^{K}(t,x_2)$ and recall $\tau_1$ as above. Using again the DPP \eqref{eq:DPP}, we have
\begin{equation*}
\begin{aligned}
\varphi^{K}(t,x_2)\!-\!\varphi^{K}(t,x_1)
&\le \E\Big[\int_0^{\sigma_t^{\alpha}\wedge\tau_1}\!\!\big(\alpha_s\!-\!c_A\big)\ud s\! +\!\varphi^{K}(Z_{\tau_1\wedge\sigma^\alpha_t}^{t;\alpha}, X^{x_2}_{\tau_1\wedge\sigma^\alpha_t})\Big]\!-\!\varphi^{K}(t,x_1)\!+\!\eps\\
&\le \E\Big[\int_0^{\sigma_t^{\alpha}\wedge\tau_1}\!\!\big(\alpha_s-c_A\big)\ud s\Big]\!+\!\eps,
\end{aligned}
\end{equation*}
where for the second inequality we used 
\begin{equation*}
\varphi^{K}(Z_{\tau_1\wedge\sigma^\alpha_t}^{t;\alpha}, X^{x_2}_{\tau_1\wedge\sigma^\alpha_t})=\mathds{1}_{\{\tau_1<\sigma^\alpha_t\}}\varphi^{K}(Z_{\tau_1\wedge\sigma^\alpha_t}^{t;\alpha}, x_1)\le \varphi^K(t,x_1),
\end{equation*}
because $\varphi^K$ is non-negative and decreasing in time. Moreover, $K(X^{x_2}_s)\ge K(x_1)$ for $s\in[\![0,\tau_1]\!]$ and therefore $c_A\le \alpha_s\le 1/(2K(x_1))$. That yields
\begin{equation*}
\begin{aligned}
\varphi^{K}(t,x_2)\!-\!\varphi^{K}(t,x_1)\le \frac{1}{2K(x_1)}\E\big[\tau_1\wedge\big(\tfrac{T-t}{c^2_A}\big)\big]\!+\!\eps,
\end{aligned}
\end{equation*}
where we also used $\sigma_t^\alpha\le (T-t)/c^2_A$ because $\alpha_s\ge c_A$. Further assuming that $x_1\ge \bar x$ we can make the bound dependent on $\bar x$ rather than on $x_1$ by replacing $K(x_1)$ with $K(\bar x)\le K(x_1)$. Moreover, the quantities under expectation are independent of $\eps$ and we can let $\eps\downarrow 0$.

In summary, we have obtained, for $\bar x\le x_1<x_2<\infty$
\begin{equation*}
-\tfrac{c_A}{4}\E\big[\tau_1\wedge \big(\tfrac{T-t}{c^2_A}\big)\big]\le \varphi^{K}(t,x_2)\!-\!\varphi^{K}(t,x_1)\le \frac{1}{2K(\bar x)}\E\big[\tau_1\wedge\big(\tfrac{T-t}{c^2_A}\big)\big].
\end{equation*}
Recalling that $0\le K(x)\le \frac{1}{4c_A}$ for all $x\ge 0$, we have $\frac{1}{2K(\bar{x})}\ge 2c_A\ge c_A/4$ and therefore 
\begin{equation}\label{eq:plug}
|\varphi^{K}(t,x_2)\!-\!\varphi^{K}(t,x_1)|\le \frac{1}{2K(\bar x)}\E\big[\big(\tfrac{T}{c^2_A}\big)\wedge\tau_1\big].
\end{equation}

It remains to find an upper bound for the right-hand side of the inequality above. Set $\Delta x=x_2-x_1>0$ for simplicity and notice that $\tau_1=\inf\{s\ge 0: (c_A/\sigma) s+ B_s\le -\Delta x/\sigma\}$. We are going to use 
\begin{equation*}
\begin{aligned}
\P(\tau_1>t)&=\P\Big(\inf_{0\le s\le t}\big((c_A/\sigma) s+B_s\big)>-\Delta x/\sigma\Big)\\
&=\P\Big(\sup_{0\le s\le t}\big(\widehat B_s-(c_A/\sigma) s\big)<\Delta x/\sigma\Big)=\Phi\Big(\frac{\Delta x+c_A t}{\sigma\sqrt t}\Big)-\e^{-2 c_A\Delta x/\sigma^2}\Phi\Big(-\frac{\Delta x-c_A t}{\sigma\sqrt t}\Big),
\end{aligned} 	
\end{equation*}
where in the second equality we set $\widehat B=-B$ and in the final formula $\Phi(z)$ is the cumulative of a standard normal (see \cite[Ch.\ II, eq.\ (4.4.21)]{peskir2006optimal}). The right-hand side of the above expression vanishes for $\Delta x\to 0$ and noticing that the derivative with respect to $\Delta x$ is bounded, we also obtain the upper bound
\begin{equation*}
\P(\tau_1>t)\le (2/\sigma^2)\Big(c_A+\frac{\sigma}{\sqrt t}\Big)\Delta x.
\end{equation*}
Then, we have
\begin{equation*}
\begin{aligned}
\E\big[\big(\tfrac{T}{c^2_A}\big)\wedge\tau_1\big]&=\int_0^{\frac{T}{c^2_A}}\P(\tau_1>t)\ud t\le (2/\sigma^2)\Delta x\Big(\frac{T}{c_A}+2\sigma\frac{\sqrt T}{c_A}\Big)\le (4/\sigma^2)\big(T+2\sigma\sqrt T\big)\Delta x,
\end{aligned}
\end{equation*}
where in the last inequality we used $c_A\ge 1/2$ (cf.\ its definition after \eqref{eq:SDEYW}). Plugging the final bound into the right-hand side of \eqref{eq:plug} completes the argument.
\end{proof}

We have shown that $\varphi^K$ is locally Lipschitz in $[0,T]\times(0,\infty)$. The next concavity result is crucial for much of the remainder of the paper.
\begin{proposition}\label{prop:phi^A_conc_A}
The mapping $t\mapsto\varphi^K(t,x)$ is concave for each $x\in[0,\infty)$.
\end{proposition}

\begin{proof}
The function $t\mapsto\varphi^K(t,x)$ is continuous by Proposition \ref{cor:contphi}. Then it is concave if it is mid-point concave (cf.\ \cite[Thm.\ 1.1.4]{niculescu2006convex}), i.e., if 
for fixed $x\in[0,\infty)$ and any $t_1,t_2\in[0,T]$, 
\begin{equation*}
\varphi^K\big(\tfrac12 t_1+\tfrac12 t_2,x\big)\ge \tfrac{1}{2} \varphi^K(t_1,x)+\tfrac{1}{2} \varphi^K(t_2,x).
\end{equation*}
The focus of the remainder of the proof is on showing the above inequality.

Let $t_1,t_2\in[0,T]$, $x\in[0,\infty)$ and define $\bar{t}=\frac12(t_1+t_2)$. Notice that the class of admissible controls $\cA^K_x$ does not depend on the variable $t$. 
Given any two controls $\alpha_1,\alpha_2\in\cA^K_x$ we let $\sigma_1\coloneqq\sigma_{t_1}^{\alpha_1}$ and $\sigma_2\coloneqq\sigma_{t_2}^{\alpha_2}$ for the ease of exposition.
For $s\in[0,\infty)$, we set
\begin{equation*}
\begin{aligned}
\hat{\alpha}_1(s)&=\alpha_1(s)\mathds{1}_{[0,\sigma_1]}(s)+2c_A\mathds{1}_{(\sigma_1,\infty)}(s),\quad\hat{\alpha}_2(s)=\alpha_2(s)\mathds{1}_{[0,\sigma_2]}(s)+2c_A\mathds{1}_{(\sigma_2,\infty)}(s),
\end{aligned}
\end{equation*}
and $\bar{\alpha}(s)=\frac{1}{2}(\hat{\alpha}_1(s)+\hat{\alpha}_2(s))$. Since the process $(X^x_t)_{t\ge 0}$ is uncontrolled, by definition of $\cA^K_x$ in \eqref{eq:AK} we easily see that $\bar \alpha\in\cA^K_x$.
By definition of $\bar{\alpha}$, we have 
\begin{equation*}
Z^{\bar{t};\bar{\alpha}}_s=\bar{t}+\int_0^s\big(\bar{\alpha}(u)\big)^2\ud u \le \frac{1}{2}(t_1+t_2)+\frac{1}{2}\int_0^s\Big[\big(\hat{\alpha}_1(u)\big)^2+\big(\hat{\alpha}_2(u)\big)^2\Big]\ud u=\frac{1}{2}\big(Z^{t_1;\hat{\alpha}_1}_s+Z^{t_2;\hat{\alpha}_2}_s\big).
\end{equation*}
Then, letting 
\begin{equation}\label{eq:barsigma}
\bar{\sigma}\coloneqq\inf\big\{s\ge 0:\tfrac{1}{2}\big(Z^{t_1;\hat{\alpha}_1}_s+Z^{t_2;\hat{\alpha}_2}_s\big)= T\big\},
\end{equation}
it holds $\sigma_{\bar{t}}^{\bar{\alpha}}\ge\bar{\sigma}$ and $\bar \sigma\in[\![\sigma_1\wedge\sigma_2, \sigma_1\vee\sigma_2]\!]$. Therefore, using the form of $\hat \alpha_1$ and $\hat \alpha_2$ we obtain 
\begin{equation}\label{eq:convex1}
\begin{aligned}
\varphi^K(\bar{t},x)&\!\ge\! \E\Big[\int_0^{\sigma_{\bar{t}}^{\bar{\alpha}}\wedge\rho_x}\!\!(\bar{\alpha}(s)\!-\!c_A)\ud s\Big]\\
&=\E\Big[\mathds{1}_{\{\sigma_1=\sigma_2\}}\int_0^{\sigma_{\bar{t}}^{\bar{\alpha}}\wedge\rho_x}\!\!(\bar{\alpha}(s)\!-\!c_A)\ud s\Big]\\
&\quad+\!\E\Big[\mathds{1}_{\{\sigma_1<\sigma_2\}}\tfrac{1}{2}\Big(\!\int_0^{\sigma_1\wedge\rho_x}\!\!(\alpha_1(s)\!-\!c_A)\ud s\!+\!\int_{\sigma_1\wedge\rho_x}^{\sigma_{\bar{t}}^{\bar{\alpha}}\wedge\rho_x}\!c_A\ud s\!+\!\int_0^{\sigma_{\bar{t}}^{\bar{\alpha}}\wedge\rho_x}\!\!(\hat \alpha_2(s)\!-\!c_A)\ud s\Big)\Big]\\
&\quad+\!\E\Big[\mathds{1}_{\{\sigma_1>\sigma_2\}}\tfrac{1}{2}\Big(\!\int_0^{\sigma_{\bar{t}}^{\bar{\alpha}}\wedge\rho_x}\!\!(\hat \alpha_1(s)\!-\!c_A)\ud s\!+\!\int_0^{\sigma_2\wedge\rho_x}\!\!(\alpha_2(s)\!-\!c_A)\ud s\!+\!\int_{\sigma_2\wedge\rho_x}^{\sigma_{\bar{t}}^{\bar{\alpha}}\wedge\rho_x}\!c_A\ud s\Big)\Big]\\
&\eqqcolon (A)+(B)+(C),
\end{aligned}
\end{equation}
where we used that $\hat \alpha_j(s)=\alpha_j(s)$ for $s\in[\![0,\sigma_j]\!]$ and $\hat \alpha_j(s)=2 c_A$ for $s\in(\!(\sigma_j,\infty)\!)$. Noticing that $\bar{\alpha}(s)\ge c_A$ for all $s\ge0$ and $\sigma_{\bar{t}}^{\bar{\alpha}}\ge \sigma_1\wedge\sigma_2$, for the first term above we have
\begin{equation}\label{eq:sig1=sig2}
\begin{aligned}
(A)&\ge \E\Big[\mathds{1}_{\{\sigma_1=\sigma_2\}}\int_0^{\sigma_1\wedge\sigma_2\wedge\rho_x}\!\!\!(\bar{\alpha}(s)\!-\!c_A)\ud s\Big]\\
&=\tfrac{1}{2}\E\Big[\mathds{1}_{\{\sigma_1=\sigma_2\}}\int_0^{\sigma_1\wedge\rho_x}\!\!\!(\alpha_1(s)\!-\!c_A)\ud s\!+\!\mathds{1}_{\{\sigma_1=\sigma_2\}}\int_0^{\sigma_2\wedge\rho_x}\!\!\!(\alpha_2(s)\!-\!c_A)\ud s\Big].
\end{aligned}
\end{equation}

In order to provide a lower bound on $(B)$, under expectations we separately consider the events $\{\sigma_1<\sigma_2\}\cap\{\sigma^{\bar \alpha}_{\bar t}\le \sigma_2\}$ and $\{\sigma_1<\sigma_2\}\cap\{\sigma^{\bar \alpha}_{\bar t}> \sigma_2\}$. On the event $\{\sigma_1<\sigma_2\}\cap\{\sigma^{\bar \alpha}_{\bar t}\le \sigma_2\}$, we have
\begin{equation*}
\begin{aligned}
&\int_0^{\sigma_1\wedge\rho_x}\!\!(\alpha_1(s)\!-\!c_A)\ud s\!+\!\!\int_{\sigma_1\wedge\rho_x}^{\sigma_{\bar{t}}^{\bar{\alpha}}\wedge\rho_x}\!\!c_A\,\ud s\!+\!\int_0^{\sigma_{\bar{t}}^{\bar{\alpha}}\wedge\rho_x}\!\!(\hat \alpha_2(s)\!-\!c_A)\ud s\\
&=\int_0^{\sigma_1\wedge\rho_x}\!\!(\alpha_1(s)\!-\!c_A)\ud s\!+\!\big(\sigma_{\bar{t}}^{\bar{\alpha}}\!\wedge\!\rho_x\!-\!\sigma_1\!\wedge\!\rho_x \big)c_A\!+\!\!\int_0^{\sigma_2\wedge\rho_x}\!\!(\alpha_2(s)\!-\!c_A)\ud s\!-\!\!\int_{\sigma_{\bar{t}}^{\bar{\alpha}}\wedge\rho_x}^{\sigma_2\wedge\rho_x}\!\!(\alpha_2(s)\!-\!c_A)\ud s
\end{aligned}
\end{equation*}
where again we used $\hat \alpha_2=\alpha_2$ on $[\![0,\sigma_2]\!]$. We claim here and we prove at the end that
\begin{equation}\label{eq:^_claim}
H\coloneqq\big(\sigma_{\bar{t}}^{\bar{\alpha}}\!\wedge\!\rho_x-\sigma_1\!\wedge\!\rho_x\big)c_A-\int_{\sigma_{\bar{t}}^{\bar{\alpha}}\wedge\rho_x}^{\sigma_2\wedge\rho_x}(\alpha_2(s)-c_A)\ud s\ge0,\quad\text{ on }\{\sigma_1<\sigma_2\}\cap\{\sigma^{\bar \alpha}_{\bar t}\le \sigma_2\}.
\end{equation}
Using the latter we obtain
\begin{equation*}
\begin{aligned}
&\int_0^{\sigma_1\wedge\rho_x}\!(\alpha_1(s)-c_A)\ud s+\int_{\sigma_1\wedge\rho_x}^{\sigma_{\bar{t}}^{\bar{\alpha}}\wedge\rho_x}\!c_A\,\ud s+\int_0^{\sigma_{\bar{t}}^{\bar{\alpha}}\wedge\rho_x}(\alpha_2(s)-c_A)\ud s\\
&\ge\int_0^{\sigma_1\wedge\rho_x}\!(\alpha_1(s)-c_A)\ud s+\int_0^{\sigma_2\wedge\rho_x}(\alpha_2(s)-c_A)\ud s, \qquad\text{ on }\{\sigma_1<\sigma_2\}\cap\{\sigma^{\bar \alpha}_{\bar t}\le \sigma_2\}.
\end{aligned}
\end{equation*}
On the event $\{\sigma_1<\sigma_2\}\cap\{\sigma^{\bar \alpha}_{\bar t}> \sigma_2\}$, we also have
\begin{equation*}
\begin{aligned}
&\int_0^{\sigma_1\wedge\rho_x}\!\!(\alpha_1(s)\!-\!c_A)\ud s\!+\!\int_{\sigma_1\wedge\rho_x}^{\sigma_{\bar{t}}^{\bar{\alpha}}\wedge\rho_x}\!\!c_A\,\ud s\!+\!\int_0^{\sigma_{\bar{t}}^{\bar{\alpha}}\wedge\rho_x}\!\!(\hat \alpha_2(s)\!-\!c_A)\ud s\\
&\ge\int_0^{\sigma_1\wedge\rho_x}\!(\alpha_1(s)-c_A)\ud s+\int_0^{\sigma_2\wedge\rho_x}(\alpha_2(s)-c_A)\ud s,
\end{aligned}
\end{equation*}
where we used that $\hat \alpha_2\ge c_A$ to find a lower bound for the third integral.

In conclusion we have obtained
\begin{equation}\label{eq:sig1<sig2}
\begin{aligned}
(B)\ge \tfrac{1}{2}\E\Big[\mathds{1}_{\{\sigma_1<\sigma_2\}}\int_0^{\sigma_1\wedge\rho_x}\!\!\!(\alpha_1(s)\!-\!c_A)\ud s\!+\!\mathds{1}_{\{\sigma_1<\sigma_2\}}\int_0^{\sigma_2\wedge\rho_x}\!\!\!(\alpha_2(s)\!-\!c_A)\ud s\Big].
\end{aligned}
\end{equation}
Repeating the same arguments we obtain a lower bound for $(C)$ in the form 
\begin{equation}\label{eq:sig2<sig1}
\begin{aligned}
(C)\ge \tfrac{1}{2}\E\Big[\mathds{1}_{\{\sigma_1>\sigma_2\}}\int_0^{\sigma_1\wedge\rho_x}\!\!\!(\alpha_1(s)\!-\!c_A)\ud s\!+\!\mathds{1}_{\{\sigma_1>\sigma_2\}}\int_0^{\sigma_2\wedge\rho_x}\!\!\!(\alpha_2(s)\!-\!c_A)\ud s\Big].
\end{aligned}
\end{equation}
Plugging \eqref{eq:sig1=sig2}, \eqref{eq:sig1<sig2} and \eqref{eq:sig2<sig1} into \eqref{eq:convex1} yields
\begin{equation*}
\varphi^K(\bar{t},x)\ge \tfrac{1}{2}\E\Big[\int_0^{\sigma_1\wedge\rho_x}\!(\alpha_1(s)-c_A)\ud s\Big]+\tfrac{1}{2}\E\Big[\int_0^{\sigma_2\wedge\rho_x}\!(\alpha_2(s)-c_A)\ud s\Big].
\end{equation*}
By arbitrariness of $\alpha_1,\alpha_2\in\cA^K_x$, we have $\varphi^K(\bar{t},x)\ge \frac{1}{2}\varphi^K(t_1,x)+\frac{1}{2}\varphi^K(t_2,x)$, as needed. 

In order to conclude the proof, it remains to show \eqref{eq:^_claim}. Recall that we are considering the event $\{\sigma_1<\sigma_2\}\cap\{\sigma^{\bar \alpha}_{\bar t}\le \sigma_2\}$. Then, 
$\rho_x\le \sigma_{\bar{t}}^{\bar{\alpha}}\implies\sigma_2\wedge\rho_x=\sigma_{\bar{t}}^{\bar{\alpha}}\wedge\rho_x=\rho_x$,
and we obtain
\begin{equation*}
\begin{aligned}
H
&=\mathds{1}_{\{\rho_x\le\sigma_{\bar{t}}^{\bar{\alpha}}\}}\Big(\big(\rho_x-\sigma_1\!\wedge\!\rho_x\big)c_A\Big)+\mathds{1}_{\{\rho_x>\sigma_{\bar{t}}^{\bar{\alpha}}\}}\Big(\big(\sigma_{\bar{t}}^{\bar{\alpha}}-\sigma_1\big)c_A-\int_{\sigma_{\bar{t}}^{\bar{\alpha}}}^{\sigma_2\wedge\rho_x}(\alpha_2(s)-c_A)\ud s\Big)\\
&\ge \mathds{1}_{\{\rho_x>\sigma_{\bar{t}}^{\bar{\alpha}}\}}\Big(\big(\sigma_{\bar{t}}^{\bar{\alpha}}-\sigma_1\big)c_A-\int_{\sigma_{\bar{t}}^{\bar{\alpha}}}^{\sigma_2\wedge\rho_x}(\alpha_2(s)-c_A)\ud s\Big).
\end{aligned}
\end{equation*}
We now recall that $\sigma_{\bar{t}}^{\bar{\alpha}}\ge\bar{\sigma}$ (cf.\ \eqref{eq:barsigma}) and, on the event $\{\sigma_1<\sigma_2\}$ that we are considering, $\bar \sigma\in[\![\sigma_1, \sigma_2]\!]$. Then, using also $\alpha_2(s)\ge c_A$ for all $s\ge0$ we arrive at
\begin{equation}\label{eq:^_claim1}
\begin{aligned}
H&\ge\mathds{1}_{\{\rho_x>\sigma_{\bar{t}}^{\bar{\alpha}}\}}\Big(\big(\bar{\sigma}-\sigma_1\big)c_A-\int_{\bar{\sigma}}^{\sigma_2}(\alpha_2(s)-c_A)\ud s\Big).
\end{aligned}
\end{equation}
Next, we rewrite $\bar \sigma $ in a more explicit fashion. Once again, recall that on $\{\sigma_1<\sigma_2\}$ we have $\bar\sigma\ge \sigma_1$. Moreover, $Z^{t_1;\hat \alpha_1}_s=Z^{t_1;\alpha_1}_s$ for $s\in[\![0,\sigma_1]\!]$ and 
\begin{equation*}
Z^{t_1;\hat \alpha_1}_s=Z^{t_1;\hat \alpha_1}_{\sigma_1}+4c_A^2 (s-\sigma_1)=T+4c_A^2 (s-\sigma_1),\quad \text{for $s\in(\!(\sigma_1,\infty)\!)$}.
\end{equation*}
Then $\bar \sigma$ rewrites as
\begin{equation*}
\begin{aligned}
\bar{\sigma}&=\inf\big\{s \ge \sigma_1:\tfrac{1}{2}\big(Z^{t_1;\hat{\alpha}_1}_s+Z^{t_2;\hat{\alpha}_2}_s\big)=T\big\}\\
&=\sigma_1+\inf\Big\{s \ge 0:\tfrac{1}{2}\big(T+4c^2_A s+Z^{t_2;\alpha_2}_{\sigma_1+s}\big)=T\Big\}=\sigma_1+\inf\Big\{s \ge 0:4c^2_A s+Z^{t_2;\alpha_2}_{\sigma_1+s}=T\Big\},
\end{aligned}
\end{equation*}
where in the second equality we also used that $Z^{t_2;\hat{\alpha}_2}_{s}=Z^{t_2;\alpha_2}_{s}$ for $s\in[\![0,\bar\sigma]\!]\subset[\![0,\sigma_2]\!]$ on $\{\sigma_1<\sigma_2\}$. The chain of equalities above shows that
\begin{equation*}
\bar\sigma-\sigma_1=\inf\Big\{s \ge 0:4c^2_A s+Z^{t_2;\alpha_2}_{\sigma_1+s}=T\Big\}\eqqcolon \zeta,\quad\text{on $\{\sigma_1<\sigma_2\}\cap\{\sigma^{\bar \alpha}_{\bar t}\le \sigma_2\}$}.
\end{equation*}
Since $4c^2_A \zeta+Z^{t_2;\alpha_2}_{\sigma_1+\zeta}=T$ by continuity of paths, then
\begin{equation}\label{eq:^_claim2}
\bar{\sigma}-\sigma_1=\frac{1}{4c^2_A}\Big(T-Z^{t_2;\alpha_2}_{\bar{\sigma}}\Big),\quad\text{on $\{\sigma_1<\sigma_2\}\cap\{\sigma^{\bar \alpha}_{\bar t}\le \sigma_2\}$}.
\end{equation}
Plugging \eqref{eq:^_claim2} into \eqref{eq:^_claim1} yields
\begin{equation}\label{eq:H1}
H\ge \mathds{1}_{\{\rho_x>\sigma_{\bar{t}}^{\bar{\alpha}}\}}\Big(\tfrac{1}{4c_A}\big(T-Z^{t_2;\alpha_2}_{\bar{\sigma}}\big)-\int_{\bar{\sigma}}^{\sigma_2}(\alpha_2(s)-c_A)\ud s\Big),\quad\text{on $\{\sigma_1<\sigma_2\}\cap\{\sigma^{\bar \alpha}_{\bar t}\le \sigma_2\}$}.
\end{equation}
Now we find an upper bound on the integral.
Fix $\omega\in\{\sigma_1<\sigma_2\}\cap\{\sigma^{\bar \alpha}_{\bar t}\le \sigma_2\}$. We have
\begin{equation}\label{eq:^_claim3}
\int_{\bar{\sigma}(\omega)}^{\sigma_2(\omega)}(\alpha_2(s,\omega)-c_A)\ud s=\int_0^{\sigma_2(\omega)-\bar{\sigma}(\omega)}(\alpha_2(\bar{\sigma}(\omega)+s,\omega)-c_A)\ud s.
\end{equation}
Let us notice
\begin{equation*}
\sigma_2(\omega)=\bar{\sigma}(\omega)+\inf\Big\{s\ge0: Z_{\bar{\sigma}(\omega)}^{t_2;\alpha_2}(\omega)+\int_0^s\big(\alpha_2(\bar{\sigma}(\omega)+u,\omega)\big)^2\ud u= T\Big\}\eqqcolon \bar{\sigma}(\omega)+\bar{\eta}(\omega),
\end{equation*}
so that $\bar \eta(\omega)$ is the first time the process
\begin{equation*}
s\mapsto Z_{\bar{\sigma}(\omega)}^{t_2;\alpha_2}(\omega)+\int_0^s\big(\alpha_2(\bar{\sigma}(\omega)+u,\omega)\big)^2\ud u
\end{equation*}
reaches $T$. Recall the formulation of the deterministic control problem in \eqref{eq:phi(t)} with $[0,\infty)\times[0,T]\times\cA_d\ni(s,t,a)\mapsto \cZ(s,t,a)=t+\int_0^s a^2(u)\ud u$ and the time $\eta(t,a)$. It is then clear that 
\begin{equation*}
Z_{\bar{\sigma}(\omega)}^{t_2;\alpha_2}(\omega)+\int_0^s\big(\alpha_2(\bar{\sigma}(\omega)+u,\omega)\big)^2\ud u=\cZ\big(s,Z_{\bar{\sigma}(\omega)}^{t_2;\alpha_2}(\omega),\alpha_2(\bar{\sigma}(\omega)+\cdot,\omega)\big)
\end{equation*}
and
\begin{equation*}
\bar \eta(\omega)=\eta\big(Z_{\bar{\sigma}(\omega)}^{t_2;\alpha_2}(\omega),\alpha_2(\bar{\sigma}(\omega)+\cdot,\omega)\big).
\end{equation*}
Since $\omega$ is fixed, it follows from \eqref{eq:phi(t)} and \eqref{eq:^_claim3} that
\begin{equation*}
\begin{aligned}
\int_{\bar{\sigma}(\omega)}^{\sigma_2(\omega)}\!\!\big(\alpha_2(s,\omega)\!-\!c_A\big)\ud s&=\int_0^{\bar{\eta}(\omega)}\big(\alpha_2(\bar{\sigma}(\omega)\!+\!s,\omega)\!-\!c_A\big)\ud s\\
&\le \sup_{a\in\cA_d}\int_0^{\eta(Z_{\bar{\sigma}(\omega)}^{t_2;\alpha_2}(\omega),a)}\!\!(a(s)\!-\!c_A)\ud s=\phi\big(Z_{\bar{\sigma}(\omega)}^{t_2;\alpha_2}(\omega)\big)=\frac{1}{4c_A}\Big(T\!-\!Z_{\bar{\sigma}(\omega)}^{t_2;\alpha_2}(\omega)\Big).
\end{aligned}
\end{equation*}
Since $\omega\in\{\sigma_1<\sigma_2\}\cap\{\sigma^{\bar \alpha}_{\bar t}\le \sigma_2\}$ was arbitrary, plugging the last equation above into \eqref{eq:H1} the claim in \eqref{eq:^_claim} is proved.
\end{proof}

Our next proposition shows that $\partial_t\varphi^K(t,x)$ is strictly smaller than zero for $x>0$. This fact will be used later to choose a specific function $K(x)$ in the definition of the class $\cA^K_x$. Since $\varphi^K(\cdot,x)$ is concave, it admits right- and left-derivatives at all points $t\in(0,T)$ and we denote them by $\partial^+_t\varphi^K(t,x)$ and $\partial^-_t\varphi^K(t,x)$, respectively. 

\begin{proposition}\label{prop:d_time}
For any $(t,x)\in(0,T)\times(0,\infty)$, we have
\begin{equation}\label{eq:dt}
-\frac{1}{4 c_A}\le \partial^+_t\varphi^K(t,x)\le\partial^-_t\varphi^K(t,x) \le -\frac{1}{4c_A}\P\Big(\rho_x>(T-t) c^{-2}_A\Big).
\end{equation}
The upper and lower bounds are independent of the choice of the function $K(\cdot)$. 
\end{proposition}
\begin{proof}
The first inequality holds by \eqref{eq:lip_time} and the second one by concavity. For the final inequality let $0\le t_1<t_2\le T$ and fix $x \in(0,\infty)$. Since $\cA^K_x$ is independent of time, admissible controls for $\varphi^K(t_1,x)$ and $\varphi^K(t_2,x)$ are the same. For $\alpha\in\cA^K_x$ we have $\sigma^\alpha_{t_2}\le \sigma^\alpha_{t_1}$. Then, by the DPP
\begin{equation*}
\varphi^K(t_1,x)=\sup_{\alpha\in\cA^K_x}\E\Big[\int_0^{\sigma_{t_2}^\alpha\wedge \rho_x}\!\!(\alpha_s\!-\!c_A)\ud s\!+\!\mathds{1}_{\{\sigma_{t_2}^\alpha<\rho_x\}}\varphi^K(Z^{t_1;\alpha}_{\sigma_{t_2}^\alpha},X_{\sigma_{t_2}^\alpha}^x)\Big].
\end{equation*}
Using $Z^{t_1;\alpha}_{\sigma^\alpha_{t_2}}=Z^{t_2;\alpha}_{\sigma^\alpha_{t_2}}-(t_2-t_1)=T-(t_2-t_1)$, we get
\begin{equation*}
\begin{aligned}
&\varphi^K(t_2,x)\!-\!\varphi^K(t_1,x)\\
&=\sup_{\alpha\in\cA^K_x}\E\Big[\int_0^{\sigma_{t_2}^\alpha\wedge \rho_x}\!\!(\alpha_s\!-\!c_A)\ud s\Big] \!-\!\sup_{\alpha\in\cA^K_x}\E\Big[\int_0^{\sigma_{t_2}^\alpha\wedge \rho_x}\!\!(\alpha_s\!-\!c_A)\ud s\!+\!\mathds{1}_{\{\sigma_{t_2}^\alpha<\rho_x\}}\varphi^K(Z^{t_1;\alpha}_{\sigma_{t_2}^\alpha},X_{\sigma_{t_2}^\alpha}^x)\Big]\\
&\le\sup_{\alpha\in\cA^K_x}\E\Big[\!-\!\mathds{1}_{\{\sigma_{t_2}^\alpha<\rho_x\}}\varphi^K(T\!-\!(t_2\!-\!t_1),X_{\sigma_{t_2}^\alpha}^x)\Big].
\end{aligned}
\end{equation*}
For $0<\eps<x$ 
we have $\rho_{x-\eps}\le \rho_x$ and therefore 
\begin{equation}\label{eq:lowerbnd0}
\begin{aligned}
\varphi^K(t_2,x)-\varphi^K(t_1,x)&\le -\inf_{\alpha\in\cA^K_x}\E\Big[\mathds{1}_{\{\sigma_{t_2}^\alpha<\rho_{x-\eps}\}}\varphi^K\big(T-(t_2-t_1),X_{\sigma_{t_2}^\alpha}^{x}\big)\Big],
\end{aligned}
\end{equation}
because $\varphi^K\ge 0$. 
Notice that for arbitrary $(t,y)$, picking the deterministic control $\cA^K_y\ni\alpha_t\equiv 2c_A$ yields $\sigma_t^{\alpha}=(T-t)/(4c_A^2)$ and the lower bound:
\begin{equation*}
\begin{aligned}
\varphi^K(t,y)\ge 
c_A \E\Big[\int_{0}^{\frac{T-t}{4c_A^2}}\mathds{1}_{\{s\le \rho_y\}}\,\ud s\Big]= c_A \int_{0}^{\frac{T-t}{4c_A^2}}\P(s\le \rho_y)\,\ud s\eqqcolon c_A \int_{0}^{\frac{T-t}{4c_A^2}}g(s,y)\,\ud s.
\end{aligned}
\end{equation*}
In particular, on the event $\{\sigma_{t_2}^\alpha<\rho_{x-\eps}\}$ we have $X_{\sigma_{t_2}^\alpha}^{x-\eps}>0\iff X_{\sigma_{t_2}^\alpha}^{x}>\eps$ and
\begin{equation}\label{eq:lowerbnd}
\begin{aligned}
\varphi^K\big(T-(t_2-t_1),X_{\sigma_{t_2}^\alpha}^{x}\big)\ge c_A \int_{0}^{\frac{t_2-t_1}{4c_A^2}}g\big(s,X_{\sigma_{t_2}^\alpha}^{x}\big)\,\ud s\ge c_A \int_{0}^{\frac{t_2-t_1}{4c_A^2}}g(s,\eps)\,\ud s,
\end{aligned}
\end{equation}
where the final inequality holds because $y\mapsto g(s,y)$ is non-decreasing.

Therefore, plugging \eqref{eq:lowerbnd} into \eqref{eq:lowerbnd0} and using the form of $g(s,\eps)$ 
yields
\begin{equation*}
\begin{aligned}
\varphi^K(t_2,x)-\varphi^K(t_1,x)
&\le -\Big(c_A\int_0^{\frac{t_2-t_1}{4c^2_A}}\P(s\le \rho_{\eps})\ud s\Big)\inf_{\alpha\in\cA^K_x}\P\Big(\sigma_{t_2}^\alpha<\rho_{x-\eps}\Big)\\
&\le-\Big(c_A\int_0^{\frac{t_2-t_1}{4c^2_A}}\P(s\le \rho_{\eps})\ud s\Big)\P\Big((T-t_2) c^{-2}_A<\rho_{x-\eps}\Big),
\end{aligned}
\end{equation*}
where in the second inequality we used 
$ \sigma_{t_2}^\alpha\le \frac{T-t_2}{c^2_A}$ for all $\alpha\in\cA^K_x$, because $\alpha_t\ge c_A$.
Dividing by $t_2-t_1$ and taking limits as $t_1\uparrow t_2$, we get
\begin{equation*}
\partial^-_t\varphi^K(t_2,x)\le -\frac{1}{4 c_A}\P\Big((T-t_2) c^{-2}_A<\rho_{x-\eps}\Big),
\end{equation*}
because $s\mapsto\P(\rho_{\eps}\ge s)$ is continuous with $\P(\rho_{\eps}> 0)=\P(\rho_{\eps}\ge 0)=1$. Letting $\eps\downarrow0$, we have $\rho_{x-\eps}\uparrow\rho_x$. Then, by dominated convergence and left-continuity of $s\mapsto\mathds{1}_{\{s>c^{-2}_A(T-t_2)\}}$, we get the final inequality in \eqref{eq:dt}.
\end{proof}

\begin{remark}\label{rem:opt_cont}
Let $K_*(x)\coloneqq \frac{1}{4c_A}\P(\rho_x>T c^{-2}_A)$ with $\rho_x$ as in Proposition \ref{prop:d_time}.
By concavity of $\varphi^K(\cdot,x)$, for any fixed $x\in(0,\infty)$, we have $\partial_t^-\varphi^K(t,x)=\partial_t^+\varphi^K(t,x)$ for almost every $t\in[0,T]$, and the specific form of the Hamiltonian in \eqref{eq:hamil} tells us that 
$\alpha^*_s=-1/[2 \partial_t^{-}\varphi^K(Z^{t;\alpha^*}_s,X^{x}_s)]$
is the candidate optimal control.
Therefore, by Proposition \ref{prop:d_time}
\begin{equation*}
2c_A \le \alpha^*_s \le \frac12[K_*(X_s^x)]^{-1}.
\end{equation*}

The bound in Proposition \ref{prop:d_time} holds for any function $K(x)$ defining $\cA^K_x$. Then, it is natural to choose $K(x)=K_*(x)$ and denote the associated $\varphi^{K_*}$ simply by $\varphi^*$. Similarly, we write $\cA^*_x=\cA^{K_*}_x$.
\end{remark}

\subsection{Bounds on second order spatial derivative}\label{sec:phixx}

Having established estimates for the time derivative of $\varphi^*=\varphi^{K_*}$ and for its first-order spatial derivative, we now proceed to obtain bounds on the second-order spatial derivative. Since we work with $K(x)=K_*(x)$, the Lipschitz constant in Proposition \ref{prop:phi_bnd} is precisely $2(T+2\sigma\sqrt{T})/(\sigma^2 K_*(\bar x))$. In the proof of the next proposition we set
$\delta_T\coloneqq 2(T+2\sigma\sqrt{T})/\sigma^2$.
\begin{proposition}\label{prop:twice_diff}
For fixed $t\in[0,T]$, the second-order spatial derivative $\partial_{xx} \varphi^*(t,x)$ exists for almost every $x\in(0,\infty)$. Moreover, for any bounded interval $I \subset(0,\infty)$
\begin{equation}\label{eq:D2}
\|\partial_{xx}\varphi^*(t,\,\cdot\,)\|_{L^{\infty}(I)}\le C_0
\end{equation}
for a constant $C_0=C_0(I)>0$ independent of $t\in[0,T]$.
\end{proposition}

\begin{proof}
Let $x\in (0,\infty)$ and take $0<a<x<b<\infty$. The function $\varphi^*$ restricted to the domain $[0,T]\times I$ with $I=(a,b)$ is Lipschitz by Proposition \ref{prop:phi_bnd}--(i+ii), and $\varphi^*(T,x)=0$ for all $x\in I$. Then, we extend $\varphi^*$ outside $[0,T]$ defining
\begin{equation*}
\tilde{\varphi}^*(t,x)=\varphi^*(t,x)\mathds{1}_{[0,T]\times I}(t,x),\quad (t,x)\in [0,\infty)\times I.
\end{equation*}
The function $\tilde{\varphi}^*$ is Lipschitz continuous on $[0,\infty)\times I$ with the same Lipschitz constant as $\varphi^*$ on $[0,T]\times I$. By DPP \eqref{eq:DPP} applied to $\varphi^*$ with $\tau=\nu_I\coloneqq\inf\{s\ge 0: X_s^x\notin I\}$, we have
\begin{equation}\label{eq:tilde_varphi}
\begin{aligned}
\varphi^*(t,x)
&=\sup_{\alpha\in\cA^*_x}\E\Big[\int_0^{\sigma_t^\alpha\wedge\nu_I}(\alpha_s-c_A)\,\ud s+\mathds{1}_{\{\nu_I<\sigma_t^\alpha\}}\varphi^*(Z^{t;\alpha}_{\nu_I},X_{\nu_I}^x)\Big]\\
&=\sup_{\alpha\in\cA^*_x}\E\Big[\int_0^{\sigma_t^\alpha\wedge\nu_I}(\alpha_s-c_A)\,\ud s+\tilde{\varphi}^*(Z^{t;\alpha}_{\nu_I},X^x_{\nu_I})\Big],
\end{aligned}
\end{equation}
where in the first equality we used that $\varphi^*(Z^{t;\alpha}_{\sigma_t^\alpha\wedge\nu_I},X^x_{\sigma_t^\alpha\wedge\nu_I})\mathds{1}_{\{\sigma_t^\alpha<\nu_I\}}=\varphi^*(T,X^x_{\sigma_t^\alpha})\mathds{1}_{\{\sigma_t^\alpha<\nu_I\}}=0$ and in the second equality we used that 
\begin{equation*}
\mathds{1}_{\{\nu_I<\sigma_t^\alpha\}}\varphi^*(Z^{t;\alpha}_{\nu_I},X_{\nu_I}^x)=\mathds{1}_{\{\nu_I<\sigma_t^\alpha\}}\tilde\varphi^*(Z^{t;\alpha}_{\nu_I},X_{\nu_I}^x)=\tilde\varphi^*(Z^{t;\alpha}_{\nu_I},X_{\nu_I}^x),
\end{equation*}
because $\tilde\varphi^*(Z^{t;\alpha}_{\nu_I},X_{\nu_I}^x)\mathds{1}_{\{\nu_I\ge \sigma_t^\alpha\}}=0$ by definition of $\tilde\varphi^*$.

Fix $t\in[0,T]$. Taking the sub-optimal strategy $\alpha\equiv c_A$ and using \eqref{eq:lip_time}, we get from \eqref{eq:tilde_varphi}
\begin{equation*}
\begin{aligned}
\varphi^*(t,x)\ge\E\Big[\tilde{\varphi}^*(t+c^2_A\nu_I,X^x_{\nu_I})\Big]\ge \E\Big[\tilde{\varphi}^*(t,X^x_{\nu_I})-\frac14c_A\nu_I\Big].
\end{aligned}
\end{equation*}
It is well-known that $\E[B_{\nu_I}]=0$, because $\nu_I$ is the exit time of $B$ from a time-dependent interval (for the reader's convenience we provide a short proof of this fact in Appendix \ref{app:stopping}, Lemma \ref{lem:OS}). Noticing that $c_A\nu_I=(x+c_A\nu_I+\sigma B_{\nu_I})-x-\sigma B_{\nu_I}=X^x_{\nu_I}-x-\sigma B_{\nu_I}$, we obtain 
\begin{equation*}
\varphi^*(t,x)\ge \E\Big[\tilde{\varphi}^*(t,X^x_{\nu_I})-\tfrac14X^x_{\nu_I}\Big]+\tfrac14 x.
\end{equation*}
Denoting $\tilde{f}(x;t)\coloneqq \tilde{\varphi}^*(t,x)-\frac14 x$ and using the inequality above with $\tilde{\varphi}^*=\varphi^*$ on $[0,T]\times I$, we have
\begin{equation*}
\tilde{f}(x;t)\ge \E\big[\tilde{f}(X_{\nu_I}^x;t)\big].
\end{equation*}
The right-hand side of the inequality can be rewritten in terms of the scale function $S(x)$ of the process $X$. That is,
\begin{equation*}
\E\big[\tilde{f}(X_{\nu_I}^x;t)\big]=\tilde{f}(a;t)\frac{S(b)-S(x)}{S(b)-S(a)}+\tilde{f}(b;t)\frac{S(x)-S(a)}{S(b)-S(a)}.
\end{equation*}
Then, $\tilde{f}(\,\cdot\,;t)$ is $S$-concave on $[a,b]$, i.e., for all $x\in[a,b]$ 
\begin{equation*}
\tilde{f}(x;t)\ge \tilde{f}(a;t)\frac{S(b)-S(x)}{S(b)-S(a)}+\tilde{f}(b;t)\frac{S(x)-S(a)}{S(b)-S(a)}.
\end{equation*}
As a consequence it admits left- and right-derivative with respect to the scale function at all points. More precisely (cf.\ \cite[Prop.\ 2.6]{dayanik2003optimal}), the right- and left-limits 
\begin{equation*}
D_S^+\tilde{f}(x;t)\coloneqq \lim_{y\to x+}\frac{\tilde{f}(y;t)-\tilde{f}(x;t)}{S(y)-S(x)}\quad\text{and}\quad D_S^-\tilde{f}(x;t)\coloneqq \lim_{y\to x-}\frac{\tilde{f}(x;t)-\tilde{f}(y;t)}{S(x)-S(y)}
\end{equation*}
exist for all $x\in I$; moreover, the mappings $x\mapsto D_S^+\tilde{f}(x;t)$ and $x\mapsto D_S^-\tilde{f}(x;t)$ are non-increasing with the first one right-continuous and the second one left-continuous; finally, since $S$ is continuous, then $D_S^+\tilde{f}(x;t)=D_S^-\tilde{f}(x;t)\eqqcolon D_S\tilde{f}(x;t)$ in $I\setminus C$ with $C\subset I$ at most countable. As a result of these facts and continuous differentiability of $S$ at all points, we have
\begin{equation*}
\begin{aligned}
D^\pm_S\tilde{f}(x;t)&=\lim_{y\to x\pm}\frac{\tilde{f}(x;t)-\tilde{f}(y;t)}{x-y}\frac{x-y}{S(x)-S(y)}=\frac{1}{S'(x)}\lim_{y\to x\pm}\frac{\tilde{f}(x;t)-\tilde{f}(y;t)}{x-y},\quad x\in I.
\end{aligned}
\end{equation*}
Therefore, 
$\partial^{\pm}_x \tilde{f}(x;t)=D^{\pm}_S\tilde{f}(x;t) S'(x)$ for $x\in I$ with $\partial^{\pm}_x \tilde{f}(x;t)=D_S\tilde{f}(x;t) S'(x)\eqqcolon \partial_x \tilde{f}(x;t)$ for $x\in I\setminus C$. Notice that $x\mapsto \partial^{-}_x \tilde{f}(x;t)$ and $x\mapsto \partial^{+}_x \tilde{f}(x;t)$ are left- and right-continuous, respectively.

Since $x\mapsto D^\pm_S\tilde f(x;t)$ is non-increasing on $I$, it is differentiable Lebesgue almost everywhere on $I$ with $\partial_x(D^\pm_S\tilde f)(x;t)\le 0$. As a result $\partial_{x}(\partial_x^\pm\tilde f(x;t))$ exists a.e.\ on $I$ with 
\begin{equation}\label{eq:2nd_deriv}
\begin{aligned}
\partial_{x}\big(\partial^\pm_x \tilde{f}(x;t)\big)&=\partial_x (D^\pm_S\tilde{f})(x;t) S'(x)+D^\pm_S\tilde{f}(x;t)S''(x)\le D^\pm_S\tilde{f}(x;t)S''(x).
\end{aligned}
\end{equation}
We wish to bound the last expression from above by a constant. We recall that, up to an affine transformation, the scale function reads as $S(x)=1-\e^{-2c_Ax/\sigma^2}$ and therefore $S''(x)=-4 c_A^2/\sigma^4\e^{-2c_A x/\sigma^2}$. 
By definition of $\tilde f$ we get $\partial^\pm_x \tilde f(x;t)=\partial^\pm_x\varphi^*(t,x)-\frac14$ and $\varphi^*(t,\cdot)$ is Lipschitz on $I$ with Lipschitz constant bounded by $\delta_T/K_*(a)$ (cf.\ Proposition \ref{prop:phi_bnd}--(ii)). Then, for $x\in I$
\begin{equation*}
\begin{aligned}
|D^\pm_S\tilde{f}(x;t)|&=\frac{1}{S'(x)}\big|\partial^\pm_x \tilde{f}(x;t)\big|\le \frac{\e^{2 c_A x/\sigma^2}}{2 c_A/\sigma^2}\big(\tfrac14+\big|\partial^\pm_x\varphi^*(t,x)\big|\big) \le \frac{\e^{2 c_A b/\sigma^2}}{2 c_A/\sigma^2}\Big(\frac14+\frac{\delta_T}{K_*(a)}\Big).
\end{aligned}
\end{equation*}
Substituting this bound into \eqref{eq:2nd_deriv} and using $|S''(x)|\le 4c^2_A/\sigma^4\e^{-2 c_A a/\sigma^2}$ yields, for a.e.\ $x\in I$
\begin{equation*}
\partial_{x}\big(\partial_x^\pm\tilde f(x;t)\big)\le \frac{2 c_A}{\sigma^2}\e^{2 c_A (b-a)/\sigma^2}\Big(\frac14+\frac{\delta_T}{ K_*(a)}\Big)\eqqcolon C_1.
\end{equation*}

For fixed $t\in[0,T]$, we have shown that $x\mapsto \tilde{f}(x;t)=\tilde \varphi^*(t,x)-\frac14 x$ is locally semi-concave. More precisely,
recalling $\tilde{\varphi}^*=\varphi^*$ on $[0,T]\times I$, the mapping 
\begin{equation}\label{eq:semiconc_A}
[a,b]\ni x\mapsto \varphi^*(t,x)-\frac14x- C_1 x^2
\end{equation}
is concave.

Next, we are going to adopt similar arguments to show that $x\mapsto \varphi^*(t,x)$ is also semi-convex for fixed $t\in[0,T]$. Take $x\in I$. We recall that $\alpha\in \cA^*_x$ and so $c_A\le \alpha_s\le \frac{1}{2K_*(X_s^x)}$ for all $s\in[0,\infty)$. In particular 
$c_A\le \alpha_s\le \frac12[K_*(a)]^{-1}$, for $s\in [\![0,\nu_I]\!]$.
Then, from \eqref{eq:tilde_varphi} we have
\begin{equation*}
\begin{aligned}
\varphi^*(t,x)\le&\sup_{\alpha\in\cA^*_x}\E\Big[\frac{1}{2 K_*(a)}(\sigma_t^\alpha\wedge\nu_I)+\tilde{\varphi}^*(Z^{t;\alpha}_{\nu_I},X^x_{\nu_I})\Big]\le \E\Big[\frac{1}{2 K_*(a)}\nu_I+\tilde{\varphi}^*(t,X^x_{\nu_I})\Big],
\end{aligned}
\end{equation*}
where in the second inequality we used that $t\mapsto\tilde{\varphi}^*(t,x)$ is non-increasing. Rewriting 
\begin{equation*}
\frac{1}{2 K_*(a)}\nu_I=\frac{1}{ 2c_A K_*(a)}X^x_{\nu_I}-\frac{1}{2c_A K_*(a)}(x+\sigma B_{\nu_I})
\end{equation*} 
and using again $\E[B_{\nu_I}]=0$ and the equivalence $\tilde \varphi^*=\varphi^*$ on $[0,T]\times I$, we obtain
\begin{equation*}
\tilde{\varphi}^*(t,x)\le\E\Big[\frac{1}{2 c_A K_*(a)}X^x_{\nu_I}+\tilde{\varphi}^*(t,X^x_{\nu_I})\Big]-\frac{1}{2 c_A K_*(a)}x.
\end{equation*}
Defining the function $\hat{f}(x;t)\coloneqq \tilde{\varphi}^*(t,x)+\frac{1}{2 c_A K_*(a)}x$, the equation above yields
\begin{equation*}
\hat{f}(x;t)\le\E\big[\hat{f}(X^x_{\nu_I};t)\big]=\hat{f}(a;t)\frac{S(b)-S(x)}{S(b)-S(a)}+\hat{f}(b;t)\frac{S(x)-S(a)}{S(b)-S(a)}.
\end{equation*}
Then, for fixed $t\in[0,T]$, the mapping $x\mapsto \hat f(x;t)$ is $S$-convex on $I$. Using again \cite[Prop.\ 2.6]{dayanik2003optimal} we obtain that $D_S^+\hat{f}(x;t)$ and $D_S^-\hat{f}(x;t)$ exist for all $x\in I$, both are non-decreasing in $x$, the first one is right-continuous and the second one is left-continuous. Moreover, by continuity of $S$ we have $D_S^+\hat{f}(x;t)=D_S^-\hat{f}(x;t)\eqqcolon D_S\hat f(x;t)$ for a.e.\ $x\in I$.
As in \eqref{eq:2nd_deriv}, we have
\begin{equation}\label{eq:hatfxx}
\partial_{x}\big(\partial^\pm_x \hat{f}(x;t)\big)=\partial_x (D^\pm_S\hat{f})(x;t)S'(x)+D^\pm_S\hat{f}(x;t)S''(x)\ge D^\pm_S\hat{f}(x;t)S''(x).
\end{equation}
Recalling that $\hat{f}(x;t)\coloneqq \tilde{\varphi}^*(t,x)+\frac{1}{2 c_A K_*(a)}x$ and $|\partial^\pm_x \tilde{\varphi}^*(t,x)|\le \delta_T/K_*(a)$ (Proposition \ref{prop:phi_bnd}--(ii)), we obtain for $x\in I$,
\begin{equation*}
\begin{aligned}
|D^\pm_S\hat{f}(x;t)|=\frac{1}{S'(x)}\big|\partial^\pm_x\hat{f}(x;t)\big|&\le \frac{\e^{2 c_A x/\sigma^2}}{2 c_A/\sigma^2}\Big(\frac{1}{2 c_A K_*(a)}+\big|\partial^\pm_x\tilde{\varphi}^*(t,x)\big|\Big)\le \frac{\e^{2 c_A b/\sigma^2}}{2 c_A K_*(a)/\sigma^2}\big(1+\delta_T\big),
\end{aligned}
\end{equation*}
where we also used $c_A\ge 1/2$ for the final inequality.
Combining the bound above with $|S''(x)|\le 4c_A^2/\sigma^4\e^{-2c_A a/\sigma^2}$ and \eqref{eq:hatfxx} we deduce
\begin{equation*}
\partial_x\big(\partial^\pm_x \hat{f}(x;t)\big)\ge -\frac{2c_A/\sigma^2\e^{ 2c_A (b-a)/\sigma^2}}{K_*(a)}\big(1+\delta_T\big)\eqqcolon -C_2.
\end{equation*}
Then, the mapping
\begin{equation}\label{eq:semiconv}
[a,b]\ni x\mapsto \varphi^*(t,x)+\frac{1}{2c_A K_*(a)}x+C_2 x^2
\end{equation}
is convex. 

Combining \eqref{eq:semiconc_A} with \eqref{eq:semiconv}, for any fixed $t\in[0,T]$ we know that $\partial_{xx}\varphi^*(t,\cdot)$ exists in the sense of distributions and it defines a signed measure on $I$. Moreover, for any test function $\zeta\in C^\infty_c(I)$, $\zeta\ge 0$, we get
\begin{equation*}
\Big|\int_I \zeta(x)\partial_{xx}\varphi^*(t,\ud x)\Big|\le \big(C_1+C_2\big)\int_I\zeta(x)\ud x.
\end{equation*} 
Then, there is a set $N_t\subset I$ of zero Lebesgue measure such that $\partial_{xx}\varphi^*(t,x)$ exists for $x\in I\setminus N_t$ and, moreover,
\begin{equation*}
\big\|\partial_{xx}\varphi^*(t,\,\cdot\,)\big\|_{L^\infty(I)}\le C_1+C_2.
\end{equation*}
The constants $C_1$ and $C_2$ depend on $I$ but are independent of $t$. That concludes the proof.
\end{proof}

From the proposition above, for each fixed $t\in[0,T]$ the weak derivative $\partial_{xx}\varphi^*(t,\,\cdot\,)$ is well-defined and locally bounded on $(0,\infty)$. However, a priori the null set where $\partial_{xx}\varphi^*(t,\,\cdot\,)$ is not defined may depend on $t$. In the next proposition we show that indeed $\partial_{xx}\varphi^*$ is well-defined as a locally bounded measurable function on $[0,T]\times[0,\infty)$, i.e., $\partial_{xx}\varphi^*\in L^\infty_{\ell oc}([0,T]\times[0,\infty))$. We recall the definition of the Sobolev space $W^{1,2;p}(\cO)$ as the space of functions $f\in L^p(\cO)$ whose weak derivatives $\partial_t f$, $\partial_x f$ and $\partial_{xx}f$ exist as functions in $L^p(\cO)$. We use $W^{1,2;p}_{\ell oc}(\cO)$ for the class of functions $f\in W^{1,2;p}(K)$ for any bounded subset $K$ of $\cO$.

The proof of our next proposition closely resembles the one of \cite[Prop.\ 9.3]{brezis2010functional} and so we postpone it to the Appendix.
\begin{proposition}\label{prop:W12p}
We have $\varphi^*\in W^{1,2;\infty}_{\ell oc}([0,T]\times[0,\infty))\cap C([0,T]\times[0,\infty))$.
\end{proposition}

We conclude the section providing bounds on $\varphi^*$ akin to those arising from the dynamic programming principle but for a time-changed version of \eqref{eq:DPP}. The aim is to link $\varphi^*$ to the original problem \eqref{eq:orig_value} but in a different kind of weak formulation. Later on, we will show that indeed the optimally controlled dynamics exists in a strong sense and therefore all formulations are equivalent. We recall $\bar \cB$ from Definition \ref{def:B} (recall also Remark \ref{rem:barphi}). Now we introduce an analogue class of controls on our fixed probability space $(\Omega,\cF,\P)$, but in a weaker formulation. That is, a control $\beta$ is a collection $(\F',(\beta_t)_{t\ge 0}, (B'_t)_{t\ge 0},(W_t^{\beta})_{t\ge 0})$ such that $\F'\coloneqq(\cF'_t)_{t\ge 0}$ is a filtration on $\cF$, $(B'_t)_{t\ge 0}$ is a $\F'$-Brownian motion, $(\beta_t)_{t\ge 0}$ is a $\F'$-progressively measurable, process such that
$\E[\int_0^T\beta^2_s\ud s]<\infty$,
and the process $W^{\beta}$ follows the dynamics
\begin{equation}\label{eq:Wbetaweak}
W_t^{\beta}=W^\beta_0+\int^t_0 c_A \beta_v^2\,\ud v+\int^t_0\beta_v \,\sigma\ud B'_v,
\end{equation}
with $W^\beta_0$ an $\cF'_0$-measurable initial value.
When $W^\beta_0=x\in[0,\infty)$ we use the notation $W^{x;\beta}_t$ for the associated process.
We denote the class of such controls by $\cB$ and, for future reference, we notice that such class is analogous to the one used in \cite[Ch.\ 2, Sec.\ 4]{krylov2009controlled} (in the proof of Theorem \ref{thm:PDE_q-orig} we will use \cite[Cor.\ 2.4.8]{krylov2009controlled}).
Let us also consider a class $\cB^*_x\subset \cB$ of admissible controls defined as
\begin{equation*}
\cB^*_x\coloneqq \Big\{\beta\in\cB\,\Big|\,2K_*(W_s^{x;\beta})\le \beta_s\le \frac{1}{c_A},\,\text{for all}\, s\ge 0\Big\},
\end{equation*}
with the function $K_*(x)$ as in Remark \ref{rem:opt_cont}. 

In what follows, given a set $I\subset[0,\infty)$ we denote $\mathrm{cl}(I)$ its closure, relative to $[0,\infty)$.
For $x\in[0,\infty)$ we define $\cK_*(x)\coloneqq[2K_*(x),c^{-1}_A]$ and recall that $K_*$ is non-decreasing with $K_*(\infty)=1/(4 c_A)$. We must emphasise that a rigorous use of the DPP for controls from the class $\cB^*_x$ is not immediate, because those controls are both constrained and adapted to a generic filtration. Hence, the rigorous use of pseudo-Markovian property overarching the DPP is not granted (cf.\ \cite{claisse2016pseudo}). 
\begin{theorem}\label{thm:DPP_time}
Fix $(t,x)\in[0,T]\times(0,\infty)$. Let $I\subset (0,\infty)$ be a bounded open interval with $x\in I$ and $\mathrm{cl}(I)\subset (0,\infty)$. Given a process $\beta\in\cB$ we denote $\tau_I^\beta\coloneqq\inf\{s\ge 0: W_s^{x;\beta} \notin I\}$. Then, for all $\lambda\in[0,T-t]$, the function $\varphi^*$ satisfies the two properties below:
\begin{itemize}
\item[(i)] The upper bound holds,
\begin{equation}\label{eq:DPP2a}
\begin{aligned}
\varphi^*(t,x)\le \sup_{\beta\in\cB^*_x}\E\Big[\int_0^{\tau_I^\beta\wedge\lambda}(1-c_A\beta_u)\beta_u\ud u+\varphi^*\Big(t+\tau_I^\beta\wedge\lambda,W_{\tau_I^\beta\wedge\lambda}^{x;\beta}\Big)\Big].
\end{aligned}
\end{equation}
\item[(ii)] For any process $\beta\in\cB$ of the form 
\begin{equation}\label{eq:betaDPP}
\beta_s=\mathds{1}_{[0,\tau_I^{\beta}]}(s)\beta_0+\mathds{1}_{(\tau_I^{\beta},\infty)}(s)\frac{1}{c_A},\quad s\ge 0,
\end{equation}
with deterministic $\tfrac{1}{c_A}\ge \beta_0\ge 2K_*(\sup I)$, we have $\beta\in \cB_x^*$ and
\begin{equation*}
\begin{aligned}
\varphi^*(t,x)&\ge\E\Big[\int_0^{\tau_I^\beta\wedge\lambda}(1-c_A\beta_u)\beta_u\ud u+\varphi^*\Big(t+\tau_I^\beta\!\wedge\!\lambda,W_{\tau_I^\beta\wedge\lambda}^{x;\beta}\Big)\Big].
\end{aligned}
\end{equation*}
\end{itemize}
\end{theorem}
\begin{proof}
We start by proving (i). From \eqref{eq:DPP}, for any $\eps>0$ we can choose $\hat \alpha=\hat \alpha^\eps\in\cA^*_x(\F)$ such that 
\begin{equation}\label{eq:DPPalpha}
\varphi^*(t,x)\le \E\Big[\int_0^{\nu_I\wedge \zeta^\lambda(\hat \alpha) }(\hat \alpha_u-c_A)\ud u+\varphi^*\big(Z^{t;\hat \alpha}_{\nu_I\wedge \zeta^\lambda(\hat \alpha)},X_{\nu_I\wedge \zeta^\lambda(\hat \alpha)}^x\big)\Big]+\eps,
\end{equation}
where $\zeta^\lambda(\hat \alpha)\coloneqq\inf\{s\ge 0: Z^{t;\hat \alpha}_s=t+\lambda\}$ and $\nu_I=\nu_I(x)\coloneqq\inf\{s\ge 0:X^x_s\notin I\}$. We then introduce the process $\alpha_u=\hat \alpha_u$ for $u\in[\![0,\nu_I\wedge\zeta^\lambda(\hat \alpha)]\!]$ and $\alpha_u= c_A$ for $u\in(\!(\nu_I\wedge\zeta^\lambda(\hat \alpha),\infty)\!)$ and notice that $\nu_I\wedge\zeta^\lambda(\hat \alpha)=\nu_I\wedge\zeta^\lambda(\alpha)$.

We can use the strictly increasing process
\begin{equation*}
S^\alpha(u)\coloneqq Z^{t:\alpha}_u-t=\int_0^u \alpha^2_r\ud r,\quad u\ge 0,
\end{equation*}
to define a time-change. Letting the process $s\mapsto T^\alpha(s)$ be the inverse of $u\mapsto S^\alpha(u)$, we get
\begin{equation*}
T^\alpha(s)=\int_0^s\alpha^{-2}_{T^\alpha(u)}\ud u=\int_0^s\beta^{2}_{u}\ud u,\quad s\ge 0,
\end{equation*}
where we set $\beta_u\coloneqq \alpha^{-1}_{T^\alpha(u)}$. The process $M^\alpha_s\coloneqq B_{T^\alpha(s)}$ is a continuous martingale with respect to the time-changed filtration $\H^\alpha=(\cH^\alpha_s)_{s\ge 0}$ with $\cH^\alpha_s=\cF_{T^\alpha(s)}$, with quadratic variation $\langle M^\alpha\rangle_s=T^\alpha(s)$ (cf.\ \cite[Prop.\ V.1.5]{revuz2013continuous}). We also notice that by construction $\beta_u$ is $\H^\alpha$-progressively measurable (cf.\ \cite[Prop.\ V.1.4]{revuz2013continuous}) with $\beta_u\ge 2 K_*\big(X^x_{T^\alpha(u)}\big)\ge 2 K_*(\inf I)>0$ because of the definition of $\alpha_t$ in terms of $\hat \alpha\in\cA^*_x(\F)$.

Now, let us define a $\H^\alpha$-adapted process $(\widetilde B^\alpha_s)_{s\ge 0}$ as
\begin{equation*}
\widetilde B^\alpha_s\coloneqq\int_0^s\beta^{-1}_u\ud M^\alpha_u=\int_0^s\alpha_{T^\alpha(u)}\ud M^\alpha_u,\quad s\ge 0.
\end{equation*}
It is clear that $\widetilde B^\alpha$ is a continuous martingale with quadratic variation 
\begin{equation*}
\langle\widetilde B^\alpha\rangle_s=\int_0^s\beta^{-2}_u\ud\langle M^\alpha\rangle_u,
=\int_0^s\beta^{-2}_u\ud T^\alpha(u)=s.
\end{equation*}
and therefore it is a $\H^\alpha$-Brownian motion. Moreover, $M^\alpha_s=\int_0^s \beta_u\ud \widetilde B^\alpha_u$. Notice that $Z^{t;\alpha}_{T^\alpha(s)}=t+s$ and let us now define a process 
\begin{equation*}
W^{x;\beta}_s\coloneqq X^{x}_{T^\alpha(s)}=x+c_A T^\alpha(s)+\sigma M^\alpha_s=x+c_A\int_0^s\beta^2_u\ud u+\int_0^s \sigma\beta_u\ud \widetilde B^\alpha_u,\quad s\ge 0.
\end{equation*} 
We emphasise that we are slightly abusing notation, because $\beta$ depends on $\alpha$ and therefore we should use $\beta^\alpha$. However, we drop the superscript for the ease of notation. 

We recall that $\beta_s=\alpha^{-1}_{T^\alpha(s)}$ is $\H^\alpha$-progressively measurable and since $\alpha\in\cA^*_x(\F)$ (cf.\ \eqref{eq:AK}) it satisfies the bounds
\begin{equation*}
2K_*(X_{T^\alpha(s)}^x)=2K_*(W^{x;\beta}_s)\le \beta_s\le \frac{1}{c_A}.
\end{equation*}
Then $\beta=(\H^\alpha,(\beta_t)_{t\ge 0},(\widetilde B^\alpha_t)_{t\ge 0},(W^\beta_t)_{t\ge 0})\in\cB^*_x$. Recalling $\tau^\beta_I=\inf\{s\ge 0: W_s^{x;\beta}\notin I\}$, we have $\nu_I=0\iff x\notin I\iff \tau^\beta_I=0$. If instead $\tau^\beta_I>0$, we have
\begin{equation}\label{eq:change_exit}
\begin{aligned}
\tau^\beta_I>t\quad& \iff\quad W_u^{x;\beta}\in I\text{ for all $u\in[0,t]$}\\
&\iff\quad X_{T^\alpha(u)}^{x} \in I\text{ for all $u\in[0,t]$}\\
& \iff\quad X_{r}^{x}\in I\text{ for all $r \in[0,T^\alpha(t)]$}\iff T^\alpha(t)<\nu_I.
\end{aligned}
\end{equation}
Taking $t=S^\alpha(r)$ the above equivalence reads as $\nu_I>r\iff \tau^\beta_I>S^\alpha(r)$.
By continuity of $T^\alpha(\cdot)$, we have $\lim_{t\uparrow \tau^\beta_I} T^{\alpha}(t)=T^{\alpha}(\tau^\beta_I)\le \nu_I$. In particular, $\tau^\beta_I=\infty\implies \nu_I=\infty$. By the above chain of equivalences we can also show the reverse implication, i.e., $\nu_I=\infty\implies \tau^\beta_I=\infty$. Moreover, on $\{\nu_I<\infty\}=\{\tau^\beta_I<\infty\}$ we have $X^{x}_{T^\alpha(\tau^\beta_I)}=W_{\tau^\beta_I}^{x;\beta}\in \partial I$, which implies $T^{\alpha}(\tau^\beta_I)\ge \nu_I$. Combining the two inequalities we have $\nu_I=T^{\alpha}(\tau^\beta_I)$ and therefore $S^\alpha(\nu_I)=\tau^\beta_I$ because the time change is invertible. Now, notice that
$S^\alpha(\nu_I\wedge \zeta^\lambda(\alpha))=S^\alpha(\nu_I)\wedge S^\alpha(\zeta^\lambda(\alpha))=\tau^\beta_I\wedge \lambda$
and $W^{x;\beta}_{S^\alpha(u)}=X^x_u$.
Therefore, we deduce 
\begin{equation*}
\begin{aligned}
\Big(t+\int_0^{\nu_I\wedge\zeta^\lambda(\alpha)}\alpha^2_s\ud s, X^{x}_{\nu_I\wedge\zeta^\lambda(\alpha)}\Big)=\Big(t+S^\alpha(\nu_I)\wedge \lambda, W^{x;\beta}_{S^\alpha(\nu_I)\wedge\lambda}\Big)=\Big(t+\tau_I^\beta\wedge\lambda,W_{\tau_I^\beta\wedge\lambda}^{x;\beta}\Big).
\end{aligned}
\end{equation*}
Furthermore, we have 
\begin{equation*}
\begin{aligned}
\int_0^{\nu_I\wedge\zeta^\lambda(\alpha)}\big(\alpha_s-c_A\big)\ud s=\int_0^{\nu_I\wedge\zeta^\lambda(\alpha)}\big(\beta^{-1}_{S^\alpha(u)}-c_A\big)\ud u=\int_0^{\tau_I^\beta\wedge\lambda}(1-c_A\beta_s)\beta_s\ud s,
\end{aligned}
\end{equation*}
where in the final equality we changed variable of integration setting $u=T^\alpha(s)$. Recalling \eqref{eq:DPPalpha} and $\alpha_t=\hat\alpha_t$ for $t\in[\![0,\nu_I\wedge\zeta^\lambda(\hat\alpha)]\!]=[\![0,\nu_I\wedge\zeta^\lambda(\alpha)]\!]$ we deduce
\begin{equation*}
\begin{aligned}
\varphi^*(t,x)\le \E\Big[\int_0^{\tau_I^\beta\wedge\lambda}(1-c_A\beta_u)\beta_u\ud u+\varphi^*(t+\tau_I^\beta\wedge\lambda,W_{\tau_I^\beta\wedge\lambda}^{x;\beta})\Big]+\eps.
\end{aligned}
\end{equation*}
Recall that $\beta_t=\beta^\alpha_t$ by construction. Since $\beta=(\H^\alpha,(\beta^\alpha_t)_{t\ge 0},(\widetilde B^\alpha_t)_{t\ge 0},(W^{\beta^\alpha}_t)_{t\ge 0})\in\cB^*_x$, then 
\begin{equation*}
\begin{aligned}
\varphi^*(t,x)\le \sup_{\beta\in\cB^*_x}\E\Big[\int_0^{\tau_I^\beta\wedge\lambda}(1-c_A\beta_u)\beta_u\ud u+\varphi^*(t+\tau_I^\beta\wedge\lambda,W_{\tau_I^\beta\wedge\lambda}^{x;\beta})\Big]+\eps,
\end{aligned}
\end{equation*}
and letting $\eps\to 0$ concludes the proof of (i). 

In the remainder of the proof we show (ii). It is clear by definition of the process $\beta$ in \eqref{eq:betaDPP} that it belongs to $\cB_x^*$, because $2K_*(W^{x;\beta}_t)\le 2K_*(\sup I)\le 1/c_A$ for $t\in[\![0,\tau^\beta_I]\!]$. Indeed, all arguments here are valid for an arbitrary control $\beta\in\cB^*_x$, with the sole exception of the use of DPP in the final paragraph of the proof. Thus, we give our arguments in slight more generality than needed for (ii). 

Let us fix a control $\beta=(\F',(\beta_t)_{t\ge0}, (B'_t)_{t\ge 0},(W^\beta_t)_{t\ge 0})\in\cB^*_x$ with $\beta_t\ge \delta$ for all $t\ge 0$ for some $\delta>0$. The process
$T_\beta(t)\coloneqq\int_0^t\beta_u^2\ud u$ is strictly increasing, absolutely continuous with $T_\beta(0)=0$ and $T_\beta(\infty)=\infty$. 
Its inverse $S_\beta(t)\coloneqq (T_\beta)^{-1}(t)$ satisfies $S_\beta(0)=0$, $S_\beta(\infty)=\infty$ and 
\begin{equation*}
\frac{\ud }{\ud t}S_\beta(t)=\frac{1}{\beta^2_{S_\beta(t)}}.
\end{equation*}
Then, the representation $S_\beta(t)=\int_0^t \beta^{-2}_{S_\beta(s)}\ud s$ holds for $t\ge 0$.

By Dambis--Dubins--Schwarz' theorem \cite[Thm.\ 3.4.6]{karatzas2014brownian}, the process $(\widetilde B^\beta_t)_{t\ge 0}$ defined by
$\widetilde B^\beta_t=\int_0^{S_\beta(t)}\beta_s\ud B'_s$,
is a Brownian motion for the time-changed filtration $\G^\beta=(\cG^\beta_t)_{t\ge 0}$, with $\cG^\beta_t\coloneqq \cF'_{S_\beta(t)}$. Moreover, 
$\int_0^s\beta_{u}\ud B'_u=\widetilde B^\beta_{T_\beta(t)}$.
Using the formula above we easily see that 
\begin{equation*}
W^{x;\beta}_t=x+c_A T_\beta(t)+\sigma\widetilde B^\beta_{T_\beta(t)}= X^{x;\beta}_{T_\beta(t)},
\end{equation*}
where $X^{x;\beta}_t=x+c_A t+\sigma\widetilde B^\beta_t$ for $t\ge 0$. By the time change, $X^{x;\beta}_t=W^{x;\beta}_{S_\beta(t)}$.

Letting $\nu^\beta_I\coloneqq\inf\{s\ge 0: X_s^{x;\beta}\notin I\}$, the exact same arguments as in \eqref{eq:change_exit}
yield $\nu^\beta_I=T_{\beta}(\tau^\beta_I)$ and therefore $S_\beta(\nu^\beta_I)=\tau^\beta_I$ because the time change is invertible. It is also worth noticing that, because $T_\beta$ is the inverse of $S_\beta$ and the latter is continuous, for any $\lambda>0$, 
\begin{equation*}
T_\beta(\lambda)=\inf\big\{t\ge 0:S_\beta(t)=\lambda\big\}=\inf\Big\{t\ge 0:\int_0^t\beta^{-2}_{S_\beta(s)}\ud s=\lambda\Big\}\eqqcolon \hat \sigma^\lambda_\beta.
\end{equation*}
Now $T_\beta(\tau_I^\beta\wedge \lambda)=T_\beta(\tau_I^\beta)\wedge T_\beta(\lambda)=\nu^\beta_I\wedge \hat \sigma^\lambda_\beta$ and, by taking the inverse, also 
\begin{equation*}
\tau^\beta_I\wedge\lambda=S_\beta(\nu^\beta_I\wedge\hat \sigma^\lambda_\beta)=\int_0^{\nu^\beta_I\wedge\hat \sigma^\lambda_\beta}\beta^{-2}_{S_\beta(s)}\ud s=\int_0^{\nu^\beta_I\wedge\hat \sigma^\lambda_\beta}\alpha^2_s\ud s,
\end{equation*} 
where we set $\alpha_t\coloneqq 1/\beta_{S_\beta(t)}$. 

The above formulae allow us to represent the pair $\big(t+\tau_I^\beta\wedge\lambda,W_{\tau_I^\beta\wedge\lambda}^{x;\beta}\big)$ as follows:
\begin{equation*}
\begin{aligned}
(t+\tau_I^\beta\wedge\lambda,W_{\tau_I^\beta\wedge\lambda}^{x;\beta})&=\Big(t+\int_0^{\nu^\beta_I\wedge\hat \sigma^\lambda_\beta}\beta^{-2}_{S_\beta(s)}\ud s, X^{x;\beta}_{\nu^\beta_I\wedge\hat\sigma^\lambda_\beta}\Big)=\Big(t+\int_0^{\nu^\beta_I\wedge\hat \sigma^\lambda_\beta}\alpha^2_s\ud s, X^{x;\beta}_{\nu^\beta_I\wedge\hat\sigma^\lambda_\beta}\Big).
\end{aligned}
\end{equation*}
Since $\alpha_{T_\beta(t)}=\beta^{-1}_t$, we have by a simple change of variable of integration
\begin{equation*}
\begin{aligned}
\int_0^{\tau_I^\beta\wedge\lambda}(1-c_A\beta_u)\beta_u\ud u&=\int_0^{\tau_I^\beta\wedge\lambda}\big(\alpha_{T_\beta(u)}-c_A\big)\ud T_\beta(u)=\int_0^{\nu^\beta_I\wedge\hat\sigma^\lambda_\beta}\big(\alpha_s-c_A\big)\ud s.
\end{aligned}
\end{equation*}
Then, for any $\beta\in \cB^*_x$ with $\beta_t\ge \delta>0$ for all $t\ge 0$ we have
\begin{equation}\label{eq:timechangeDPP}
\begin{aligned}
&\E\Big[\int_0^{\tau_I^\beta\wedge\lambda}(1-c_A\beta_u)\beta_u\ud u+\varphi^*\big(t+\tau_I^\beta\wedge\lambda,W_{\tau_I^\beta\wedge\lambda}^{x;\beta}\big)\Big]\\
&=\E\Big[\int_0^{\nu^\beta_I\wedge\hat\sigma^\lambda_\beta}\big(\alpha_s-c_A\big)\ud s+\varphi^*\Big(t+\int_0^{\nu^\beta_I\wedge\hat \sigma^\lambda_\beta}\alpha^2_s\ud s, X^{x;\beta}_{\nu^\beta_I\wedge\hat\sigma^\lambda_\beta}\Big)\Big].
\end{aligned}
\end{equation}
Finally, notice that $(\alpha_t)_{t\ge 0}$ is progressively measurable for the time-changed filtration $\G^\beta$ (cf.\ \cite[Prop.\ V.1.4]{revuz2013continuous}) and it satisfies the bounds
\begin{equation*}
c_A\le \alpha_t\le \frac{1}{2 K_*(W^{x;\beta}_{S_\beta(t)})}=\frac{1}{2 K_*(X^{x;\beta}_{t})},\quad t\ge 0,
\end{equation*}
because $2K_*(W^{x;\beta}_t)\le \beta_t\le 1/c_A$. 

For our choice of $\beta$ in \eqref{eq:betaDPP}, the form of the process $W^{x;\beta}$ is particularly simple and it is easy to check that $\tau^\beta_I$ is a stopping time for the filtration generated by the increments of the Brownian motion $B'$. Therefore, the process $\beta$ itself is adapted to such filtration and it is not hard to verify that the process $\alpha_t=1/\beta_{S_\beta(t)}$ can be expressed as
\begin{equation*}
\alpha_t=\mathds{1}_{[0,\nu_I^{\beta}]}(t)\frac{1}{\beta_0}+\mathds{1}_{(\nu_I^{\beta},\infty)}(t)c_A,\quad t\ge 0.
\end{equation*}
Because of this construction, both the process $X^{x;\beta}$ and the control $(\alpha_t)_{t\ge 0}$ are adapted to the filtration $\widetilde\F^\beta=(\widetilde\cF^\beta_t)_{t\ge 0}$ generated by the increments of the Brownian motion $\widetilde B^\beta$. 
Therefore, the control $(\alpha_t)_{t\ge 0}$ 
is one particular process in the class $\cA^*_x(\widetilde\F^\beta)$. This allows the use of DPP as in \eqref{eq:DPP}: for an arbitrary $\hat\alpha\in\cA^*_x(\widetilde\F^\beta)$ we set $\zeta^\lambda(\hat \alpha)=\inf\{s\ge 0: Z^{t;\hat \alpha}_s=t+\lambda\}$
and \eqref{eq:timechangeDPP} yields
\begin{equation*}
\begin{aligned}
&\E\Big[\int_0^{\tau_I^\beta\wedge\lambda}(1-c_A\beta_u)\beta_u\ud u+\varphi^*\big(t+\tau_I^\beta\wedge\lambda,W_{\tau_I^\beta\wedge\lambda}^{x;\beta}\big)\Big]\\
&\le \sup_{\hat \alpha\in\cA^*_x(\widetilde\F^\beta)}\E\Big[\int_0^{\nu^\beta_I\wedge\zeta^\lambda(\hat \alpha)}\big(\hat \alpha_s-c_A\big)\ud s+\varphi^*\Big(Z^{t;\hat \alpha}_{\nu^\beta_I\wedge\zeta^\lambda(\hat \alpha)}, X^{x;\beta}_{\nu^\beta_I\wedge\zeta^\lambda(\hat \alpha)}\Big)\Big]=\varphi^*(t,x),
\end{aligned}
\end{equation*}
where the final equality holds thanks to the DPP in \eqref{eq:DPP}. This completes the proof of (ii). 
\end{proof}

\section{The Hamilton-Jacobi-Bellman equation}\label{sec:HJB}

Theorem \ref{thm:DPP_time} of the previous section gives us a form of the Dynamic Programming Principle that we are now going to use in order to derive an HJB equation for $\varphi^*$. Let us introduce some notations. Given a sufficiently regular function $f:[0,T]\times[0,\infty)\to \R$ and a constant $\beta\in\R$, we denote 
\begin{equation*}
H(t,x;\beta,f)=\Big(\frac{\sigma^2}{2}\beta^2\partial_{xx}f\!+\!c_A\beta^2\partial_{x}f\Big)(t,x)+(1-c_A\beta)\beta,
\end{equation*}
for $(t,x)\in[0,T]\times[0,\infty)$. For $x\in[0,\infty)$ recall $\cK_*(x)=[2K_*(x),c^{-1}_A]$ and let us introduce the Hamiltonian of our problem:
\begin{equation}\label{eq:HK}
\cH^*_\cK(t,x;f)\coloneqq\sup_{\beta\in\cK_*(x)}H(t,x;\beta,f).
\end{equation}
Our next theorem turns the DPP-like results from Theorem \ref{thm:DPP_time} into an HJB equation.
\begin{theorem}\label{thm:PDE_q-orig}
The function $\varphi^*$ belongs to $W^{1,2;\infty}_{\ell oc}([0,T]\times[0,\infty))\cap C([0,T]\times[0,\infty))$ and it satisfies 
\begin{equation}\label{eq:PDE_q-orig}
\begin{cases}\displaystyle
\partial_t\varphi^*(t,x)+\cH^*_{\cK}(t,x;\varphi^*)=0& \text{a.e.}\ (t,x)\in[0,T)\times(0,\infty),\\
\varphi^*(t,0)=0& \forall t\in[0,T),\\
\varphi^*(T,x)=0& \forall x\in[0,\infty).
\end{cases}
\end{equation}
\end{theorem}

\begin{proof}
The regularity of $\varphi^*$ was shown in Proposition \ref{prop:W12p}, then we only need to show that it solves the HJB. Clearly the function $\partial_t\varphi^*(t,x)+\cH^*_\cK(t,x;\varphi^*)$ is well-defined for a.e.\ $(t,x)$. 
The rest of the proof is divided into two main steps. 

{\em Step 1.} Here we show $\partial_t\varphi^*(t,x)+\cH^*_\cK(t,x;\varphi^*)\le 0$ 
for almost every $(t,x)\in [0,T]\times(0,\infty)$. Fix $(t_0,x_0)\in [0,T)\times(0,\infty)$. There is $R>0$ such that $x_0-2R>0$ and $t_0+2R\le T$. Therefore also $x_0-\theta-R>0$ and $t_0+\theta+R<T$ for every $\theta\in(0,R)$. We fix a constant $\beta_0\in \cK_*(x_0+2R)$ and for any $x\in(x_0-R,x_0+R)$ we consider a control $\beta\in\cB^*_{x}$ with process
$(\beta_s)_{s\in[0,\infty)}$ given by 
\begin{equation*}
\beta_s=\mathds{1}_{[0,\tau_R^{\beta}(x))}(s)\beta_0+\mathds{1}_{[\tau_R^{\beta}(x),\infty)}(s)\frac{1}{c_A},\quad \text{for $s\ge0$}, 
\end{equation*}
where 
$\tau_R^\beta(x)\coloneqq \inf\{s\ge 0 : W_{s}^{x;\beta} \notin (x-R,x+R)\}$.

Fix $\lambda\in(0,R]$. Then, by Theorem \ref{thm:DPP_time}-(ii) we have for $(t,x)\in(t_0,t_0+R)\times(x_0-R,x_0+R)$,
\begin{equation*}
\begin{aligned}
\varphi^*(t,x)&\ge\E\Big[\int_0^{\tau_R^\beta(x)\wedge\lambda}(1-c_A\beta_u)\beta_u\ud u+\varphi^*\big(t+\tau_R^\beta(x)\!\wedge\!\lambda,W_{\tau_R^\beta(x)\wedge\lambda}^{x;\beta}\big)\Big]\\
&=\varphi^*(t,x)+\E\Big[\!\int_0^{\tau_R^\beta(x)\wedge \lambda}\!\big[\big(\partial_t \varphi^* \!+\!\tfrac{1}{2}\sigma^2\beta_u^2\partial_{xx}\varphi^*\!+\!c_A\beta_u^2\partial_{x}\varphi^*\big)(t\!+\!s,W_s^{x;\beta})\!+\!(1\!-\!c_A\beta_s)\beta_s\big]\ud s\Big],
\end{aligned}
\end{equation*}
where the equality holds by an application of It\^o--Krylov's formula (cf.\ \cite[Thm.\ 2.10.1]{krylov2009controlled}) which we recall in Appendix \ref{app:Ito-Krylov} for the reader's convenience. Since $\beta_u=\beta_0$ for $u\in[\![0,\tau^\beta_R(x)]\!]$, then the above inequality yields
\begin{equation}\label{eq:sub_optim}
\begin{aligned}
\E\Big[\int_0^{\tau_R^{\beta}(x)\wedge \lambda}\!\partial_t\varphi^*(t\!+\!s,W_s^{x;\beta_0})\!+\!H(t\!+\!s,W_s^{x;\beta_0};\beta_0,\varphi^*)\ud s\Big]\le 0,
\end{aligned}
\end{equation}
where $W_s^{x;\beta_0}=x+c_A\beta_0^2s+\sigma\beta_0 B_s$.
In particular, for a fixed $\theta\in(0,R),$ \eqref{eq:sub_optim} holds for all $(t,x)\in [t_0,t_0+\theta]\times [x_0-\theta,x_0+\theta]$ so that
\begin{equation}\label{eq:sub_optim1}
\begin{aligned}
0\ge &\,\frac{1}{\lambda}\int_{t_0}^{t_0+\theta}\int_{x_0-\theta}^{x_0+\theta}\E\Big[\int_0^{\tau_R^{\beta}(x)\wedge \lambda}\!\!\Big(\partial_t\varphi^*(t\!+\!s,W_s^{x;\beta_0})\!+\!H(t\!+\!s,W_s^{x;\beta_0};\beta_0,\varphi^*)\Big)\ud s\Big]\ud x\ud t\\
=&\frac{1}{\lambda}\int_{t_0}^{t_0+\theta}\int_{x_0-\theta}^{x_0+\theta}\E\Big[\int_0^{\tau_R^{\beta}(x)\wedge \lambda}\!\!\big(\partial_t\varphi^*(t\!+\!s,W_s^{x;\beta_0})\!+\!H(t\!+\!s,W_s^{x;\beta_0};\beta_0,\varphi^*)\\
&\qquad\qquad\qquad\qquad\qquad\qquad\qquad\qquad\qquad-\!\partial_t\varphi^*(t,x)\!-\!H(t,x;\beta_0,\varphi^*)\big)\ud s\Big]\ud x\ud t\\
&+\int_{t_0}^{t_0+\theta}\int_{x_0-\theta}^{x_0+\theta}\big(\partial_t\varphi^*(t,x)\!+\!H(t,x;\beta_0,\varphi^*)\big)\E\Big[\frac{\tau_R^{\beta}(x)\wedge \lambda}{\lambda}\Big]\ud x\ud t\eqqcolon I_1(\lambda)+I_2(\lambda).
\end{aligned}
\end{equation}

The aim is to let $\lambda\downarrow0$. 
For $\cO_R\coloneqq [t_0,t_0\!+\!2R]\times[x_0\!-\!2R,x_0\!+\!2R]$ we denote $\Lambda^\cO_R(t,x;\beta_0,\varphi^*)\coloneqq \mathds{1}_{\cO_R}(t,x)\big(\partial_t\varphi^*(t,x)+H(t,x;\beta_0,\varphi^*)\big)$. Then we have
\begin{equation*}
\begin{aligned}
&\big|\mathds{1}_{\{s<\tau_R^{\beta}(x)\}}\!\big(\partial_t\varphi^*(t\!+\!s,W_s^{x;\beta_0})\!+\!H(t\!+\!s,W_s^{x;\beta_0};\beta_0,\varphi^*)\!-\!\partial_t\varphi^*(t,x)\!-\!H(t,x;\beta_0,\varphi^*)\big)\big|\\
&=\big|\mathds{1}_{\{s<\tau_R^{\beta}(x)\}}\!\big(\Lambda^\cO_R(t\!+\!s,W_s^{x;\beta_0};\beta_0,\varphi^*)\!-\!\Lambda^\cO_R(t,x;\beta_0,\varphi^*)\big)\big| \\
&\le\big|\Lambda^\cO_R(t\!+\!s,W_s^{x;\beta_0};\beta_0,\varphi^*)-\Lambda^\cO_R(t,x;\beta_0,\varphi^*)\big|.
\end{aligned}
\end{equation*}
Recall the form of the dynamics of $W^{x;\beta_0}$ from \eqref{eq:Wbetaweak} and notice that the mapping 
\begin{equation*}
(t,x,s,\omega)\mapsto \Lambda^\cO_R(t\!+\!s,W_s^{x;\beta_0}(\omega);\beta_0,\varphi^*)-\Lambda^\cO_R(t,x;\beta_0,\varphi^*),
\end{equation*}
is jointly measurable and that $\Lambda^\cO_R\in (L^{\infty}\cap L^1)(\R^2)$. 
We use Fubini's theorem to exchange the order of integration in the expression for $I_1(\lambda)$ and we obtain
\begin{equation*}
| I_1(\lambda)|\le \E\Big[\frac{1}{\lambda}\int_0^{\lambda}\!\big\|\Lambda^\cO_R(\,\cdot+\!s,\cdot+c_A\beta_0^2s+\sigma\beta_0 B'_s;\beta_0,\varphi^*)-\Lambda^\cO_R(\,\cdot\,,\,\cdot\,;\beta_0,\varphi^*)\big\|_{L^1(\R^2)}\ud s\Big].
\end{equation*}
For almost every $\omega\in\Omega$, 
\begin{equation*}
\lim_{s\to0}\big\|\Lambda^\cO_R(\,\cdot+\!s,\cdot+c_A\beta_0^2s+\sigma\beta_0 B'_s(\omega);\beta_0,\varphi^*)-\Lambda^\cO_R(\,\cdot\,;\beta_0,\varphi^*)\big\|_{L^1(\R^2)}=0,
\end{equation*}
by well-known continuity of Brownian paths and of the $L^p$-norms\footnote{For $p\in[1,\infty)$ and $f\in (L^\infty\cap L^p)(\R^2)$, and for any $\eps>0$, there is $f_\eps\in C(\R^2)$ with compact support such that $\|f_\eps-f\|_{L^p(\R^2)}<\eps$. Therefore, for any $u\in\R$, 
\begin{equation*}
\begin{aligned}
\|f(\cdot+\!u,\cdot+\!u)-f\|_{L^p(\R^2)}\le&\,\|f(\cdot+\!u,\cdot+\!u)\!-\!f_\eps(\cdot+\!u,\cdot+\!u)\|_{L^p(\R^2)}\!+\!\|f_\eps(\cdot+\!u,\cdot+\!u)\!-\!f_\eps\|_{L^p(\R^2)}\\
&+\|f_\eps-f\|_{L^p(\R^2)}\\
<&\,2\eps\!+\!\|f_\eps(\cdot+\!u,\cdot+\!u)\!-\!f_\eps\|_{L^p(\R^2)}.
\end{aligned}
\end{equation*}
Thus, $\lim_{u\to0}\|f(\cdot+\!u,\cdot+\!u)-f\|_{L^p(\R^2)}<2\eps$. Letting $\eps\to 0$ we conclude.}. In particular, that implies
\begin{equation*}
\lim_{\lambda\downarrow0}\frac{1}{\lambda}\int_0^{\lambda}\!\big\|\Lambda^\cO_R(\,\cdot+\!s,\cdot+c_A\beta_0^2s+\sigma\beta_0 B'_s;\tilde{\beta},\varphi^*)-\Lambda^\cO_R(\,\cdot\,;\beta_0,\varphi^*)\big\|_{L^1(\R^2)}\ud s=0,
\end{equation*}
and by dominated convergence 
\begin{equation}\label{eq:sub_optim2}
\begin{aligned}
\lim_{\lambda\downarrow0} | I_1(\lambda)|=0.
\end{aligned}
\end{equation}
For the term $I_2(\lambda)$, 
we notice that $\tau_R^{\beta}(x)(\omega)>0$ for all $x\in(x_0-\theta,x_0+\theta)$ and almost every $\omega\in\Omega$ by continuity of paths of $W^{x;\beta_0}_s$.
Then, for every $x\in(x_0-\theta,x_0+\theta)$
\begin{equation*}
\lim_{\lambda\downarrow0}\E\Big[\frac{\tau_R^{\beta}(x)\wedge \lambda}{\lambda}\Big]=1,
\end{equation*}
by Dominated convergence. Using Dominated convergence once again we conclude
\begin{equation}\label{eq:sub_optim3}
\begin{aligned}
\lim_{\lambda\downarrow 0} I_2(\lambda)&=\int_{t_0}^{t_0+\theta}\int_{x_0-\theta}^{x_0+\theta}\big(\partial_t\varphi^*(t,x)+H(t,x;\beta_0,\varphi^*)\big)\lim_{\lambda\downarrow 0}\E\Big[\frac{\tau_R^{\beta}(x)\wedge \lambda}{\lambda}\Big]\ud x\ud t\\
&=\int_{t_0}^{t_0+\theta}\int_{x_0-\theta}^{x_0+\theta}\big(\partial_t\varphi^*(t,x)+H(t,x;\beta_0,\varphi^*)\big)\ud x\ud t.
\end{aligned}
\end{equation}
Plugging \eqref{eq:sub_optim2} and \eqref{eq:sub_optim3} into \eqref{eq:sub_optim1}, we get 
\begin{equation}\label{eq:Hint}
\int_{t_0}^{t_0+\theta}\!\!\int_{x_0-\theta}^{x_0+\theta}\big(\partial_t\varphi^*(t,x)+H(t,x;\beta_0,\varphi^*)\big)\ud x\ud t\le 0.
\end{equation}
By the Lebesgue theorem \cite[Thm.\ 1.32]{evans2015measure}, there is a set $N\subset[0,T]\times[0,\infty)$ of zero Lebesgue measure such that for $(t_0,x_0)\notin N$ we have
\begin{equation*}
\begin{aligned}
&\lim_{\theta\to 0}\frac{1}{2\theta^{2}}\int_{t_0}^{t_0+\theta}\!\!\int_{x_0-\theta}^{x_0+\theta}\partial_{t}\varphi^*(t,x)\ud x\ud t=\partial_t\varphi^*(t_0,x_0),\\
&\lim_{\theta\to 0}\frac{1}{2\theta^{2}}\int_{t_0}^{t_0+\theta}\!\!\int_{x_0-\theta}^{x_0+\theta}\partial_{x}\varphi^*(t,x)\ud x\ud t=\partial_x\varphi^*(t_0,x_0),\\
&\lim_{\theta\to 0}\frac{1}{2\theta^{2}}\int_{t_0}^{t_0+\theta}\!\!\int_{x_0-\theta}^{x_0+\theta}\partial_{xx}\varphi^*(t,x)\ud x\ud t=\partial_{xx}\varphi^*(t_0,x_0).
\end{aligned}
\end{equation*}
Then, \eqref{eq:Hint} yields $\partial_t\varphi^*(t_0,x_0)+H(t_0,x_0;\beta_0,\varphi^*)\le 0$ for a.e.\ $(t_0,x_0)\in[0,T]\times[0,\infty)$. Recall that $\beta_0$ was taken from the set $\cK_*(x_0+2R)$, but $R>0$ can be taken arbitrarily small and therefore, by continuity of $x\mapsto K_*(x)$ and $\beta\mapsto H(t_0,x_0;\beta,\varphi^*)$, we conclude that $\beta_0$ can be chosen in $\cK_*(x_0)$. 
Thus, we have shown
\begin{equation}\label{eq:zeroL}
\partial_t\varphi^*(t,x)+\cH^*_\cK(t,x;\varphi^*)=\partial_t\varphi^*(t,x)+\sup_{\beta\in \cK_*(x)}H(t,x;\beta,\varphi^*)
\le 0,\quad\text{a.e.\ $(t,x)\in[0,T]\times[0,\infty)$}.
\end{equation}

{\em Step 2.} Now we prove the reverse inequality to \eqref{eq:zeroL}. We use again It\^o--Krylov's formula in the right-hand side of \eqref{eq:DPP2a} in Theorem \ref{thm:DPP_time}-(i). For $(t_0,x_0)\in[0,T)\times(0,\infty)$ and $R\in(0,x_0)$, we get 
\begin{equation}\label{eq:HJBineq}
\begin{aligned}
0\le&\sup_{\beta\in\cB^*_{x_0}}\E\Big[\int_0^{\tau_R^\beta(x_0)\wedge R}\big[\big(\partial_t \varphi^* +\frac{\sigma^2}{2}\beta_u^2\partial_{xx}\varphi^*+c_A\beta_u^2\partial_{x}\varphi^*\big)(t_0+s,W_s^{x_0;\beta})+(1-c_A\beta_s)\beta_s\big]\ud s\Big]\\
\le &\sup_{\beta\in\cB^*_{x_0}}\E\Big[\int_0^{\tau_R^\beta(x_0)\wedge R}\big(\partial_t\varphi^*+\cH^*_\cK\big(\,\cdot\,;\varphi^*\big)\big)(t_0+s,W_s^{x_0;\beta})\ud s\Big]\\
\le &\sup_{\beta\in\cB^*_{x_0}}\E\Big[\int_0^{\tau_R^\beta(x_0)\wedge R}\mathds{1}_{\cS}(\cdot)\big(\partial_t\varphi^*+\cH^*_\cK\big(\,\cdot\,;\varphi^*\big)\big)(t_0+s,W_s^{x_0;\beta})\ud s\Big],
\end{aligned}
\end{equation}
where the second inequality is simply by definition of $\cH^*_\cK$ and in the third one we set 
\begin{equation*}
\cS\coloneqq \Big\{(t,x)\in [0,T]\times[0,\infty):\partial_t\varphi^*(t,x)+\cH^*_\cK(t,x;\varphi^*)> 0\Big\}.
\end{equation*} 
For any $\beta\in\cB^*_{x_0}$ it holds $0<2K_*(x_0-R)\le\beta_s\le c^{-1}_A$ for $s\le \tau^\beta_R(x_0)\wedge R$. Then, from \cite[Thm.\ 2.2.4]{krylov2009controlled}, there is a constant $C_\beta>0$ such that
\begin{equation*}
\begin{aligned}
0&\le \E\Big[\int_0^{\tau_R^\beta(x_0)\wedge R}\mathds{1}_{\cS}(\cdot)\big(\partial_t\varphi^*+\cH^*_\cK\big(\cdot;\varphi^*\big)(t_0+s,W_s^{x_0;\beta})\big)\ud s\Big]\\
&\le C_\beta \big\|\mathds{1}_{\cS}(\cdot)\big(\partial_t\varphi^*+\cH^*_\cK(\,\cdot\,;\varphi^*)\big)\big\|_{L^{p}(\cO_R)}=0,
\end{aligned}
\end{equation*}
where the final equality holds because $\cS$ is of zero Lebesgue measure due to \eqref{eq:zeroL}. Combining the above with \eqref{eq:HJBineq} 
we deduce
\begin{equation*}
\sup_{\beta\in\cB^*_{x_0}}\E\Big[\int_0^{\tau_R^\beta(x_0)\wedge R}\big(\partial_t\varphi^*+\cH^*_\cK\big(\,\cdot\,;\varphi^*\big)\big)(t_0+s,W_s^{x_0;\beta})\ud s\Big]=0.
\end{equation*}
Thanks to \eqref{eq:zeroL} and the above equation, we apply \cite[Cor.\ 2.4.8]{krylov2009controlled}, whose statement and main idea of proof are given in Corollary \ref{cor:krylov248} for completeness, so that
\begin{equation*}
\partial_t\varphi^*(t,x)+\cH^*_\cK\big(t,x;\varphi^*\big)= 0,\quad\text{a.e.\ $(t,x)\in (t_0,t_0+R)\times(x_0-R,x_0+R)$}.
\end{equation*}
By arbitrariness of $t_0$, $x_0$ and $R$, we have
\begin{equation*}
\partial_t\varphi^*(t,x)+\cH^*_\cK\big(t,x;\varphi^*\big)= 0, \quad\text{a.e.\ $(t,x)\in[0,T)\times(0,\infty)$}.
\end{equation*}
Thus, $\varphi^*$ is a strong solution of \eqref{eq:PDE_q-orig}.
\end{proof}

In the Hamiltonian in \eqref{eq:HK} the set of controls is constrained to the interval $\cK_*(x)$. In the next proposition we show that indeed the maximiser $\beta^*$ lies always in the interior of $\cK_*(x)$ and therefore we can remove the constraint altogether. Recalling \eqref{eq:HK} we introduce the unconstrained Hamiltonian 
\begin{equation*}
\cH(t,x;f)\coloneqq\sup_{\beta\in\R}H(t,x;\beta,f)
\end{equation*}
and we state our next theorem. 
\begin{theorem}\label{thm:HJB}
The function $\varphi^*\in W^{1,2;\infty}_{\ell oc}([0,T]\times[0,\infty))\cap C([0,T]\times[0,\infty))$ satisfies 
\begin{equation}\label{eq:PDE_original}
\begin{cases}\displaystyle\partial_{t}\varphi^*+\cH(t,x;\varphi^*)=0,& \text{a.e.\ on }[0,T)\times(0,\infty),\\
\varphi^*(t,0)=0,& \forall t\in[0,T),\\
\varphi^*(T,x)=0,& \forall x\in[0,\infty).
\end{cases}
\end{equation}
\end{theorem}
\begin{proof}
Fix $(t,x)\in[0,T)\times(0,\infty)$ such that \eqref{eq:PDE_q-orig} holds at that point. For all $\beta\in \cK_*(x)$ we have
\begin{equation*}
\partial_{t}\varphi^*(t,x)+\frac{\sigma^2}{2}\beta^2\partial_{xx}\varphi^*(t,x)+c_A\beta^2\partial_{x}\varphi^*(t,x)+(1-c_A\beta)\beta\le 0.
\end{equation*}
Since $\beta\in\cK_*(x)$ implies $\beta>0$, multiplying by $\frac{1}{\beta^2}$, we obtain
\begin{equation*}
\frac{1}{\beta^2}\partial_{t}\varphi^*(t,x)+\frac{\sigma^2}{2}\partial_{xx}\varphi^*(t,x)+c_A\partial_{x}\varphi^*(t,x)+\frac{1}{\beta}-c_A\le 0.
\end{equation*}
Setting $\alpha=1/\beta$ we have
\begin{equation*}
\frac{\sigma^2}{2}\partial_{xx}\varphi^*(t,x)+c_A\partial_{x}\varphi^*(t,x)+\alpha^2\partial_{t}\varphi^*(t,x)+\alpha-c_A\le 0.
\end{equation*}
Since $\beta\in\cK_*(x)$ was arbitrary, we have $\alpha\in[c_A,\frac{1}{2K_*(x)}]\eqqcolon \cK'_*(x)$ (recall the definition of $\cK'(x)$ before \eqref{eq:hamil}) and we can take the supremum over $\alpha$, yielding
\begin{equation*}
\frac{\sigma^2}{2}\partial_{xx}\varphi^*(t,x)+c_A\partial_{x}\varphi^*(t,x)+\sup_{\alpha\in\cK'_*(x)}\big[\alpha^2\partial_{t}\varphi^*(t,x)+\alpha\big]-c_A\le 0.
\end{equation*}

In order to show that the above inequality is actually an equality, we notice that $\beta\mapsto H(t,x;\beta,\varphi^*)$ is continuous and $\cK_*(x)$ is compact. Thus there is a maximiser $\beta^*\in\cK_*(x)$. 
That is, 
\begin{equation*}
\partial_{t}\varphi^*(t,x)+\frac{\sigma^2}{2}(\beta^*)^2\partial_{xx}\varphi^*(t,x)+c_A(\beta^*)^2\partial_{x}\varphi^*(t,x)+(1-c_A\beta^*)\beta^*= 0,
\end{equation*}
for some $\beta^*(t,x)\in\cK_*(x)$
Now, repeating the same arguments as above with $\alpha^*=1/\beta^*$ we obtain 
\begin{equation*}
\frac{\sigma^2}{2}\partial_{xx}\varphi^*(t,x)+c_A\partial_{x}\varphi^*(t,x)+(\alpha^*)^2\partial_{t}\varphi^*(t,x)+\alpha^*-c_A= 0.
\end{equation*}
In conclusion, we have
\begin{equation}\label{eq:strong_sol_A}
\frac{\sigma^2}{2}\partial_{xx}\varphi^*(t,x)+c_A\partial_{x}\varphi^*(t,x)+\sup_{\alpha\in\cK'_*(x)}\big[\alpha^2\partial_{t}\varphi^*(t,x)+\alpha\big]-c_A= 0.
\end{equation}

The advantage of working with the above reformulation of the HJB is that now we have better control over the maximiser $\alpha^*$. Indeed, by first order conditions and recalling that $\partial_t\varphi^*\le 0$ we have the unique maximiser
\begin{equation*}
\alpha^*(t,x)=\min\Big\{\max\Big\{c_A,-\frac{1}{2\partial_t\varphi^*(t,x)}\Big\},\frac{1}{2K_*(x)}\Big\}=-\frac{1}{2\partial_t\varphi^*(t,x)},
\end{equation*} 
where the second equality holds because $c_A\le -(2\partial_t\varphi^*(t,x))^{-1}\le 1/(2K_*(x))$
by Proposition \ref{prop:d_time} and Remark \ref{rem:opt_cont}. Therefore, in \eqref{eq:strong_sol_A} there is no difference between maximising over $\cK'_*(x)$ or over $\R$. Moreover, plugging $\alpha^*$ in the equation and rearranging terms we deduce 
\begin{equation}\label{eq:PDE_supremum}
\frac{\sigma^2}{2}\partial_{xx}\varphi^*(t,x)+c_A\partial_{x}\varphi^*(t,x)-c_A=\frac{1}{4\partial_{t}\varphi^*(t,x)}.
\end{equation}
Since $\alpha^*$ is the unique maximiser in \eqref{eq:strong_sol_A}, then $\beta^*=1/\alpha^*$ used in the above construction is also the unique maximiser of \eqref{eq:PDE_q-orig} and moreover using \eqref{eq:PDE_supremum} we obtain
\begin{equation}\label{eq:beta_opt}
\beta^*(t,x)=-2\partial_t\varphi^*(t,x)=-\frac{1}{2[\frac{\sigma^2}{2}\partial_{xx}\varphi^*(t,x)+c_A\partial_{x}\varphi^*(t,x)-c_A]},
\end{equation} 
with $2K_*(x)\le \beta^*(t,x)\le 1/c_A$. Then, the maximiser in the Hamiltonian \eqref{eq:HK} lies inside the set $\cK_*(x)$ and there is no difference between maximising over $\cK_*(x)$ or over $\R$. This shows that the first equation in \eqref{eq:PDE_original} holds at $(t,x)$ and, by arbitrariness of $(t,x)$, it holds almost everywhere. The boundary conditions hold directly by continuity of $\varphi^*$.
\end{proof}
From the argument of proof of the above theorem we obtain a useful corollary. Let us recall the Hamiltonian $\cM_K$ from \eqref{eq:hamil} and let us introduce the unconstrained analogue 
\begin{equation*}
\cM(p)\coloneqq \sup_{\alpha\in\R}\big(\alpha^2p+\alpha\big).
\end{equation*}
Combining \eqref{eq:strong_sol_A} and \eqref{eq:PDE_supremum} we deduce the next result.

\begin{corollary}\label{cor:PDE_aux}
The function $\varphi^*\in W^{1,2;\infty}_{\ell oc}([0,T]\times[0,\infty))\cap C([0,T]\times[0,\infty))$ satisfies the two equivalent PDEs:
\begin{equation}\label{eq:PDE_aux}
\begin{aligned}
\tfrac{1}{2}\sigma^2\partial_{xx}\varphi^*+c_A\partial_{x}\varphi^*+\cM(\partial_t\varphi^*)&=c_A,\quad \text{a.e.\ on }[0,T)\times(0,\infty),\\
-\frac{1}{4\partial_{t}\varphi^*}+\tfrac{1}{2}\sigma^2\partial_{xx}\varphi^*+c_A\partial_{x}\varphi^*&=c_A,\quad \text{a.e.\ on }[0,T)\times(0,\infty),
\end{aligned}
\end{equation}
with boundary conditions $\varphi^*(t,0)=0$ for $t\in[0,T)$ and $\varphi^*(T,x)=0$ for $x\in[0,\infty)$.
\end{corollary}

In the next theorem we lift the regularity of $\varphi^*$ to $C^{\infty}((0,T)\times(0,\infty))$. For that we need a bound provided by\footnote{Similar result on the regularity of partial differential equation in divergence form were proven originally by De Giorgi (elliptic case, \cite{degiorgi1956diff} and \cite{degiorgi1957diff}) and Nash (parabolic case, \cite{nash1958continuity}).} \cite[Lem.\ 4.1]{krylov1981properties} which we state in Lemma \ref{lem:krylov} for completeness. For the use of the lemma we introduce the set 
\begin{equation*}
C_{r}(s,y)\coloneqq\{(t,x)\in\R^2:s< t< s+r^2, |x-y|<r\},
\end{equation*} 
for $(s,y)\in \R^2$ and $r>0$. 
Notice that while \cite{krylov1981properties} considers forward parabolic PDEs, here we deal with a backward parabolic PDE and so our $C_{r}(s,y)$ is adjusted to reflect such difference. We also introduce suitable H\"older spaces with respect to the parabolic distance (cf.\ \cite[Sec.\ 2, Ch.\ 3]{friedman2008partial}) that are needed in the proof of the theorem. For $\delta\in(0,1)$ and $\cD\subset \R^2$, we denote $C^\delta(\cD)$ the set of $\delta$-H\"older continuous functions $f:\cD\to\R$. For $j,k\in\N$, we denote $C^{j,k;\delta}(\cD)$ the set of $\delta$-H\"older continuous functions $f:\cD\to\R$ such that partial derivatives with respect to time, up to the $j$-th order, and partial derivatives with respect to space, up to the $k$-th order, are $\delta$-H\"older continuous on $\cD$. Finally, we use $C^{j,k;\delta}_{\ell oc}(\cD)$ for the class of functions $f\in C^{j,k;\delta}(D)$ for any compact subset $D$ of $\cD$. These spaces equipped with H\"older norms are Banach spaces (see details in \cite[Sec.\ 2, Ch.\ 3]{friedman2008partial}).

\begin{theorem}\label{thm:smooth}
We have $\varphi^*\in C^{\infty}((0,T)\times(0,\infty))$.
\end{theorem}

\begin{proof} 
The proof is divided into three steps.

{\em Step 1.} First we show that $\partial_t\varphi^*\in C((0,T)\times(0,\infty))$. We evaluate \eqref{eq:PDE_aux} at $(t,x)$ and $(t+h,x)$, for small $h>0$, $h\in\bQ$. That yields for a.e.\ $(t,x)$,
\begin{equation*}
\begin{aligned}
&-\frac{1}{4\partial_{t}\varphi^*(t+h,x)}+\frac{\sigma^2}{2}\partial_{xx}\varphi^*(t+h,x)+c_A\partial_{x}\varphi^*(t+h,x)-c_A\\
&=-\frac{1}{4\partial_{t}\varphi^*(t,x)}+\frac{\sigma^2}{2}\partial_{xx}\varphi^*(t,x)+c_A\partial_{x}\varphi^*(t,x)-c_A,
\end{aligned}
\end{equation*}
where the null set can be taken independent of $h\in\bQ$ (for each $h\in\bQ$ we have a null set $\cN_h$ but then we consider the equation on $([0,T)\times(0,\infty))\setminus\cN$ with $\cN=\cup_{h\in\bQ}\cN_h$).
Rearranging terms, we get
\begin{equation*}
\begin{aligned}
&\frac{1}{4}\Big(\frac{\partial_{t}\varphi^*(t+h,x)-\partial_{t}\varphi^*(t,x)}{\partial_{t}\varphi^*(t+h,x)\partial_{t}\varphi^*(t,x)}\Big)+\frac{\sigma^2}{2}\Big(\partial_{xx}\varphi^*(t+h,x)-\partial_{xx}\varphi^*(t,x)\Big)\\
&+c_A\Big(\partial_{x}\varphi^*(t+h,x)-\partial_{x}\varphi^*(t,x)\Big)=0.
\end{aligned}
\end{equation*}
Dividing by $h$ and setting $v^h(t,x)\coloneqq(\varphi^*(t+h,x)-\varphi^*(t,x))h^{-1}$ yields
for a.e.\ $(t,x)\in[0,T)\times(0,\infty)$,
\begin{equation}\label{eq:PDE_vt}
\partial_{t}v^h(t,x)+a^h(t,x)\partial_{xx}v^h(t,x)+b^h(t,x)\partial_{x}v^h(t,x)=0,
\end{equation}
where
\begin{equation}\label{eq:linear_coeff_h}
\begin{aligned}
a^h(t,x)\coloneqq 2\sigma^2\big(\partial_{t}\varphi^*(t+h,x)\partial_{t}\varphi^*(t,x)\big),\quad b^h(t,x)\coloneqq 4c_A\big(\partial_{t}\varphi^*(t+h,x)\partial_{t}\varphi^*(t,x)\big).
\end{aligned}
\end{equation}

By Proposition \ref{prop:d_time} we have for a.e.\ $(t,x)\in [0,T]\times[0,\infty)$,
\begin{equation}\label{eq:coeff_PDE_vt}
\begin{aligned}
0\le a^h(t,x)\le \frac{\sigma^2}{8c_A^2}\quad\text{and}\quad 0\le b^h(t,x)\le \frac{1}{4c_A}.
\end{aligned}
\end{equation}
We fix $(t_0,x_0)\in(0,T)\times(0,\infty)$ and $r>0$ such that $x_0-2r>0$ and $t_0+h+4r^2<T$. 
Since the function $v^h\in W^{1,2;p}_{\ell oc}(C_{2r}(t_0,x_0))$ solves \eqref{eq:PDE_vt} and $a^h$ is strictly separate from zero in $C_{2r}(t_0,x_0)$, then by Lemma \ref{lem:krylov}, there are $N,\delta>0$, depending on the bounds in \eqref{eq:coeff_PDE_vt}, such that
\begin{equation}\label{eq:equic_vh}
\begin{aligned}
|v^h(t,x)-v^h(s,y)|&\le N r^{-\delta}\big(|x-y|+|t-s|^{\frac{1}{2}}\big)^\delta \sup_{(t,x)\in C_{2r}(t_0,x_0)}|v^h(t,x)|\\
&\le \frac{N}{4c_A} r^{-\delta}\big(|x-y|+|t-s|^{\frac{1}{2}}\big)^\delta, 
\end{aligned}
\end{equation}
for all $(t,x),(s,y)\in C_{r}(t_0,x_0)$, where for the final inequality we used
\begin{equation}\label{eq:equib_vh}
\sup_{(t,x)\in C_{2r}(t_0,x_0)}|v^h(t,x)|\le \tfrac{1}{4c_A}.
\end{equation}

The family $(v^h)_{h\in(0,h_0)}$ is bounded in the $C^\delta(C_{r}(t_0,x_0))$-norm. Thus, there are a function $f\in C^\delta(C_{r}(t_0,x_0))$ and a sequence $(v^{h_j})_{j\in\N}$, with $h_j\to 0$ as $j\to\infty$, such that for any $\delta'\in(0,\delta)$ we have $v^{h_j}\to f$ in $C^{\delta'}(C_{r}(t_0,x_0))$.
Noticing that $v^h \to \partial_t \varphi^*$ as $h\downarrow 0$ for almost every point in $C_{r}(t_0,x_0)$, we have that $\partial_t \varphi^*=f$ a.e.\ in $C_{r}(t_0,x_0)$. That is, the function $\partial_t \varphi^*$ admits a $\delta$-H\"older continuous modification on $C_{r}(t_0,x_0)$. The construction can be extended to $[0,T)\times(0,\infty)$ as follows. Let us consider a countable covering $\{C_{r_n}(t_n,x_n),\,n\in\N\}$ of $(0,T)\times(0,\infty)$, where $r_n>0$ is such that $t_n+4r_n^2<T$ and $x_n-2r_n>0$. For each $n\in\N$, take $h_n>0$, $h_n\in\bQ$ such that $t_n+h_n+4r_n^2<T$. Then, the family $(v^h)_{h\in(0,h_n)}$ satisfies \eqref{eq:equic_vh} and \eqref{eq:equib_vh} on $C_{r_n}(t_n,x_n)$ with suitable constants $N_n,\delta_n>0$. Passing to the limit along a sequence $(h^n_k)_{k\in\N}$, with $h^n_k\to 0$ as $k\to\infty$, the sequence $(v^{h^n_k})_{k\in\N}$ converges to a function $f_n\in C^{\delta_n}(C_{r_n}(t_n,x_n))$ in the norm of $C^{\delta'_n}(C_{r_n}(t_n,x_n))$, $\delta'_n\in(0,\delta_n)$. As above it must be $f_n=\partial_t\varphi^*$ a.e.\ and so we conclude that $\partial_t\varphi^*\in C^{\delta_n}(C_{r_n}(t_n,x_n))$ on every set $C_{r_n}(t_n,x_n)$. Thus, $\partial_t\varphi^*\in C((0,T)\times(0,\infty))$ because $\cup_{n\in\N}C_{r_n}(t_n,x_n)=(0,T)\times(0,\infty)$ and the covering was arbitrary. Moreover, it is clear from the construction that on every compact subset $U\subset (0,T)\times(0,\infty)$ the constants $N,\delta>0$ only depend on the compact $U$. Hence, $\partial_t\varphi^*\in C^\delta_{\ell oc}((0,T)\times(0,\infty))$. 
\smallskip

{\em Step 2.} Here we show that $\partial_t\varphi^*\in C^{\infty}([0,T)\times(0,\infty))$. 
On any compact rectangle $U\subset(0,T)\times(0,\infty)$, 
for some $h_U>0$ and $\delta=\delta(U)>0$ 
the coefficients $a^h$, $b^h$ in \eqref{eq:linear_coeff_h} belong to $C^\delta(U)$
uniformly for $h\in(0,h_U)$. By construction, the function $v^h$ solves the Cauchy--Dirichlet problem
\begin{equation*}
\begin{cases}\partial_{t}\psi(t,x)+a^h(t,x)\partial_{xx}\psi(t,x)+b^h(t,x)\partial_{x}\psi(t,x)=0,& \text{for }(t,x)\in \mathrm{int}(U),\\
\psi(t,x)=v^h(t,x),& \text{for }(t,x)\in \partial_P U,
\end{cases}
\end{equation*}
where we denote by $\mathrm{int}(U)$ and $\partial_P U$ the interior and the parabolic boundary of $U$, respectively. Since \cite[Thm.\ 9, Ch.\ 3, Sec.\ 4]{friedman2008partial} guarantees existence and uniqueness of a classical solution of the above PDE, then $v^h\in C^{1,2;\delta}_{\ell oc}(\mathrm{int}(U))$. Further, \cite[Thm.\ 5, Ch.\ 3, Sec.\ 4]{friedman2008partial} guarantees that 
\begin{equation}\label{eq:vhholder}
\big\|v^h\big\|_{C^{1,2;\delta}(V)}\le c_V
\end{equation}
where $V\subset \mathrm{int}(U)$ is another compact and the constant $c_V>0$ depends on the H\"older norms of $a^h$ and $b^h$ on $V$, which we know are uniformly bounded for $h\in(0,h_U)$. The bound \eqref{eq:vhholder} implies that $(v^h)_{h\in(0,h_U)}$ is compact in $C^{1,2;\delta'}(V)$, for any $\delta'\in(0,\delta)$. Then $v^{h_n}\to \partial_t\varphi^*$ in $C^{1,2;\delta'}(V)$ along a sequence $h_n\to 0$. Moreover, $\partial_t\varphi^*\in C^{1,2;\delta}(V)$. Since $U$ and $V$ are arbitrary, we conclude that $\partial_t\varphi^*\in C^{1,2}((0,T)\times(0,\infty))$ and it solves 
\begin{equation*}
\partial_{t}\psi(t,x)+a(t,x)\partial_{xx}\psi(t,x)+b(t,x)\partial_{x}\psi(t,x)=0,\quad (t,x)\in(0,T)\times(0,\infty),
\end{equation*}
where 
\begin{equation}\label{eq:linear_coeff}
a(t,x)\coloneqq 2\sigma^2(\partial_{t}\varphi^*(t,x))^2\quad\text{and}\quad b(t,x)\coloneqq 4c_A(\partial_{t}\varphi^*(t,x))^2.
\end{equation}

In particular, for any compact rectangle $U$, the function $\partial_t\varphi^*$ and coefficients $a$ and $b$ belong to $C^{1,2;\delta}(U)$ for some $\delta=\delta(U)>0$ and $\partial_t\varphi^*$ solves the boundary value problem 
\begin{equation}\label{eq:PDE_aux_lin2}
\begin{cases}
\partial_{t}\psi(t,x)+a(t,x)\partial_{xx}\psi(t,x)+b(t,x)\partial_{x}\psi(t,x)=0,& \text{for }(t,x)\in \mathrm{int}(U),\\
\psi(t,x)=\partial_t\varphi^*(t,x),& \text{for }(t,x)\in \partial_P U.
\end{cases}
\end{equation}
Then, by interior regularity results for parabolic PDEs\footnote{Notice that actually all derivatives admit $\delta$-H\"older continuous extension to $U$ by Weierstrass theorem.}, for any compact $V\subset U$ we have $\partial_t\varphi^*\in C^{2,4;\delta}(V)$ with $\partial_{ttx}\varphi^*$ and $\partial_{ttxx}\varphi^*$ also in $C^\delta(V)$ (cf.\ \cite[Thm.\ 11, Ch.\ 3, Sec.\ 5]{friedman2008partial}). In particular, that lifts the regularity of the coefficients $a$ and $b$ in \eqref{eq:PDE_aux_lin2} to $C^{2,4;\delta}(V)$ with $\partial_{tx}a$, $\partial_{tx}b$, $\partial_{txx}a$ and $\partial_{txx}b$ in $C^\delta(V)$. This additional regularity allows another application of \cite[Thm.\ 11, Ch.\ 3, Sec.\ 5]{friedman2008partial}. Continuing by induction, by arbitrariness of $V$, we deduce $\partial_t\varphi^*\in C^\infty(\mathrm{int}(U))$ and since $U$ was arbitrary we conclude $\partial_t\varphi^*\in C^\infty((0,T)\times(0,\infty))$. Next we transfer that smoothness to $\varphi^*$ itself.
\smallskip

{\em Step 3.} First we show that $\varphi^*$ is classical solution of the second equation in \eqref{eq:PDE_aux} (hence also of the first one). We will subsequently deduce its $C^\infty$-regularity.

The second equation in \eqref{eq:PDE_aux} reads also as
\begin{equation}\label{eq:PDE_time}
\frac{\sigma^2}{2}\partial_{xx}\varphi^*=\frac{1}{4\partial_{t}\varphi^*}-c_A\partial_{x}\varphi^*+c_A,\quad \text{a.e.\ on }[0,T)\times(0,\infty).
\end{equation}
Since $\varphi^*\in W^{1,2;\infty}_{\ell oc}([0,T)\times(0,\infty))$ thanks to Proposition \ref{prop:W12p}, then $\partial_x \varphi^*$ is $\gamma$-H\"older continuous for any $\gamma\in(0,1)$ by the embedding $W^{1,2;p}_{\ell oc}([0,T)\times(0,\infty))\hookrightarrow C^{0,1;\gamma}_{\ell oc}([0,T)\times(0,\infty))$ with $\gamma=1-\frac{3}{p}$ and $p>3$ (see, e.g., \cite[eq.\ (E.9)]{fleming2012deterministic} or \cite[exercise 10.1.14]{krylov2008lectures}). Moreover, we proved in the previous step that $\partial_t\varphi^*\in C^\infty((0,T)\times(0,\infty))$ and we know that $\partial_{t}\varphi^*<0$ on $[0,T)\times(0,\infty)$, by Proposition \ref{prop:d_time}. Then, on any compact the terms on the right-hand side of \eqref{eq:PDE_time} are H\"older-continuous and they are continuous on $(0,T)\times(0,\infty)$. Therefore $\partial_{xx}\varphi^*$ is $\delta$-H\"older-continuous on any compact (with $\delta\in(0,1)$ depending on the compact) and it is continuous on $(0,T)\times(0,\infty)$. Hence, $\varphi^*\in C^{1,2}((0,T)\times(0,\infty))\cap C([0,T]\times[0,\infty))$ is a classical solution of \eqref{eq:PDE_aux}. 

For any $S\in[0,T)$ and $0<x_1<x_2<\infty$, the function $\psi=\varphi^*$ solves the (forward) Cauchy--Dirichlet problem obtained multiplying the second equation in \eqref{eq:PDE_aux} by $4(\partial_{t}\varphi^*)^2$. That is, 
\begin{equation}\label{eq:PDE_aux_lin}
\begin{cases}-\partial_{t}\psi(t,x)+a(t,x)\partial_{xx}\psi(t,x)+b(t,x)\partial_{x}\psi(t,x)=b(t,x),& \text{for }(t,x)\in(0,S)\times(x_1,x_2),\\
\psi(t,x_2)=\varphi^*(t,x_2),& \forall t\in(0,S],\\
\psi(t,x_1)=\varphi^*(t,x_1),& \forall t\in(0,S],\\
\psi(0,x)=\varphi^*(0,x),& \forall x\in[x_1,x_2],
\end{cases}
\end{equation}
where $a$ and $b$ are defined in \eqref{eq:linear_coeff}. From step 2 the coefficients $a$ and $b$ are in $C^\infty$ and \cite[Thm.\ 9, Ch.\ 3, Sec.\ 4]{friedman2008partial} says that \eqref{eq:PDE_aux_lin} admits a unique classical solution. Thus $\varphi^*\in C^{1,2;\delta}((0,S)\times(x_1,x_2))$
for some $\delta>0$ depending on the compact $[0,S]\times[x_1,x_2]$. Since the $a$ and $b$ are infinitely smooth, once again we can invoke interior regularity results for parabolic PDEs (cf.\ \cite[Thm.\ 11, Ch.\ 3, Sec.\ 5]{friedman2008partial}) to lift the regularity of the solution of \eqref{eq:PDE_aux_lin} in $(0,S)\times (x_1,x_2)$. Hence, $\varphi^*\in C^\infty((0,S)\times (x_1,x_2))$. Since $(0,S)\times (x_1,x_2)\subset(0,T)\times(0,\infty)$ was arbitrary, this concludes the proof of the theorem.
\end{proof}

\section{Proof of Theorem \ref{thm:main}}\label{sec:proofthm}

The final theorem in the previous section implies that $\varphi^*$ is classical solution of the HJB equation from Theorem \ref{thm:HJB}. Combining that with an application of It\^o's calculus we now show that $\varphi^*$ coincides with the value function $\bar \varphi$ of the Principal's problem in \eqref{eq:orig_value}. That also shows uniqueness of the solution to the HJB equation. Moreover, we obtain the optimal control for the Principal in Markovian form. All these facts and additional regularity in the space variable for $\varphi^*=\bar \varphi$ form the statement of our main theorem (Theorem \ref{thm:main}), the proof of which is the main content of this section.

\begin{proof}[{\bf Proof of Theorem \ref{thm:main}}]
The main task is to show the identity $\varphi^*=\bar\varphi$. Once that is accomplished, we immediately deduce that 
$\bar\varphi\in C([0,T]\times[0,\infty))\cap C^\infty((0,T)\times(0,\infty))$
is the unique bounded solution of the HJB equation and it is concave and decreasing in time, with $\partial_t \bar\varphi$ bounded (Propositions \ref{prop:phi_bnd}--(i) and \ref{prop:phi^A_conc_A}). Monotonicity and concavity in $x$ will be obtained later on. Notice that uniqueness holds because the same verification theorem performed for $\varphi^*$ in the next paragraphs would be applicable to any bounded classical solution of the HJB equation, thus yielding the equivalence with $\bar \varphi$.

Fix $\beta\in \bar \cB$ from Definition \ref{def:B} and recall the probability measure $\P_\beta$ from \eqref{eq:Pbeta}, under which the process
$B^\beta_t=\sigma^{-1}(Y_t-Y_0-\int_0^t\beta_s\ud s)$ is a Brownian motion. The contract's dynamics associated to $\beta$ reads as (cf.\ \eqref{eq:SDEYW})
\begin{equation*}
\bar W_s^{x;\beta}=x+\int_0^s c_A\beta_u^2\,\ud u+\int_0^s\sigma\beta_u \,\ud B^\beta_u,\quad\text{with $x> 0$}.
\end{equation*}
For all $n,k\in\N$, we set 
\begin{equation}\label{eq:taunsigmak}
\tau_\beta^n\coloneqq\inf\{s\ge 0: \bar W_s^{x;\beta}< n^{-1} \}\quad\text{and}\quad \sigma^k_\beta\coloneqq\inf\{s\ge 0: \bar W_s^{x;\beta}> k \}.
\end{equation}
It is clear that $\varphi^*(T,x)=0=\bar\varphi(T,x)$ and $\varphi^*(t,0)=0=\bar\varphi(t,0)$ for all $x\in[0,\infty)$ and all $t\in[0,T]$, respectively. We now show $\varphi^*(t,x)=\bar\varphi(t,x)$ for arbitrary $(t,x)\in(0,T)\times(0,\infty)$. 
Recalling that $\varphi^*\in C^{\infty}((0,T)\times (0,\infty))$ by Theorem \ref{thm:smooth}, we can apply Dynkin's formula to $\varphi^*\big(t+u,\bar W_u^{x;\beta}\big)$ on the random interval $u\in[\![0,s\wedge \tau_\beta^n\wedge\sigma_\beta^k]\!]$ with $s\in[0,T-t)$. That yields
\begin{equation}\label{eq:verif_0}
\begin{aligned}
\varphi^*(t,x)=\E_\beta\Big[&\int_0^{s\wedge \tau_\beta^n\wedge\sigma_\beta^k}\Big(-\partial_t\varphi^*-\frac{\sigma^2}{2}\beta^2_u\partial_{xx}\varphi^*-c_A\beta^2_u\partial_x\varphi^*\Big)\big(t+u,\bar W^{x;\beta}_u\big)\ud u\\
&+ \varphi^*\Big(t+s\wedge\tau_\beta^n\wedge\sigma_\beta^k,\bar W^{x;\beta}_{s\wedge \tau_\beta^n\wedge\sigma_\beta^k}\Big)\Big].
\end{aligned}
\end{equation}
Adding and subtracting $(1-c_A\beta_u)\beta_u$ in the first term under expectation we get
\begin{equation*}
\begin{aligned}
&\E_\beta\Big[\int_0^{s\wedge \tau_\beta^n\wedge\sigma_\beta^k}\Big(-\partial_t\varphi^*-\frac{\sigma^2}{2}\beta^2_u\partial_{xx}\varphi^*-c_A\beta^2_u\partial_x\varphi^*\Big)\big(t+u,\bar W^{x;\beta}_u\big)\ud u\Big]\\
&\ge \E_\beta\Big[\int_0^{s\wedge \tau_\beta^n\wedge\sigma_\beta^k}\Big(-\partial_t\varphi^*-\cH(\,\cdot\,;\varphi^*)\Big)\big(t+u,\bar W^{x;\beta}_u\big)\ud u+\int_0^{s\wedge \tau_\beta^n\wedge\sigma_\beta^k}(1-c_A\beta_u)\beta_u\ud u\Big]\\
&=\E_\beta\Big[\int_0^{s\wedge \tau_\beta^n\wedge\sigma_\beta^k}(1-c_A\beta_u)\beta_u\ud u\Big],
\end{aligned}
\end{equation*}
where the inequality holds by definition of the Hamiltonian \eqref{eq:hamil} and the equality holds by Theorem \ref{thm:HJB}. Plugging the above inequality into \eqref{eq:verif_0}, letting $s\uparrow T-t$ and using continuity of $\varphi^*$ and dominated convergence we obtain 
\begin{equation}\label{eq:verif_2}
\begin{aligned}
\varphi^*(t,x)&\ge \E_\beta\Big[\int_0^{(T-t)\wedge \tau_\beta^n\wedge\sigma_\beta^k}(1-c_A\beta_s)\beta_s\ud s+\varphi^*\Big(T\wedge(t\!+\! \tau_\beta^n\wedge\sigma_\beta^k),\bar W^{x;\beta}_{(T-t)\wedge \tau_\beta^n\wedge\sigma_\beta^k}\Big)\Big]\\
&\ge \E_\beta\Big[\int_0^{(T-t)\wedge \tau_\beta^n\wedge\sigma_\beta^k}(1-c_A\beta_s)\beta_s\ud s\Big],
\end{aligned}
\end{equation}
where in the final inequality we used $\varphi^*\ge 0$. Since $\beta\in\bar{\cB}$, we have 
\begin{equation*}
\begin{aligned}
\E_\beta\Big[\sup_{0\le s\le T-t}\big|\bar W_s^{x;\beta}\big|\Big]
&\le |x|+c_A\E_\beta\Big[\int_0^{T-t}\beta_u^2\ud u\Big]+\sigma\E_\beta\Big[\sup_{0\le s\le T-t}\Big|\int_0^s\beta_u\ud B^\beta_u\Big|^2\Big]^{1/2}\\
&\le |x|+c_A\E_\beta\Big[\int_0^{T-t}\beta_u^2\ud u\Big]+4\sigma\E_\beta\Big[\int_0^{T-t}\beta_u^2\ud u\Big]^{1/2}<\infty
\end{aligned}
\end{equation*}
where we used Jensen's inequality and then Doob's maximal inequality to bound the stochastic integral. This bounds implies that $\sigma_\beta^k\uparrow \infty$ as $k\to\infty$. By continuity of paths of $\bar W^{x;\beta}$ we also have $\tau_\beta^n\uparrow \bar\tau_\beta$ as $n\to\infty$. Therefore, by dominated convergence theorem, taking limits in \eqref{eq:verif_2} we obtain
\begin{equation*}
\varphi^*(t,x)\ge\E_\beta\Big[\int_0^{(T-t)\wedge\bar\tau_\beta}(1-c_A\beta_s)\beta_s\ud s\Big],
\end{equation*}
and by the arbitrariness of $\beta\in\bar \cB$ we deduce $\varphi^*(t,x)\ge \bar \varphi(t,x)$.
Next we show $\varphi^*(t,x)\le \bar{\varphi}(t,x)$.

Let $\beta^*(t,x)\coloneqq -2\partial_t\varphi^*(t,x)$. Since $\beta^*(t,x)$ is locally Lipschitz on $(0,T)\times(0,\infty)$ and bounded (cf.\ Proposition \ref{prop:d_time} and Theorem \ref{thm:smooth}), under $\P_0$ the SDE
\begin{equation*}
\bar W_s^*=x+4(c_A-1)\int_0^s\big(\partial_{t}\varphi^*(t\!+\!u,\bar W_u^{*})\big)^2 \,\ud u-2\int_0^s\partial_{t}\varphi^*(t\!+\!u,\bar W_u^{*})\ud Y_u,
\end{equation*}
admits a unique strong solution (i.e., adapted to the filtration generated by $Y$) on the random time-interval $s\in[\![0,\bar \tau_*\wedge(T-t)]\!]$ where $\bar \tau_*\coloneqq\inf\{s\ge 0: \bar W^*_s\le 0\}$ with $\inf\varnothing=+\infty$.

The process $\beta^*_s$ defined as 
\begin{equation}\label{eq:optimality}
\beta^*_s\coloneqq -2\partial_{t}\varphi^*(t+s,\bar W_s^*)\mathds{1}_{\{s<\bar \tau_*\}}
\end{equation}
belongs to the class $\bar\cB$ because it is clearly $\bar\F$-progressively measurable, bounded and it vanishes for $s\in[\![\bar\tau_*,\infty)\!)$. 
In particular, we can define a probability measure $\P_*=\P_{\beta^*}$ according to \eqref{eq:Pbeta} and a Brownian motion $B^*=B^{\beta^*}$ by Girsanov theorem. Under $\P_*$ we have, for $s\in[\![0,\bar\tau_*]\!]$ 
\begin{equation*}
\begin{aligned}
\bar W_s^*&=x+4 c_A\int_0^s\big(\partial_{t}\varphi^*(t\!+\!u,\bar W_u^{*})\big)^2 \,\ud u-2\int_0^s\partial_{t}\varphi^*(t\!+\!u,\bar W_u^{*}) \sigma\ud B^*_u\\
&=x+c_A\int_0^s(\beta^*_u)^2 \,\ud u+\int_0^s\beta^*_u\sigma\ud B^*_u,
\end{aligned}
\end{equation*}
and we extend the solution as $\bar W_s^*=0$ for $s\in(\!(\bar\tau_*,\infty)\!)$. Defining the stopping times $\tau^n_*=\tau^n_{\beta^*}$ and $\sigma^k_*=\sigma^k_{\beta^*}$ as in \eqref{eq:taunsigmak}, and denoting $\E_*[\,\cdot\,]=\E_{\beta^*}[\,\cdot\,]$ we obtain again \eqref{eq:verif_0} but with $(\tau^n_\beta,\sigma^k_\beta,\bar W^\beta, \beta,\E_\beta)$ therein replaced by $(\tau^n_*,\sigma^k_*,\bar W^*, \beta^*,\E_*)$. Adding and subtracting $(1-c_A\beta^*_u)\beta^*_u$ inside the integral and recalling that
$\beta^*_s$ is maximiser of the Hamiltonian in \eqref{eq:PDE_original} (see also \eqref{eq:beta_opt}), we get
\begin{equation*}
\begin{aligned}
\varphi^*(t,x)&=\E_*\Big[-\!\!\int_0^{s\wedge \tau_*^n\wedge\sigma_*^k}\!\!\!\big[\partial_t\varphi^*(\cdot)\!+\!\cH(\cdot;\varphi^*)\big]\big(t\!+\!u,\bar W^{*}_u\big)\ud u\\
&\qquad+\!\int_0^{s\wedge \tau_*^n\wedge\sigma_*^k}\!(1\!-\!c_A\beta_u^*)\beta_u^*\ud s\!+\! \varphi^*\Big(t\!+\! s\wedge\tau_*^n\wedge\sigma_*^k,\bar W^{*}_{s\wedge \tau_*^n\wedge\sigma_*^k}\Big)\Big]\\
&=\E_*\Big[\!\int_0^{s\wedge \tau_*^n\wedge\sigma_*^k}\!(1\!-\!c_A\beta_s^*)\beta_s^*\ud s\!+\! \varphi^*\Big(t\!+\! s\wedge\tau_*^n\wedge\sigma_*^k,\bar W^{*}_{s\wedge \tau_*^n\wedge\sigma_*^k}\Big)\Big],
\end{aligned}
\end{equation*}
where we used that $\varphi^*$ is classical solution of \eqref{eq:PDE_original} for the second equality.

Letting $k,n\to \infty$, we can use dominated convergence because $\varphi^*$ is bounded (cf.\ Proposition \ref{cor:contphi}) and $\beta^*_u$ is also bounded, along with the limits $\sigma^k_*\uparrow +\infty$ and $\tau^n_*\uparrow \bar \tau_{*}$. That yields
\begin{equation*}
\begin{aligned}
\quad\varphi^*(t,x)&=\E_*\Big[\!\int_0^{s\wedge \bar \tau_{*}}\!(1\!-\!c_A\beta_u^*)\beta_u^*\ud u\!+\! \varphi^*\big(t\!+\!s\wedge \bar \tau_{*},\bar W^{*}_{s\wedge \bar \tau_{*}}\big)\Big]\\
&=\E_*\Big[\!\int_0^{s\wedge \bar \tau_{*}}\!(1\!-\!c_A\beta_u^*)\beta_u^*\ud u\!+\! \varphi^*\big(t\!+\!s,\bar W^{*}_s\big)\mathds{1}_{\{s\le \bar\tau_{*}\}}\Big],
\end{aligned}
\end{equation*}
where the second equality holds because $\varphi^*(t+\bar\tau_{*},\bar W^*_{\bar\tau_{*}})=\varphi^*(t+\bar\tau_{*},0)=0$ on $\{\bar\tau_{*} < s\}$. Letting $s\uparrow T-t$ we have
\begin{equation*}
\varphi^*\big(t\!+\!s,\bar W^{*}_s\big)\mathds{1}_{\{s\le \bar\tau_{*}\}} \longrightarrow \varphi^*\big(T,\bar W^{*}_{T-t}\big)\mathds{1}_{\{T-t\le \bar\tau_{*}\}}=0,
\end{equation*}
and by dominated convergence we deduce
\begin{equation*}
\varphi^*(t,x)=\E_*\Big[\!\int_0^{s\wedge \bar \tau_{*}}\!(1\!-\!c_A\beta_u^*)\beta_u^*\ud u\Big].
\end{equation*}
Thus, $\varphi^*(t,x)\le \bar\varphi(t,x)$ as needed.

We have obtained that $\varphi^*=\bar \varphi$ and optimality of the control map $\beta^*(t,x)=-2\partial_t\bar \varphi(t,x)$ until the first time the controlled dynamics hits zero and $\beta^*(t,x)=0$ after that (cf.\ \eqref{eq:optimality}). It remains to show that $\bar \varphi$ is non-decreasing and concave in $x$. Given $\beta\in\bar\cB$, it is clear that the stopping time $\bar\tau_\beta=\bar\tau_\beta(x)$ depends on the initial value of the process $\bar W^\beta_0=x$. Moreover, for $x_2>x_1$ we have $\bar\tau_\beta(x_2)>\bar\tau_\beta(x_1)$. Let $\beta^*=\beta^{*;x_1}$ be optimal for $\bar\varphi(t,x_1)$. Then, $\beta_s=\beta^*_s\mathds{1}_{\{s\le \bar\tau_{\beta^*}(x_1)\}}+\frac{1}{2c_A}\mathds{1}_{\{s> \bar\tau_{\beta^*}(x_1)\}}$ defines a progressively measurable process which is an admissible control for $\bar\varphi(t,x_2)$, thanks to Remark \ref{rem:barphi}. By continuity of paths of $\bar W^{x_1;\beta^*}$ and $\bar W^{x_2;\beta}$ it follows $\bar\tau_{\beta^*}(x_1)<\bar\tau_{\beta}(x_2)$. Moreover, it is clear that 
$\P_\beta=\P_{\beta^*}$ on $\bar\cF_{\bar \tau_{\beta^*}}$ by the explicit form of the Radon--Nikodym derivative \eqref{eq:Pbeta}. 
Then,
\begin{equation*}
\begin{aligned}
\bar \varphi(t,x_2)&\ge \E_{\beta}\Big[\int_0^{\bar \tau_\beta(x_2)\wedge(T-t)}(1-c_A\beta_s)\beta_s\ud s\Big]\\
&=\E_{\beta}\Big[\int_0^{\bar \tau_{\beta^*}(x_1)\wedge(T-t)}(1-c_A\beta_s)\beta_s\ud s+\int_{\bar \tau_{\beta^*}(x_1)\wedge(T-t)}^{\bar \tau_{\beta}(x_2)\wedge(T-t)}\frac{1}{4c_A}\ud s\Big]\\
&\ge \E_{\beta}\Big[\int_0^{\bar \tau_{\beta^*}(x_1)\wedge(T-t)}(1-c_A\beta_s)\beta_s\ud s\Big]\\
&=\E_{\beta^*}\Big[\int_0^{\bar \tau_{\beta^*}(x_1)\wedge(T-t)}(1-c_A\beta^*_s)\beta^*_s\ud s\Big]=\bar\varphi(t,x_1),
\end{aligned}
\end{equation*}
where the penultimate equality holds because $\P_\beta=\P_{\beta^*}$ on $\bar\cF_{\bar \tau_{\beta^*}}$. Hence, $x\mapsto \bar \varphi(t,x)$ is non-decreasing. Since $\bar \varphi$ solves the second equation in \eqref{eq:PDE_aux} we deduce
\begin{equation}\label{eq:concphi}
\tfrac{\sigma^2}{2}\partial_{xx}\bar \varphi=c_A+\frac{1}{4\partial_{t}\bar \varphi}-c_A\partial_{x}\bar \varphi\le 0,
\end{equation}
where the inequality holds because $\partial_x\bar\varphi\ge 0$ and $-1/(4\partial_t\bar \varphi)\ge c_A$ (cf.\ \eqref{eq:dt}).
\end{proof}


\section{Properties of the optimal contract}\label{sec:econ}

In this section we analyse structural properties of the optimal contract, or in other words of the optimal control map $\beta^*(t,x)$.
Besides their economic implications, results contained here reveal finer regularity properties of the value function $\bar \varphi$ of the problem. 

First we provide sensitivity of the optimal contract with respect to $(t,x)$.
\begin{proposition}\label{prop:beta}
The following properties hold: 
\begin{itemize}
\item[(i)] $t\mapsto \beta^*(t,x)$ is non-decreasing with, for any $\delta>0$,
\begin{equation}
\lim_{t\to T}\sup_{\delta\le x\le 1/\delta}\big|\beta^*(t,x)-\tfrac{1}{2c_A}\big|=0;
\end{equation}

\item[(ii)] $x\mapsto \beta^*(t,x)$ is non-decreasing and concave and, for any $0<\delta<T$, it holds
\begin{equation*}
\begin{aligned}
&\lim_{x\downarrow 0}\sup_{0\le t\le T-\delta}\big|\beta^*(t,x)\big|=0,\quad \lim_{x\uparrow\infty}\sup_{0\le t\le T-\delta}\big|\tfrac{1}{2c_A}-\beta^*(t,x)\big|=0,\\
&\lim_{x\downarrow 0}\inf_{0\le s\le T-\delta}\partial_{x}\beta^*(s,x)\ge \frac{1}{\sigma^2}\quad\text{and}\quad\limsup_{x\to 0}\sup_{0\le s\le T-\delta}\partial_{xx}\beta^*(s,x)\le -\frac{2c_A}{\sigma^4}.
\end{aligned}
\end{equation*}
\end{itemize}
\end{proposition}
\begin{proof}
For notational simplicity, throughout the proof we use $\varphi$ instead of $\bar \varphi$. The claim in (i) is immediate by concavity of $t\mapsto \varphi(t,x)$ and thanks to the bound in \eqref{eq:dt}. Indeed, $\P(\rho_x>T-t)\to 1$ as $t\to T$, uniformly for $x\in[\delta,1/\delta]$. We continue with the remaining claims.

The second equation in \eqref{eq:PDE_aux} can be written as
\begin{equation*}
\partial_t\varphi\big(\tfrac12\sigma^2\partial_{xx}\varphi+c_A\partial_x\varphi-c_A\big)=\tfrac14,\quad\text{on}\ [0,T)\times(0,\infty).
\end{equation*}
Thanks to smoothness of $\varphi$ (cf.\ Theorem \ref{thm:smooth}) we can differentiate the above equation with respect to time and obtain
\begin{equation*}
\begin{aligned}
0&=\partial_{tt}\varphi\big(\tfrac12\sigma^2\partial_{xx}\varphi+c_A\partial_x\varphi-c_A\big)+\partial_t\varphi\big(\tfrac12\sigma^2\partial_{txx}\varphi+c_A\partial_{tx}\varphi\big)\\
&=\frac{\partial_{tt}\varphi}{4\partial_t\varphi}+\partial_t\varphi\big(\tfrac12\sigma^2\partial_{txx}\varphi+c_A\partial_{tx}\varphi\big)\ge \partial_t\varphi\big(\tfrac12\sigma^2\partial_{txx}\varphi+c_A\partial_{tx}\varphi\big),
\end{aligned}
\end{equation*}
where the inequality holds because $t\mapsto \varphi(t,x)$ is strictly decreasing and concave. Using again that $\partial_t\varphi<0$ we deduce
\begin{equation}\label{eq:be000}
\tfrac12\sigma^2\partial_{txx}\varphi+c_A\partial_{tx}\varphi\ge0,\quad\text{on}\ [0,T)\times(0,\infty).
\end{equation}

Let us now fix $t<T$, $x>0$ and constants $\delta,\eps,M>0$ with $\delta<T-t$ and $M>x$. Integrating \eqref{eq:be000} with respect to $(s,y,z)\in(0,\delta)\times(x,M)\times(0,\eps)$ yields
\begin{equation}\label{eq:be00}
\begin{aligned}
0&\le  \frac{1}{\delta\eps}\int_x^M\!\int_0^\eps\!\int_0^\delta\Big(\tfrac12\sigma^2\partial_{txx}\varphi(t+s,y+z)+c_A\partial_{tx}\varphi(t+s,y+z)\Big)\ud s\,\ud z\,\ud y\\
&=\frac{1}{\delta\eps}\int_x^M\!\int_0^\eps\tfrac12\sigma^2\big[\partial_{xx}\varphi(t+\delta,y+z)-\partial_{xx}\varphi(t,y+z)\big]\ud z\,\ud y\\
&\quad+\frac{1}{\delta\eps}\int_x^M\!\int_0^\eps c_A\big[\partial_{x}\varphi(t+\delta,y+z)-\partial_{x}\varphi(t,y+z)\big]\ud z\,\ud y.
\end{aligned}
\end{equation}
The first integral on the right-hand side above gives
\begin{equation}\label{eq:be0}
\begin{aligned}
&\int_x^M\!\int_0^\eps\big[\partial_{xx}\varphi(t+\delta,y+z)-\partial_{xx}\varphi(t,y+z)\big]\ud z\,\ud y\\
&=\varphi(t+\delta,M+\eps)-\varphi(t,M+\eps)+\varphi(t,M)-\varphi(t+\delta,M)\\
&\quad+\varphi(t+\delta,x)-\varphi(t,x)+\varphi(t,x+\eps)-\varphi(t+\delta,x+\eps).
\end{aligned}
\end{equation}
The second one instead gives
\begin{equation}\label{eq:be1}
\begin{aligned}
&\int_x^M\!\int_0^\eps \big[\partial_{x}\varphi(t+\delta,y+z)-\partial_{x}\varphi(t,y+z)\big]\ud z\,\ud y\\
&=\int_{x+\eps}^{M+\eps}\big(\varphi(t+\delta,y)-\varphi(t,y)\big)\ud y-\int_x^M\big(\varphi(t+\delta,y)-\varphi(t,y)\big)\ud y\\
&=\int_{M}^{M+\eps}\big(\varphi(t+\delta,y)-\varphi(t,y)\big)\ud y-\int_x^{x+\eps}\big(\varphi(t+\delta,y)-\varphi(t,y)\big)\ud y.
\end{aligned}
\end{equation}
We know from Corollary \ref{cor:PDE_aux} that $\varphi(t,x)$ can be represented as the value of the unconstrained stochastic control problem
\begin{equation*}
\varphi(t,x)=\sup_{\alpha}\E\Big[\int_0^{\sigma^\alpha_t\wedge\rho_x}\big(\alpha_s-c_A\big)\ud s\Big],
\end{equation*}
where the supremum is over the class of all $\F$-progressively measurable processes (cf.\ \eqref{eq:vphiK}). Then, using that $\lim_{x\to\infty}\rho_x=+\infty$ it is not hard to verify that 
\begin{equation*}
\lim_{x\to\infty}\varphi(t,x)=\phi(t)=\frac{T-t}{4c_A},
\end{equation*}
where $\phi(t)$ was introduced in \eqref{eq:phi(t)}. Moreover, the limit is increasing and since $t\mapsto \varphi(t,x)$ is continuous, by Dini's theorem we also deduce
\begin{equation}\label{eq:be2}
\lim_{x\to\infty}\sup_{0\le t\le T}\big|\varphi(t,x)-\phi(t)\big|=0.
\end{equation}

In light of the above, for any $\lambda>0$ we can pick large enough $M_\lambda>0$ so that 
\begin{equation*}
\sup_{x\ge M_\lambda}\sup_{0\le t\le T}\big|\varphi(t,x)-\phi(t)\big|\le \lambda.
\end{equation*}
Using this fact in \eqref{eq:be0} and replacing $\varphi(\cdot,M)$ and $\varphi(\cdot,M+\eps)$ with $\phi(\cdot)$ we get 
\begin{equation}\label{eq:be01}
\begin{aligned}
&\int_x^M\!\int_0^\eps\big[\partial_{xx}\varphi(t+\delta,y+z)-\partial_{xx}\varphi(t,y+z)\big]\ud z\,\ud y\\
&\le 4\lambda+\varphi(t+\delta,x)-\varphi(t,x)+\varphi(t,x+\eps)-\varphi(t+\delta,x+\eps).
\end{aligned}
\end{equation}
Similarly, for \eqref{eq:be1} we obtain
\begin{equation}\label{eq:be11}
\begin{aligned}
&\int_x^M\!\int_0^\eps \big[\partial_{x}\varphi(t+\delta,y+z)-\partial_{x}\varphi(t,y+z)\big]\ud z\,\ud y\\
&=\int_{M}^{M+\eps}\big(\phi(t+\delta)-\phi(t)\big)\ud y+2\lambda\eps-\int_x^{x+\eps}\big(\varphi(t+\delta,y)-\varphi(t,y)\big)\ud y\\
&=\big(\phi(t+\delta)-\phi(t)\big)\eps+2\lambda\eps-\int_x^{x+\eps}\big(\varphi(t+\delta,y)-\varphi(t,y)\big)\ud y.
\end{aligned}
\end{equation}
Recombining \eqref{eq:be01} and \eqref{eq:be11} with \eqref{eq:be00} we obtain the inequality
\begin{equation*}
\begin{aligned}
0&\le \frac{\sigma^2}{\delta\eps}\Big(2\lambda+\tfrac12\big[\varphi(t+\delta,x)-\varphi(t,x)\big]+\tfrac12\big[\varphi(t,x+\eps)-\varphi(t+\delta,x+\eps)\big]\Big)\\
&\quad+\frac{c_A}{\delta\eps}\Big(\big(\phi(t+\delta)-\phi(t)\big)\eps+2\lambda\eps-\int_x^{x+\eps}\big(\varphi(t+\delta,y)-\varphi(t,y)\big)\ud y\Big).
\end{aligned}
\end{equation*}
Since $\lambda$ can be chosen freely, we take $\lambda=p\delta\eps$ for some $p>0$. Then,
\begin{equation*}
\begin{aligned}
0&\le 2 p+\frac{1}{2\eps}\frac{\varphi(t+\delta,x)-\varphi(t,x)}{\delta}+\frac{1}{2\eps}\frac{\varphi(t,x+\eps)-\varphi(t+\delta,x+\eps)}{\delta}\\
&\quad+\frac{c_A}{\sigma^2}\Big(\frac{\phi(t+\delta)-\phi(t)}{\delta}+2p \eps
-\frac{1}{\eps}\int_x^{x+\eps}\frac{\varphi(t+\delta,y)-\varphi(t,y)}{\delta}\ud y\Big).
\end{aligned}
\end{equation*}
Letting $\delta\to 0$ first and then $\eps\to 0$ second we obtain
\begin{equation}\label{eq:be4}
0\le 2 p+2p \frac{c_A}{\sigma^2}-\tfrac12\partial_{tx}\varphi(t,x)-\frac{c_A}{\sigma^2}\big(\partial_t\varphi(t,x)-\phi'(t)\big)\le 2 p+2p \frac{c_A}{\sigma^2}-\tfrac12\partial_{tx}\varphi(t,x),
\end{equation}
where in the second inequality we used $\partial_t\varphi(t,x)\ge -1/(4c_A)=\phi'(t)$. Then, letting also $p\to 0$ we obtain $\partial_{tx}\varphi(t,x)\le 0$. This corresponds to $x\mapsto \beta^*(t,x)$ increasing and plugging back into \eqref{eq:be000} it also implies 
\begin{equation}\label{eq:be3}
\partial_{txx}\varphi(t,x)\ge -\frac{2c_A}{\sigma^2}\partial_{tx}\varphi(t,x)\ge 0,\quad \text{for $(t,x)\in[0,T)\times(0,\infty)$}.
\end{equation}
The latter corresponds to concavity of $x\mapsto \beta^*(t,x)$.

Another consequence of $\partial_{tx}\varphi(t,x)\le 0$ for $(t,x)\in[0,T)\times(0,\infty)$ is that the limits $\partial_t\varphi(t,0+)\coloneqq\lim_{x\downarrow 0}\partial_t\varphi(t,x)$ and $\partial_t\varphi(t,\infty)\coloneqq\lim_{x\uparrow \infty}\partial_t\varphi(t,x)$ are well-defined. Now we calculate such limits. For a test function $\psi\in C^1_c((0,T))$, $\psi\ge 0$ we have
\begin{equation*}
\begin{aligned}
\int_0^T\partial_t\varphi(t,0+)\psi(t)\ud t&=\lim_{x\downarrow 0}\int_0^T\partial_t\varphi(t,x)\psi(t)\ud t\\
&=-\lim_{x\downarrow 0}\int_0^T\varphi(t,x)\psi'(t)\ud t=-\int_0^T\varphi(t,0)\psi'(t)\ud t=0,
\end{aligned}
\end{equation*}
where the first equality is justified by dominated convergence, because $\partial_t\varphi$ is bounded, the second one is integration by parts and the third one is by dominated convergence again. The final equality holds because $\varphi(t,0)=0$. Then it must be $\partial_t\varphi(t,0+)=0$ for a.e.\ $t\in[0,T]$. Since $t\mapsto\partial_t\varphi(t,0+)$ is increasing limit of continuous, non-increasing functions, then it is lower semi-continuous and non-increasing. Hence, $t\mapsto\partial_t\varphi(t,0+)$ is right-continuous and it must be $\partial_t\varphi(t,0+)=0$ for all $t\in[0,T)$. Now, continuity of the limit $\partial_t\varphi(t,0+)$ allows us to use Dini's theorem to also conclude that, for any $0<\delta<T$, 
\begin{equation}\label{eq:be5}
\lim_{x\downarrow 0}\sup_{0\le t\le T-\delta}\big|\partial_t\varphi(t,x)\big|=0.
\end{equation}
By analogous arguments with test functions and using \eqref{eq:be2} we also obtain
\begin{equation*}
\lim_{x\uparrow \infty}\sup_{0\le t\le T-\delta}\big|\partial_t\varphi(t,x)-\phi'(t)\big|=0.
\end{equation*}
Thus, the first two limits in (ii) are proven.

The first inequality in \eqref{eq:be4} (with $p \to 0$) and \eqref{eq:be5} give
\begin{equation*}
\lim_{x\downarrow 0}\sup_{0\le t\le T-\delta}\partial_{tx}\varphi(t,x)\le -\frac{1}{2\sigma^2},
\end{equation*}
where the limit exists thanks to \eqref{eq:be3}, which implies $x\mapsto \partial_{tx}\varphi(t,x)$ non-decreasing. The above equation corresponds to the third limit in (ii). Finally, using the above limit into \eqref{eq:be3} we also obtain 
\begin{equation*}
\liminf_{x\to 0}\partial_{txx}\varphi(t,x)\ge \frac{c_A}{\sigma^4},
\end{equation*}
which corresponds to the fourth limit in (ii).
\end{proof}

There is a characterisation of the optimal effort $\beta^*$ as solution to a Cauchy--Dirichlet problem. This is shown in the next proposition. In a subsequent proposition we obtain a refinement that yields a probabilistic interpretation of the optimal effort.
\begin{proposition}\label{prop:betaP}
Recall that $\beta^*(t,x)=-2\partial_t\bar \varphi(t,x)$ so that $\beta^*\in C^\infty([0,T)\times(0,\infty))$ and $0\le \beta^*(t,x)\le 1/(2c_A)$.
For every $(t,x)\in[0,T)\times[0,\infty)$ we have 
\begin{equation}\label{eq:betabounds}
\frac{1}{2c_A}\P_*\big(\bar\tau_*^{t,x} > T -t\big)\le \beta^*(t,x)\le \frac{1}{2c_A}\P_*\big(\bar\tau_*^{t,x}\ge T -t\big).
\end{equation}

Moreover, $\beta^*\in C([0,T]\!\times\![0,\infty)\!\setminus\!\{(T,0)\})$ (i.e., there is a discontinuity in $(T,0)$) and it solves for $(t,x)\in[0,T)\times(0,\infty)$
\begin{equation}\label{eq:PDEbeta}
\partial_t\beta^*(t,x)+c_A\big[\beta^*(t,x)\big]^2\partial_x\beta^*(t,x)+\tfrac12\sigma^2\big[\beta^*(t,x)\big]^2\partial_{xx}\beta^*(t,x)=0,
\end{equation}
with boundary conditions $\beta^*(t,0)=0$ for $t\in[0,T)$, and $\beta^*(T,x)=1/(2c_A)$ for $x\in(0,\infty)$.
\end{proposition}
\begin{proof}
We first show a lower bound on $\beta^*$ from the probabilistic representation of $\bar\varphi$. Take $(t,x)\in[0,T)\times(0,\infty)$ and $\eps>0$. For the purpose of this proof, with no loss of generality we extend the function $\beta^*(t,x)$ to be a constant for $s > T$, say $\beta^*(s,x)=1/2c_A$ for $(s,x)\in(T,\infty)\times[0,\infty)$. Recalling the notations from \eqref{eq:W*tx}, let $\beta^*_s\coloneqq\beta^*(t+s,\bar W^{*;t,x}_{s})\mathds{1}_{\{s<\bar\tau_*^{t,x}\}}$, $s\in[0,\infty)$, be the optimal control for $\bar\varphi(t,x)$. Having extended the function $\beta^*(t,x)$ allows us to consider $(\beta^*_s,\bar W^{*;t,x}_s)$ as a process for $s\in[0,\infty)$ instead of $s\in[0,T-t]$. 
Because $\bar W^{*;t,x}$ is adapted to the filtration $\bar \F$ generated by the output process $(Y_s)_{s\ge 0}$, the process $(\beta^*_s)_{s\in[0,\infty)}$ is also an admissible control for the problem with value $\bar \varphi(t-\eps,x)$. It follows that 
\begin{equation}\label{eq:varphi-eps}
\begin{aligned}
&\bar\varphi(t,x)-\bar\varphi(t-\eps,x)\\
&\le \E_*\Big[\int_{0}^{\bar \tau^{t,x}_{*}\wedge(T-t)}\beta^*_s\big(1-c_A\beta^*_s\big)\ud s-\int_{0}^{\bar \tau^{t,x}_*\wedge(T-t+\eps)}\beta^*_s\big(1-c_A\beta^*_s\big)\ud s\Big]\\
&= -\E_*\Big[\mathds{1}_{\{T-t < \bar\tau^{t,x}_*\}}\int_{T-t}^{\bar \tau_*^{t,x}\wedge(T-t+\eps)}\beta^*_s\big(1-c_A\beta^*_s\big)\ud s\Big]\\
&= -\frac{1}{4c_A}\E_*\Big[\mathds{1}_{\{T-t < \bar\tau^{t,x}_*\}}\big(\bar \tau_*^{t,x}\wedge(T-t+\eps)-(T-t)\big)\Big],
\end{aligned}
\end{equation}
where $\E_*[\,\cdot\,]$ is the expectation under $\P_*$. It is clear that on the event $\{T-t < \bar\tau^{t,x}_*\}$ we have $\eps^{-1}|\bar \tau_*^{t,x}\wedge(T-t+\eps)-(T-t)|\le 1$ and 
\begin{equation*}
\lim_{\eps\to 0}\eps^{-1}\big(\bar \tau_*^{t,x}\wedge(T-t+\eps)-(T-t)\big)=1.
\end{equation*}
Therefore, dividing by $\eps>0$ in \eqref{eq:varphi-eps} and letting $\eps\to 0$ yields by dominated convergence theorem
\begin{equation*}
\partial_t\bar\varphi(t,x)\le -\frac{1}{4c_A}\P_*\big(\bar\tau^{t,x}_*>T-t\big)\implies \beta^*(t,x)\ge \frac{1}{2 c_A}\P_*\big(\bar\tau^{t,x}_*>T-t\big),
\end{equation*}
proving the lower bound in \eqref{eq:betabounds}.

For the upper bound we proceed in a slightly different way. First of all, we notice that because $\partial_t\bar\varphi$ is smooth in $[0,T)\times(0,\infty)$ we can differentiate with respect to $t$ the second PDE in \eqref{eq:PDE_aux}. That yields 
\begin{equation*}
\frac{\partial_{t}(\partial_{t}\bar\varphi)(t,x)}{4\big(\partial_t\bar\varphi(t,x)\big)^2}+\tfrac12\sigma^2\partial_{xx}\big(\partial_t\bar\varphi\big)(t,x)+c_A\partial_x\big(\partial_t\bar\varphi\big)(t,x)=0,\quad (t,x)\in[0,T)\times(0,\infty).
\end{equation*}
Multiplying the above expression by $-2[\beta^*(t,x)]^2=-8[\partial_t\bar\varphi(t,x)]^2$ we obtain \eqref{eq:PDEbeta}. The boundary conditions can be immediately deduced from Proposition \ref{prop:beta}. For any $s\in(0,T-t)$ and $x\in(n^{-1},n)$, an application of It\^o's formula for the dynamics of $\bar W^*$ given in \eqref{eq:W*tx} yields
\begin{equation*}
\beta^*(t+s\wedge \tau_n\wedge\tau_{1/n},\bar W^*_{s\wedge\tau_n\wedge\tau_{1/n}})=\beta^*(t,x)+\int_0^{s\wedge\tau_n\wedge\tau_{1/n}}\sigma\beta^*\big(t+u,\bar W^*_u\big)\partial_x \beta^*\big(t+u,\bar W^*_u\big)\ud B^*_u,
\end{equation*} 
where $\tau_z=\inf\{s\ge 0:\bar W^*_s =z\}$ is needed to guarantee that the stochastic integral is a martingale (we do not control $\partial_x\beta^*(t,x)$ at $x=0$ and $x=+\infty$). Taking expectations yields
\begin{equation*}
\begin{aligned}
\beta^*(t,x)&=\E_*\Big[\beta^*(t+s\wedge\tau_n\wedge\tau_{1/n},\bar W^*_{s\wedge\tau_n\wedge\tau_{1/n}})\Big]\\
&=\E_*\Big[\mathds{1}_{\{\tau_{1/n}<s\wedge\tau_n\}}\beta^*(t+\tau_{1/n},1/n)+\mathds{1}_{\{\tau_{1/n}\ge s\wedge\tau_n\}}\beta^*(t+s\wedge\tau_n,\bar W^*_{s\wedge\tau_n})\Big]\\
&\le \sup_{0\le u\le s}\big|\beta^*(t+u,1/n)\big|+\frac{1}{2c_A}\P_*\big(\tau_{1/n}\ge s\wedge\tau_n\big)\\
&\le \sup_{0\le u\le s}\big|\beta^*(t+u,1/n)\big|+\frac{1}{2c_A}\P_*\big(\bar \tau^{t,x}_*\ge s\wedge\tau_n\big),
\end{aligned}
\end{equation*} 
where for the first inequality we used $\beta^*(t,x)\le 1/(2 c_A)$ and for the second one we used $\tau_{1/n}\le \bar\tau^{t,x}_*$. Letting $n\to\infty$, we have $\tau_n\uparrow +\infty$ because the SDE in \eqref{eq:W*tx} has bounded coefficients. Then, using the uniform convergence in Proposition \ref{prop:beta}, the first term in the last expression above vanishes. Thus, we obtain for $s\in(0,T-t)$
\begin{equation*}
\beta^*(t,x)\le \frac{1}{2 c_A}\P_*\big(\bar \tau^{t,x}_*\ge s\big).
\end{equation*}
Finally, letting $s\to T-t$ we obtain the upper bound in \eqref{eq:betabounds}.
\end{proof}

A priori there may be a gap between the two expressions sandwiching the optimal control in \eqref{eq:betabounds}. Because the SDE for $W^*$ is degenerate at zero and we do not have classical bounds for the associated transition density, it is not immediate to rule out that gap. We refine the previous result with a precise probabilistic representation of the optimal control. 
\begin{proposition}\label{prop:betaexact}
We have $\beta^*(t,x)=1/(2c_A)\P_*(\bar\tau^{t,x}_*\ge T-t)$ for $(t,x)\in[0,T]\times[0,\infty)\setminus\{(T,0)\}$.
\end{proposition}
\begin{proof}
With no loss of generality, we extend $\beta^*$ by taking $\beta^*(t,x)=1/(2c_A)$ on $(T,\infty)\times[0,\infty)$.
Recalling the notations from the proof of Theorem \ref{thm:DPP_time} we notice that 
$\bar \tau_*^{t,x}=\inf\{s\ge 0:X^x_{T^*_{t,x}(s)}=0\}$,
where $X^x_s=x+c_A s+\sigma B^*_s$ with $B^*$ a Brownian motion under $\P_*$ and
\begin{equation*}
T^*_{t,x}(s)=\int_0^s\big(\beta^{*;t,x}_u\big)^2\ud u,
\end{equation*} 
for $\beta^{*;t,x}_u=\beta^*(t+u,\bar W^{*;t,x}_u)\mathds{1}_{\{u<\bar \tau_*\}}$.
Then, as noted in Theorem \ref{thm:DPP_time}, $\bar \tau^{t,x}_*=\inf\{s\ge 0: T^*_{t,x}(s)=\rho^*_x\}$ with $\rho^*_x=\inf\{s\ge 0: X^x_s\le 0\}$. Because $s\mapsto T^*_{t,x}(s)$ is continuous and strictly increasing on $[\![0,\bar \tau_*^{t,x})\!)$, we deduce
\begin{equation}\label{eq:equiv*}
\bar \tau^{t,x}_*>T-t\iff T^*_{t,x}(T-t)<\rho^*_x\quad\text{and}\quad\bar \tau^{t,x}_*\ge T-t\iff T^*_{t,x}(T-t)\le \rho^*_x.
\end{equation}
The inverse of the process $s\mapsto T^*_{t,x}(s)$, defined as $S^*_{t,x}(s)\coloneqq\inf\{u\ge 0:T^*_{t,x}(u)=s\}$, satisfies the equation
\begin{equation*}
S^*_{t,x}(s)=\int_0^s\big(\beta^*(Z^{*;t,x}_u,X^{x}_u)\big)^{-2}\ud u,\quad\text{for $s\in[\![0,\rho^*_x)\!)$},
\end{equation*}
where $Z^{*;t,x}_s=t+S^*_{t,x}(s)$. Because it is monotone increasing, $S^*_{t,x}$ can be extended to $[\![0,\rho^*_x]\!]$, possibly with $S^*_{t,x}(\rho^*_x)=+\infty$. Denote 
\begin{equation*}
T^*_{t,x}(T-t)=\inf\{s\ge 0:S^*_{t,x}(s)=T-t\}=\inf\{s\ge 0:Z^{*;t,x}_s=T\}\eqqcolon\sigma^*_{t,x},
\end{equation*}
with the usual convention $\inf\varnothing=+\infty$. Then, from \eqref{eq:equiv*} we also deduce
\begin{equation}\label{eq:equiv*2}
\bar \tau^{t,x}_*>T-t\iff \sigma^*_{t,x}<\rho^*_x\quad\text{and}\quad \bar \tau^{t,x}_*\ge T-t\iff \sigma^*_{t,x}\le \rho^*_x.
\end{equation}

Because $\beta^*>0$ and it is smooth on $[0,T]\times(0,\infty)$, for $s\in[\![0,\rho^*_x\wedge \sigma^*_{t,x})\!)$ we have a pathwise unique solution of the random ODE for $Z^*$. That is, there is a unique solution of 
\begin{equation*}
Z^{*;t,x}_s=t+\int_0^s\big(\alpha^*(Z^{*;t,x}_u,X^x_u)\big)^2\ud u,\quad s\in[\![0,\rho^*_x\wedge \sigma^*_{t,x})\!),
\end{equation*}
where for the ease of notation we set $\alpha^*(z,x)=1/\beta^*(z,x)$. Because $Z^{*;t,x}$ is an increasing process, we can extend it to $[\![0,\rho^*_x\wedge \sigma^*_{t,x}]\!]$ by simply setting $Z^{*;t,x}_{\rho^*_x\wedge \sigma^*_{t,x}}=\lim_{s\uparrow \rho^*_x\wedge \sigma^*_{t,x}}Z^{*;t,x}_s$. Moreover, we extend $Z^{*;t,x}$ to $[0,\infty)$ by simply taking 
\begin{equation*}
Z^{*;t,x}_s=Z^{*;t,x}_{\rho^*_x\wedge \sigma^*_{t,x}},\quad \text{for $s\in[\![\rho^*_x\wedge \sigma^*_{t,x},\infty)\!)$}.
\end{equation*}
From now on we slightly abuse our notations and consider $Z^{*;t,x}_s=Z^{*;t,x}_{s\wedge\rho^*_x\wedge \sigma^*_{t,x}}$ for $s\in[0,\infty)$, without further mentioning.

Taking $t_1<t_2$ and $x>0$, by pathwise uniqueness we get 
\begin{equation}\label{eq:pathunique}
Z^{*;t_2,x}_u > Z^{*;t_1,x}_u,\quad u\in[\![0,\rho^*_x\wedge\sigma^*_{t_1,x})\!).
\end{equation}
For $s\in[\![0,\rho^*_x)\!)$, using $\sigma^*_{t_2,x}<\sigma^*_{t_1,x}$ we also have
\begin{equation*}
\begin{aligned}
0\le Z^{*;t_2,x}_s-Z^{*;t_1,x}_s&=t_2-t_1+\int_0^{s\wedge \sigma^*_{t_2,x}}\Big[\big(\alpha^*(Z^{*;t_2,x}_u,X^x_u)\big)^2-\big(\alpha^*(Z^{*;t_1,x}_u,X^x_u)\big)^2\Big]\ud u\\
&\quad-\int_{s\wedge \sigma^*_{t_2,x}}^{s\wedge \sigma^*_{t_1,x}}\big(\alpha^*(Z^{*;t_1,x}_u,X^x_u)\big)^2\ud u\le t_2-t_1,
\end{aligned}
\end{equation*}
where the final inequality holds because \eqref{eq:pathunique} and the fact that $z\mapsto \beta^*(z,x)$ is positive and increasing imply
\begin{equation*}
\big(\alpha^*(Z^{*;t_2,x}_u,X^x_u)\big)^2\le\big(\alpha^*(Z^{*;t_1,x}_u,X^x_u)\big)^2,\quad u\in[\![0,\rho^*_x\wedge\sigma^*_{t_2,x})\!).
\end{equation*}
Thus we have uniform continuity of the flow $t\mapsto Z^{*;t,x}_s$, i.e.,
\begin{equation}\label{eq:unifcont}
\sup_{0\le s\le \rho^*_x}\big|Z^{*;t_2,x}_s-Z^{*;t_1,x}_s\big|\le t_2-t_1.
\end{equation}

Let us now consider a sequence $t_n\downarrow t$ with $t_n<T$. Then $\sigma^*_{t_n,x} < \sigma^*_{t_{n+1},x} < \sigma^*_{t,x}$ where all inequalities are strict because $t\mapsto Z^{*;t,x}_s$ is strictly increasing by path uniqueness. The limit $\sigma_\infty\coloneqq\lim_{n\to\infty}\sigma^*_{t_n,x}\le \sigma^*_{t,x}$ exists and it is finite; in particular, $\sigma^*_{t,x}\le 4c^2_A(T-t)$ because $Z^{*;t_n,x}_s\ge t_n+(2 c_A)^2 s$, using $\alpha^*(z,x)\ge 2c_A$ thanks to the upper bound on $\beta^*(z,x)$. It is now easy to show that indeed $\sigma_\infty=\sigma^*_{t,x}$. Suppose that $\sigma^*_{t,x}>\delta$. Because $s\mapsto Z^{*;t,x}_s$ is increasing, then $T-Z^{*;t,x}_{\delta} = c>0$ for some $c=c(\omega)$, and 
\begin{equation*}
Z^{*;t_n,x}_ {\delta}=Z^{*;t_n,x}_{\delta}-Z^{*;t,x}_{\delta}+Z^{*;t,x}_{\delta}\le T-c+(t_n-t),
\end{equation*} 
where in the final inequality we used \eqref{eq:unifcont}. For sufficiently large $n$, in the right-hand side above we get $c-(t_n-t)>0$. That implies 
$\sigma_{\infty}>\delta$. Arbitrariness of $\delta$ yields $\sigma_\infty\ge \sigma^*_{t,x}$.

From \eqref{eq:equiv*2} and \eqref{eq:betabounds} we know that
\begin{equation*}
\P_*(\sigma^*_{t_n,x}< \rho^*_x)\le 2c_A\beta^*(t_n,x)\le \P_*(\sigma^*_{t_n,x}\le \rho^*_x).
\end{equation*}
By continuity of $\beta^*$ we also deduce
\begin{equation*}
\begin{aligned}
\P_*\big(\sigma^*_{t,x}\le \rho^*_x\big)\ge 2c_A\beta^*(t,x)&=2c_A\lim_{n\to\infty}\beta^*(t_n,x)\ge \lim_{n\to\infty}\P_*(\sigma^*_{t_n,x}< \rho^*_x)\\
&=\E_*\Big[\lim_{n\to\infty}\mathds{1}_{\{\sigma^*_{t_n,x}< \rho^*_x\}}\Big]=\E_*\Big[\mathds{1}_{\{\sigma^*_{t,x}\le \rho^*_x\}}\Big],
\end{aligned}
\end{equation*}
where the final equality holds by strict monotonicity of the sequence $(\sigma_{t_n,x}^*)_{n\in\N}$. Thus we have 
\begin{equation*}
\beta^*(t,x)=\frac{1}{2c_A}\P_*\big(\sigma^*_{t,x}\le \rho^*_x\big)=\frac{1}{2c_A}\P_*\big(\bar \tau_*^{t,x}\ge T-t\big),
\end{equation*}
as claimed.
\end{proof}

We can strengthen some of the inequalities derived above thanks to the next proposition, which is also instrumental in proving Theorems \ref{thm:stop} and \ref{thm:GH}.
\begin{proposition}\label{cor:strict}
Recall $\phi(t)=(T-t)/(4c_A)$. For $(t,x)\in[0,T)\times(0,\infty)$ we have 
\begin{equation*}
\bar\varphi(t,x)<\phi(t),\quad \partial_t\bar \varphi(t,x)>\phi'(t),\quad \partial_{xx}\bar\varphi(t,x)<0,\quad \partial_{tx}\bar\varphi(t,x)<0.
\end{equation*}
\end{proposition}
\begin{proof}
For the first inequality notice that $\beta(1-c_A\beta)\le 1/(4c_A)$ implies
\begin{equation}\label{eq:attain}
\bar\varphi(t,x)=\E_*\Big[\int_0^{\bar \tau_*^{t,x}\wedge(T-t)}\beta^*_s\big(1-c_A\beta^*_s\big)\ud s\Big]\le \frac{T-t}{4 c_A}.
\end{equation}
Because of strict concavity of $\beta\mapsto \beta(1-c_A\beta)\le 1/(4c_A)$, the upper bound in the equation above is attained only if $\beta^*_s=1/(2c_A)$ for a.e.\ $s\in[0,T-t]$ and at the same time $\bar \tau^{t,x}_*\ge T-t$. More formally, let
\begin{equation*}
\Omega_0\coloneqq\big\{\omega\in\Omega: \beta^*_s(\omega)=\frac{1}{2c_A}\ \text{for a.e.}\ s\in[0,T-t]\ \text{and}\ \bar\tau^{t,x}_*(\omega)\ge T-t \big\}.
\end{equation*}
Then
\begin{equation*}
\begin{aligned}
&\E_*\Big[\int_0^{\bar \tau_*^{t,x}\wedge(T-t)}\beta^*_s\big(1-c_A\beta^*_s\big)\ud s\Big]\\
&=\frac{T-t}{4c_A}\P_*(\Omega_0)+\E_*\Big[\mathds{1}_{\Omega^c_0}\int_0^{\bar \tau_*^{t,x}\wedge(T-t)}\beta^*_s\big(1-c_A\beta^*_s\big)\ud s\Big].
\end{aligned}
\end{equation*}
For $\omega\in\Omega^c_0$ we have
\begin{equation*}
\int_0^{\bar \tau_*^{t,x}(\omega)\wedge(T-t)}\beta^*_s(\omega)\big(1-c_A\beta^*_s(\omega)\big)\ud s<\frac{T-t}{4c_A}.
\end{equation*}
Thus, in order to obtain $\bar\varphi(t,x)=\phi(t)$ we would need $\P_*(\Omega_0)=1$. However, the latter is impossible because on $\Omega_0$ we have
\begin{equation*}
\bar W^{*;t,x}_s=x+\frac{1}{4c_A}s+\frac{\sigma}{2 c_A}B^*_s,\quad s\in[0,T-t],
\end{equation*}
and denoting
\begin{equation*}
E_0\coloneqq\Big\{\inf_{0\le s\le T-t}\Big(\frac{1}{4c_A}s+\frac{\sigma}{2 c_A}B^*_s\Big)>-x\Big\},
\end{equation*}
it must be $\Omega_0\subseteq E_0$. However, $\P_*(E_0)<1$, showing that the upper bound in \eqref{eq:attain} cannot be attained for $(t,x)\in[0,T)\times(0,\infty)$.

To prove the second inequality recall that $\partial_t\bar\varphi(t,x)\ge \phi'(t)=-1/(4c_A)$ because of the lower bound in Proposition \ref{prop:d_time}-(i) and, arguing by contradiction, let us suppose that there is $(t_0,x_0)\in[0,T)\times(0,\infty)$ such that $\partial_t\bar\varphi(t_0,x_0)=-1/(4c_A)$. Then, $(t_0,x_0)$ is a minimum of $\partial_t\bar\varphi$ and because $\partial_{tt}\bar\varphi\le 0$ and $\partial_{tx}\bar\varphi\le 0$ by Proposition \ref{prop:beta}, then $\partial_t\bar\varphi(t,x)=\phi'(t)$ for $(t,x)\in[t_0,T]\times[x_0,\infty)$. The latter implies, for $(t,x)\in[t_0,T]\times[x_0,\infty)$,
\begin{equation*}
\bar\varphi(t,x)=\bar\varphi(T,x)-\int_t^T\partial_t\bar\varphi(s,x)\ud s=\phi(t),
\end{equation*}
where we used $\bar\varphi(T,x)=\phi(T)=0$. The above equation contradicts the first statement in this proposition.

The third inequality holds because of \eqref{eq:concphi}, using that $\partial_t\bar\varphi(t,x)>-1/(4c_A)$ and $\partial_x\bar\varphi\ge 0$. Finally, the fourth inequality holds because in \eqref{eq:be4} we have $c_A/\sigma^2(\partial_t\bar\varphi(t,x)-\phi'(t))>0$ uniformly with respect to $p\ge 0$ and therefore the second inequality is strict, uniformly when we let $p\to 0$.
\end{proof}

The next theorem shows that with positive probability the Agent stops exerting effort prior to the end of the contract.
\begin{theorem}\label{thm:stop}
We have $\P_*(\bar\tau^{0,x}_{*}<T)> 0$ for every $x\in(0,\infty)$. Moreover, $\lim_{x\to\infty}\P_*(\bar\tau^{0,x}_{*}<T)=0$ and $\lim_{x\to 0}\P_*(\bar\tau^{0,x}_{*}<T)=1$.
\end{theorem}
\begin{proof}
Since $\P_*(\bar\tau^{0,x}_{*}\ge T)=-4c_A \partial_t\bar \varphi(0,x)=2c_A\beta^*(0,x)$, thanks to Propositions \ref{prop:betaP} and \ref{prop:betaexact}, the claim holds because $-1/(4c_A) < \partial_t\bar \varphi(0,x)<0$ for all $x\in(0,\infty)$ by Proposition \ref{cor:strict}. Indeed,
\begin{equation*}
\P_*(\bar\tau^{0,x}_{*}<T)=1-\P_*(\bar\tau^{0,x}_{*}\ge T)=1-2c_A\beta^*(0,x)>0.
\end{equation*}
The two claimed limits hold because of the limits of $\beta^*$ (cf.\ Proposition \ref{prop:beta}).
\end{proof}

At time zero, the Principal's expected utility is given by $u(0,x,y)=y-x+\bar \varphi(0,x)$ (cf.\ \eqref{eq:orig_value}), where $y$ is the current output and $x$ is the promised compensation to the Agent at time zero. Notice that the function $u$ is defined on $[0,T]\times[0,\infty)^2$ but imposing the participation constraint is equivalent to restricting the study of $u(0,x,y)$ to the subset $[x_0,\infty)\times[0,\infty)$, where we recall that $x_0=\max\{0,(-1/\gamma_A)\log R_0\}$. 

The Principal should choose $x$ optimally, so as to maximise her utility. In particular, we are going to show that the choice $x=0$ is never optimal. Moreover, should it be optimal to choose $x\in(0,x_0)$, then there is no optimal contract that satisfies the participation constraint. In economic jargon we refer to the strictly positive optimal $x>0$ as ``golden hello''.

\begin{definition}[Golden Hello]
For a given level of output $y\ge 0$ at time zero, the Golden Hello $GH(y)\in[0,\infty)$ is defined as
\begin{equation*}
GH(y)\coloneqq\argmax_{x\ge 0}u(0,x,y).
\end{equation*}
Indeed, $\argmax_{x\ge0}u(0,x,y)=\argmax_{x\ge 0}(\bar \varphi(0,x)-x)$ and $GH(y)=GH$ is independent of $y$.
\end{definition}

The next theorem shows the existence of a nontrivial Golden Hello for any time-horizon of the contract. The statement also shows that $\partial_x\bar \varphi$ diverges to $+\infty$ as $x\downarrow 0$. Therefore, Theorem \ref{thm:main} provides maximal regularity of the function $\bar\varphi$ on $[0,T)\times[0,\infty)$. In the next theorem we consider $GH$ as function of the maturity of the contract. That is, we consider $GH=GH(T)$, because we implicitly use that $\bar\varphi$ depends also on $T$ (see Remark \ref{rem:timehomo}). 
\begin{theorem}\label{thm:GH}
For any $T>0$ we have $GH(T)\in(0,\infty)$. Moreover, $T\mapsto GH(T)$ is increasing with $\lim_{T\downarrow 0}GH(T)=0$. In particular, the first claim is implied by $\partial_x \bar\varphi(t,0+)=+\infty$, for all $t\in[0,T)$. 
\end{theorem}

\begin{proof}
Recall that $\bar \varphi=\varphi^*$. The main step in the proof shows that 
\begin{equation}\label{eq:integrdiv}
\lim_{x\downarrow 0}\int_x^1\frac{1}{\beta^*(0,y)}\ud y=\infty.
\end{equation}
Notice that the limit exists because $\beta^*\ge 0$ and $\beta^*(t,x)>0$ for $(t,x)\in[0,T)\times(0,\infty)$ (cf.\ Proposition \ref{prop:d_time}). Moreover, we know from Proposition \ref{prop:beta} that $\beta^*(0,x)\to 0$ as $x\to 0$. 

If the integral above diverges, using $\beta^*(t,x)=-2\partial_t\varphi^*(t,x)$ and the second PDE in \eqref{eq:PDE_aux} in Corollary \ref{cor:PDE_aux}, we deduce
\begin{equation*}
\begin{aligned}
\infty&=-\lim_{x\downarrow 0}\int_x^1\frac{1}{4\partial_t \varphi^*(0,y)}\ud y\\
&=-\lim_{x\downarrow 0}\int_x^1\Big(\frac12\sigma^2 \partial_{xx} \varphi^*(0,y)+c_A\partial_x\varphi^*(0,y)-c_A\Big)\ud y\\
&=-\tfrac12\sigma^2\partial_{x}\varphi^*(0,1)+\lim_{x\downarrow 0}\tfrac12\sigma^2\partial_x\varphi^*(0,x)-c_A\varphi^*(0,1)+c_A,
\end{aligned}
\end{equation*}
where we also used that $\varphi^*(0,0)=0$ and the limit $\partial_x\varphi^*(0,0+)\coloneqq\lim_{x\downarrow 0}\partial_x\varphi^*(0,x)$ exists by concavity in $x$ of $\varphi^*$ (Theorem \ref{thm:main}). The above equation shows that $\partial_x\varphi^*(0,0+)=+\infty$ and therefore $GH>0$ by definition of $GH$ and concavity of $x\mapsto \varphi^*(0,x)-x$. Recall from Remark \ref{rem:timehomo} that denoting $\varphi^*(t,x;T)$ the value of the problem with maturity $T$, it holds $\varphi^*(t,x;T)=\varphi^*(0,x;T-t)$. Concavity of the map $x\mapsto \varphi^*(0,x;T)-x$ and boundedness of $\varphi^*(0,x;T)$ imply that $GH(T)\in(0,\infty)$ and it satisfies the first order conditions $\partial_x\varphi^*(0,GH(T);T)=1$. By the implicit function theorem 
\begin{equation*}
\frac{\ud GH(T)}{\ud T}=-\frac{\partial_{xT}\varphi^*(0,GH(T);T)}{\partial_{xx}\varphi^*(0,GH(T);T)}=\frac{\partial_{xt}\varphi^*(0,GH(T);T)}{\partial_{xx}\varphi^*(0,GH(T);T)}> 0,
\end{equation*}
where the inequality holds because $\partial_{tx}\varphi^*< 0$ and $\partial_{xx}\varphi^*< 0$ by Proposition \ref{cor:strict}. Therefore $T\mapsto GH(T)$ is increasing, as claimed. 

In order to show $\lim_{T\downarrow 0}GH(T)=0$ we denote $x^-(T)\coloneqq\inf\{x>0:\varphi^*(0,x;T)-x<0\}$. 
By construction, $0<GH(T)< x^-(T)$ and since for $x>0$ we have $\lim_{T\to 0}\varphi^*(0,x;T)=0$, then it is clear that 
\begin{equation*}
0\le \lim_{T\downarrow 0}GH(T)\le\lim_{T\downarrow 0}x^-(T)=0. 
\end{equation*}

In the remainder we prove \eqref{eq:integrdiv}. Let us argue by contradiction and assume 
\begin{equation}\label{eq:integrnodiv}
\lim_{x\downarrow 0}\int_x^1\frac{1}{\beta^*(0,y)}\ud y<\infty.
\end{equation}
Notice that since $\beta^*$ is non-decreasing in $t$ (cf.\ Proposition \ref{prop:beta}-(i)), the above condition implies that the integral is indeed finite when we replace $\beta^*(0,y)$ with $\beta^*(t,y)$ for any $t\in[0,T)$. Moreover, the map $y\mapsto\beta^*(t,y)$ is increasing and concave (cf.\ Propositions \ref{prop:beta}-(ii) and \ref{cor:strict}), therefore the limit $\partial_x\beta^*(t,0+)$ is well-defined. Because of \eqref{eq:integrnodiv} it must be 
\begin{equation}\label{eq:limbetax}
\partial_x\beta^*(t,0+)=+\infty,\quad \text{for all $t\in[0,T)$},
\end{equation} 
as otherwise $\beta^*(t,y)\sim y$ when $y\downarrow 0$ and the integral in \eqref{eq:integrnodiv} should diverge. 

For a fixed $\delta\in(0,1]$, let us define a process $F^\delta_t\coloneqq f_\delta(t,\bar W^*_t)$ via the Lamperti transform with 
\begin{equation*}
f_\delta(t,x)=\int_\delta^x\frac{1}{\beta^*(t,y)}\ud y.
\end{equation*} 
Let us start by noticing that thanks to \eqref{eq:integrnodiv}
\begin{equation}\label{eq:taus}
\bar \tau^{0,x}_{*}=\inf\{s\in[0,T]: W^{*;0,x}_{s}\le 0\}=\inf\{s\in[0,T]: F^{\delta;f(0,x)}_s\le f_\delta(s,0)\}\eqqcolon\theta^x_\delta,
\end{equation}
where we use $F^{\delta;f(0,x)}$ to keep track of the dependence of the process on $x$.
The derivatives 
\begin{equation*}
\partial_x f_\delta(t,x)=\frac{1}{\beta^*(t,x)},\quad \partial_{xx} f_\delta(t,x)=-\frac{\partial_x\beta^*(t,x)}{\big(\beta^*(t,x)\big)^2}, \quad\partial_t f_\delta(t,x)=-\int_\delta^x\frac{\partial_t\beta^*(t,y)}{(\beta^*(t,y))^2}\ud y
\end{equation*}
are well-defined and continuous for $(t,x)\in[0,T)\times(0,\infty)$. 
From \eqref{eq:PDEbeta} we know
\begin{equation*}
\partial_t\beta^*(t,x)=-c_A\big[\beta^*(t,x)\big]^2\partial_x\beta^*(t,x)-\tfrac12\sigma^2\big[\beta^*(t,x)\big]^2\partial_{xx}\beta^*(t,x),\quad (t,x)\in[0,T)\times(0,\infty).
\end{equation*} 
Substituting into the expression for $\partial_t f_\delta$ gives
\begin{equation*}
\begin{aligned}
\partial_t f_\delta(t,x)&=\int_\delta^x\Big(c_A\partial_x\beta^*(t,y)+\tfrac12\sigma^2\partial_{xx}\beta^*(t,y)\Big)\ud y\\
&=c_A\beta^*(t,x)-c_A\beta^*(t,\delta)+\tfrac12\sigma^2\partial_x\beta^*(t,x)-\tfrac12\sigma^2\partial_x\beta^*(t,\delta)\\
&\le c_A\beta^*(t,x)+\tfrac12\sigma^2\partial_x\beta^*(t,x)-\tfrac12\sigma^2\partial_x\beta^*(t,\delta),
\end{aligned}
\end{equation*}
where the inequality holds because $\beta^*\ge 0$. 
This leads to a useful observation:
\begin{equation*}
\begin{aligned}
&\partial_tf_\delta(t,x)+c_A\big[\beta^*(t,x)\big]^2\partial_x f_\delta(t,x)+\tfrac12\sigma^2\big[\beta^*(t,x)\big]^2\partial_{xx}f_\delta(t,x)\\
&\le c_A\beta^*(t,x)+\tfrac12\sigma^2\partial_x\beta^*(t,x)-\tfrac12\sigma^2\partial_x\beta^*(t,\delta)+c_A\beta^*(t,x)-\tfrac12\sigma^2\partial_x\beta^*(t,x)\\
&\le 1-\tfrac12\sigma^2\partial_x\beta^*(t,\delta),
\end{aligned}
\end{equation*}
where the final inequality holds because $0\le \beta^*(t,x)\le 1/(2c_A)$.

From It\^o's formula and the considerations above we deduce an upper bound for the dynamics of the process $F^\delta$. That is,
\begin{equation}\label{eq:boundZd}
\begin{aligned}
F^\delta_t&= F^\delta_0+\sigma B^*_t\\
&\quad+\int_0^t\Big(\partial_t f_\delta(s,W^*_s)+c_A\big[\beta^*(s,W^*_s)\big]^2\partial_xf_\delta(s,W^*_s)+\tfrac12\sigma^2\big[\beta^*(s,W^*_s)\big]^2\partial_{xx}f_\delta(s,W^*_s)\Big)\ud s\\
&\le F^\delta_0+\sigma B^*_t+\int_0^t\big(1-\tfrac12\sigma^2\partial_x\beta^*(s,\delta)\big)\ud s\eqqcolon f_\delta(0,x)+\sigma B^*_t+\Lambda^\delta_t.
\end{aligned}
\end{equation}
Recalling now \eqref{eq:taus} and Proposition \ref{prop:betaexact} we can obtain an upper bound on $\beta^*$ as follows
\begin{equation*}
\begin{aligned}
2c_A \beta^*(0,x)&= \P_*\big(\bar \tau^{0,x}_*\ge T\big)=\P_*\big(\theta^x_\delta \ge T\big)=\P_*\Big(\inf_{0\le s < T}\big(F^\delta_s-f_\delta(s,0)\big)>0\Big)\\
&\le\P_*\Big(\inf_{0\le s\le T}\big(f_\delta(0,x)+\sigma B^*_s+\Lambda^\delta_s-f_\delta(s,0)\big)>0\Big)\\
&=\P_*\Big(\inf_{0\le s\le T}\big(\sigma B^*_s+\Lambda^\delta_s-f_\delta(s,0)\big)>-f_\delta(0,x)\Big)\\
&\le \P_*\Big(\inf_{0\le s\le T}\big(\sigma B^*_s+\Lambda^\delta_s-f_\delta(s,0)\big)>-f_0(0,x)\Big),
\end{aligned}
\end{equation*}
where in the first inequality we used \eqref{eq:boundZd} and in the second one we used $f_\delta(0,x)\le f_0(0,x)$. Because $s\mapsto \beta^*(s,y)$ is non-decreasing, we have
\begin{equation*}
-f_\delta(s,0)=\int_0^\delta\frac{1}{\beta^*(s,y)}\ud y\le \int_0^\delta\frac{1}{\beta^*(0,y)}\ud y=f_0(0,\delta)\le f_0(0,1).
\end{equation*}
Therefore, we can continue with the chain of inequalities and write
\begin{equation*}
\begin{aligned}
2c_A \beta^*(0,x) &\le \P_*\Big(\inf_{0\le s\le T}\big(\sigma B^*_s+\Lambda^\delta_s\big)>-f_0(0,x)-f_0(0,1)\Big)\\
&\le \P_*\Big(\sigma B^*_{T/2}+\Lambda^\delta_{T/2}>-f_0(0,x)-f_0(0,1)\Big),
\end{aligned}
\end{equation*}
where the second inequality is obvious because $\inf_{0\le s\le T}\big(\sigma B^*_s+\Lambda^\delta_s\big)\le \sigma B^*_{T/2}+\Lambda^\delta_{T/2}$. Taking limits as $\delta\downarrow 0$ we get
\begin{equation}\label{eq:betacontr}
\begin{aligned}
2c_A \beta^*(0,x)&\le \lim_{\delta \to 0}\P_*\Big(\sigma B^*_{T/2}+\Lambda^\delta_{T/2}>-f_0(0,x)-f_0(0,1)\Big)=0,
\end{aligned}
\end{equation}
where the final equality holds because $+\infty>\Lambda^\delta_{T/2}>\Lambda^{\delta'}_{T/2}$ for $\delta>\delta'>0$ and, by monotone convergence $\Lambda^{\delta}_{T/2}\downarrow -\infty$ as $\delta\downarrow 0$ thanks to \eqref{eq:limbetax}. 

The inequality in \eqref{eq:betacontr} yields a contradiction with $\beta^*(0,x)=-2\partial_t\varphi^*(0,x)>0$ for $x>0$ (cf.\ Proposition \ref{prop:d_time}). Thus, \eqref{eq:integrnodiv} cannot hold and it must be that \eqref{eq:integrdiv} holds instead. 
\end{proof}

\appendix

\section{Some technical results}\label{app:stopping}

We start by proving the existence of the solution to the BSDE \eqref{eq:BSDE}. The proof is essentially the same as in \cite[Thm.\ 3.1]{briand2007one} but for the case $\eta> 0$ we do not need to impose that $\xi\ge 0$ has exponential moments.
\begin{proposition}\label{prop:bsde}
Under condition {\em \bf (H1)} the BSDE \eqref{eq:BSDE} admits a solution $(\bar W_t,\bar Z_t)_{t\in[0,T]}$.
\end{proposition}
\begin{proof}
The proof for $\eta<0$ and $\E_0[\exp(-\eta/\sigma^2 \xi)]<\infty$ is the same as in \cite[Thm.\ 3.1]{briand2007one}. Here we consider $\eta > 0$ and $\E_0[\xi^2]<\infty$.

Let us set $\widehat W_t\coloneqq \E_0[\exp(-\eta\xi/\sigma^2)|\bar\cF_t]$. Because $\eta\xi\ge 0$, then $(\widehat W_t)_{t\in[0,T]}$ is a bounded martingale with $0< \widehat W_t\le 1$ for $t\in[0,T]$ and, in particular, $\widehat W_0=\E_0[\exp(-\eta\xi/\sigma^2)]$. By the martingale representation theorem 
\begin{equation*}
\widehat W_t=\widehat W_0+\int_0^t H_s\ud \bar B_s,\quad t\in[0,T],
\end{equation*}
for some previsible process $(H_s)_{s\in[0,T]}$ such that 
$\E_0\Big[\int_0^T|H_s|^2\ud s\Big]<\infty$.

Let $\bar W_t\coloneqq-\sigma^2/\eta\log \widehat W_t$. Then, $\bar W_T=\xi$ and by conditional Jensen's inequality
\begin{equation}\label{eq:boundsW}
0 \le \bar W_t=-\frac{\sigma^2}{\eta}\log\Big(\E_0[\exp(-\eta\xi/\sigma^2)|\bar \cF_t]\Big)\le \E_0[\xi|\bar \cF_t].
\end{equation}
For $n\in\N$ we set
\begin{equation*}
\tau_n\coloneqq\inf\left\{t\ge 0:\max\Big\{\int_0^t\Big|\frac{H_s}{\widehat W_s}\Big|^2\ud s,\Big|\int_0^tH_s\ud \bar B_s\Big|,\E_0[\xi|\bar \cF_t]\Big\}>n\right\}\wedge T.
\end{equation*}
Then, by It\^o's formula
\begin{equation}\label{eq:pbsde}
\bar W_{t\wedge \tau_n}=\bar W_0-\frac{\sigma^2}{\eta}\int_0^{t\wedge \tau_n}\frac{H_s}{\widehat W_s}\ud \bar B_s+\frac12\frac{\sigma^2}{\eta}\int_0^{t\wedge \tau_n}\Big|\frac{H_s}{\widehat W_s}\Big|^2\ud s.
\end{equation}
Rearranging terms and using \eqref{eq:boundsW} we obtain
\begin{equation*}
\int_0^{t\wedge \tau_n}\Big|\frac{H_s}{\widehat W_s}\Big|^2\ud s\le\frac{2\eta}{\sigma^2}\E_0[\xi|\bar \cF_{t\wedge\tau_n}]+ 2\Big|\int_0^{t\wedge \tau_n}\frac{H_s}{\widehat W_s}\ud \bar B_s\Big|.
\end{equation*}
Taking the square of both sides, using $(a+b)^2\le 2 a^2+2b^2$ and then taking expectations we obtain
\begin{equation*}
\begin{aligned}
\E_0\left[\Big(\int_0^{t\wedge \tau_n}\Big|\frac{H_s}{\widehat W_s}\Big|^2\ud s\Big)^2\right]&\le \frac{8\eta^2}{\sigma^4}\E_0\Big[\Big(\E_0[\xi|\bar \cF_{t\wedge\tau_n}]\Big)^2\Big]+8\E_0\Big[\int_0^{t\wedge \tau_n}\Big|\frac{H_s}{\widehat W_s}\Big|^2\ud s\Big]\\
&\le \frac{8\eta^2}{\sigma^4}\E_0\Big[\E_0[\xi^2|\bar \cF_{t\wedge\tau_n}]\Big]+\frac12\E_0\Big[\Big(\int_0^{t\wedge \tau_n}\Big|\frac{H_s}{\widehat W_s}\Big|^2\ud s\Big)^2\Big]+32,
\end{aligned}
\end{equation*}
where the second inequality uses conditional Jensen's inequality for the first term and $8 x\le \frac12 x^2+32$ for the second one. Then, rearranging terms in the inequality and letting $t\uparrow T$, by Monotone convergence we get
\begin{equation}\label{eq:boundbsde}
\sup_{n\in\N}\E_0\left[\Big(\int_0^{\tau_n}\Big|\frac{H_s}{\widehat W_s}\Big|^2\ud s\Big)^2\right]\le \frac{16\eta^2}{\sigma^4}\E_0\big[\xi^2\big]+64.
\end{equation}

Now, $\tau_n\le \tau_{n+1}\le T$ and therefore $\tau_\infty\coloneqq\lim_{n\to\infty}\tau_n\le T$ is well-defined as a limit. By monotonicity of the sequence of stopping times
\begin{equation*}
\begin{aligned}
\P_0\big(\tau_\infty<T\big)&\le \lim_{n\to\infty}\P_0\big(\tau_n<T\big)\\
&\le\lim_{n\to\infty}\P_0\Big(\tau_n<T,\sup_{0\le t\le T}\Big\{\E_0[\xi|\bar \cF_t]+\Big|\int_0^t H_s\ud \bar B_s\Big|\Big\}>n\Big)\\
&\quad+\lim_{n\to\infty}\P_0\Big(\tau_n<T,\int_0^{\tau_n}\Big|\frac{H_s}{\widehat W_s}\Big|^2\ud s\ge n\Big).
\end{aligned}
\end{equation*}
Markov's inequality along with $(a+b)^2\le 2a^2+2b^2$ and Doob's martingale inequality yield
\begin{equation*}
\begin{aligned}
&\P_0\Big(\sup_{0\le t\le T}\Big\{\E_0[\xi|\bar \cF_t]+\Big|\int_0^t H_s\ud \bar B_s\Big|\Big\}>n\Big)\\
&\le \frac{2}{n^2}\E_0\Big[\sup_{0\le t\le T}\E_0[\xi|\bar \cF_t]^2+\sup_{0\le t\le T}\Big|\int_0^t H_s\ud \bar B_s\Big|^2\Big]\\
&\le \frac{8}{n^2}\left(\sup_{0\le t\le T}\E_0[\E_0[\xi|\bar \cF_t]^2]+\sup_{0\le t\le T}\E_0\Big[\Big|\int_0^t H_s\ud \bar B_s\Big|^2\Big]\right)\\
&\le \frac{8}{n^2}\left(\E_0[\xi^2]+\E_0\Big[\int_0^T \big|H_s\big|^2\ud s\Big]\right)\xrightarrow{n\to\infty}0,
\end{aligned}
\end{equation*}
where in the final inequality we used conditional Jensen's inequality, tower property and It\^o's isometry. Analogously we obtain
\begin{equation}\label{eq:Ptauinf}
\begin{aligned}
\P_0\Big(\tau_n<T,\int_0^{\tau_n}\Big|\frac{H_s}{\widehat W_s}\Big|^2\ud s\ge n\Big)\le \frac{1}{n^2}\E_0\Big[\Big(\int_0^{\tau_n}\Big|\frac{H_s}{\widehat W_s}\Big|^2\ud s\Big)^2\Big]\le \frac{16\eta^2}{\sigma^4 n^2}\E_0\big[\xi^2\big]+\frac{64}{n^2}\xrightarrow{n\to\infty} 0,
\end{aligned}
\end{equation}
where we used \eqref{eq:boundbsde} for the second inequality. Plugging everything back into \eqref{eq:Ptauinf} we deduce $\tau_\infty=T$, $\P_0$-a.s. This fact and Monotone convergence, combined with \eqref{eq:boundbsde} yield
\begin{equation*}
\E_0\left[\Big(\int_0^T\Big|\frac{H_s}{\widehat W_s}\Big|^2\ud s\Big)^2\right]=\sup_{n\in\N}\E_0\left[\Big(\int_0^{\tau_n}\Big|\frac{H_s}{\widehat W_s}\Big|^2\ud s\Big)^2\right]\le \frac{16\eta^2}{\sigma^4}\E_0\big[\xi^2\big]+64.
\end{equation*}

Then, we can let $n\to\infty$ in \eqref{eq:pbsde} and obtain
\begin{equation*}
\bar W_{t}=\bar W_0-\frac{\sigma^2}{\eta}\int_0^{t}\frac{H_s}{\widehat W_s}\ud \bar B_s+\frac12\frac{\sigma^2}{\eta}\int_0^{t}\Big|\frac{H_s}{\widehat W_s}\Big|^2\ud s,\quad t\in[0,T].
\end{equation*}
Since $\bar W_T=\xi$, setting $\bar Z_t=-(\sigma/\eta) H_t/\widehat W_t$ the pair $(\bar W_t,\bar Z_t)_{t\in[0,T]}$ solves the BSDE in \eqref{eq:BSDE}.
\end{proof}

Next we prove a well-known result used for the proof of Proposition \ref{prop:twice_diff}. The lemma says that the Brownian motion $(B_t)_{t\ge 0}$ is a martingale until the first exit time from the moving strip 
\begin{equation*}
\{(s,z)\in[0,\infty)\times\R: a-(c_A/\sigma) s\le z\le b-(c_A/\sigma) s\}
\end{equation*} 
for any finite $a<b$.
\begin{lemma}\label{lem:OS}
Let $(a,b)$ be bounded with $a<0<b$. Fix a constant $\theta>0$ and let 
\begin{equation*}
\tau\coloneqq\inf\{s\ge 0: \theta s+B_s\notin (a,b)\}.
\end{equation*} 
Then $\E[B_{\tau}]=0$. 
\end{lemma}

\begin{proof}
We first prove that $\E[\tau]<\infty$.
The expectation $\E[\tau]$ can be computed using Green's function associated to the process $\theta s+B_s$ (see \cite[Eq.\ (13.1.12)]{peskir2006optimal} for details). Recalling the scale function $S(x)=1-\e^{-2 \theta x}$, we have
\begin{equation*}
\E[\tau]=\int_a^0\frac{(S(b)-S(0))(S(y)-S(a))}{S(b)-S(a)}\frac{2}{S'(y)}\ud y+\int_0^b\frac{(S(b)-S(y))(S(0)-S(a))}{S(b)-S(a)}\frac{2}{S'(y)}\ud y<\infty.
\end{equation*}
For $t\in[0,\infty)$, we have
$\E[B^2_{t\wedge\tau}]=\E[t\wedge\tau]\le \E[\tau]<\infty$.
Therefore, the family $(B_{t\wedge\tau})_{t\in[0,\infty)}$ is uniformly integrable and the optional sampling theorem applies (cf.\ \cite[Thm.\ II.3.2]{revuz2013continuous}), yielding $\E[B_{\tau}]=\E[B_0]=0$.
\end{proof}

Next we prove Proposition \ref{prop:W12p}, which says that $\varphi^*\in W^{1,2;p}_{\ell oc}([0,T]\times[0,\infty))\cap C([0,T]\times[0,\infty))$.

\begin{proof}[{\bf Proof of Proposition \ref{prop:W12p}}]
Continuity of $\varphi^*$ on $[0,T]\times[0,\infty)$ was proven in Proposition \ref{cor:contphi}. Since $\varphi^*$ is locally Lipschitz in $[0,T]\times[0,\infty)$, then $\varphi^*\in W^{1,1;\infty}_{\ell oc}([0,T]\times[0,\infty))$ by \cite[Prop.\ 9.3-(iii)]{brezis2010functional}. It remains to prove that $\partial_{xx}\varphi^*\in L^\infty_{\ell oc}([0,T]\times[0,\infty))$. 

We follow \cite[Prop.\ 9.3]{brezis2010functional}.
Let us consider a compact subset $I$ of $(0,\infty)$ and define the linear operator $T_\varphi: C_c^{\infty}([0,T]\times I)\to\R$ given by 
\begin{equation}\label{eq:Tphi}
C_c^{\infty}([0,T]\times I)\ni\psi\mapsto T_\varphi(\psi)\coloneqq\int_{[0,T]\times I} \varphi^*(t,x)\partial_{xx}\psi(t,x)\ud t \ud x,
\end{equation}
where $\ud t\ud x$ is the product Lebesgue measure on the Borel $\sigma$-algebra on $[0,T]\times I$. 

Since $\varphi^*(t,x)\partial_{xx}\psi(t,x)$ is a bounded continuous function, 
by Fubini's theorem \cite[Thm.\ 4.5]{brezis2010functional},
\begin{equation*}
T_\varphi(\psi)=\int_{0}^{T}\Big(\int_{I} \varphi^*(t,x)\partial_{xx}\psi(t,x)\ud x\Big) \ud t.
\end{equation*}
Since $\partial_{xx}\varphi^*(t,\,\cdot\,)$ exists a.e.\ for each fixed $t\in[0,T]$ and it is bounded on $I$ (Proposition \ref{prop:twice_diff}), we introduce the function $\Lambda:[0,T]\to\R$ given by
\begin{equation*}
\Lambda(t)\coloneqq\int_{I} \partial_{xx}\varphi^*(t,x)\psi(t,x)\ud x=\int_{I} \varphi^*(t,x)\partial_{xx}\psi(t,x)\ud x.
\end{equation*}
The second expression above tells us that $t\mapsto \Lambda (t)$ is actually continuous because $\varphi^*\cdot\partial_{xx}\psi$ is bounded continuous on $[0,T]\times I$.
Moreover, recalling the constant $C_0>0$ from \eqref{eq:D2} we have
\begin{equation*}
\begin{aligned}
|T_\varphi(\psi)|&=\Big|\int_0^T\Lambda(t) \ud t\Big|\le \int_0^T \Big|\int_{I} \partial_{xx}\varphi^*(t,x)\psi(t,x)\ud x\Big| \ud t\\
&\le \int_0^T\Big(\int_{I} |\partial_{xx}\varphi^*(t,x)| |\psi(t,x)|\ud x\Big) \ud t\\
&\le C_0\int_0^T\Big(\int_{I}|\psi(t,x)|\ud x\Big) \ud t=C_0\|\psi \|_{L^1([0,T]\times I)}.
\end{aligned}
\end{equation*}
Since $C^{\infty}_c([0,T]\times I)$ is dense in $L^1([0,T]\times I)$ (cf.\ \cite[Cor.\ 4.3]{brezis2010functional}), the Hahn-Banach Theorem \cite[Cor.\ 1.3]{brezis2010functional} guarantees that $T_\varphi$ can be extended to a bounded linear operator on $L^1([0,T]\times I)$ 
so that
\begin{equation*}
|T_\varphi(\psi)|\le C_0\|\psi \|_{L^1([0,T]\times I)},\quad\text{for all $\psi\in L^1([0,T]\times I)$}.
\end{equation*}
By Riesz representation theorem \cite[Thm.\ 4.14]{brezis2010functional}, there exists a unique $f\in L^{\infty}([0,T]\times I)$ such that, for any $\psi\in L^1([0,T]\times I)$
\begin{equation}\label{eq:riesz}
T_\varphi(\psi)=\int_{[0,T]\times I} f(t,x)\psi (t,x)\ud t\ud x,
\end{equation}
and $\|f\|_{L^\infty([0,T]\times I)}\le C_0$.
Comparing \eqref{eq:riesz} and \eqref{eq:Tphi} we deduce that $f=\partial_{xx}\varphi^*$ by definition of weak-derivative.
By the arbitrariness of $I$, we have $\partial_{xx}\varphi^*\in L^\infty_{\ell oc}([0,T]\times[0,\infty))$.
\end{proof}

\section{Some results in stochastic control}\label{app:Ito-Krylov}
In order to make our paper self-contained, in this section we collate useful results from stochastic control and PDE theory sourced from Krylov's work.
The first theorem is a generalisation of It\^o's formula, known as It\^o--Krylov's formula. In fact, we state only a part of the original result from \cite[Thm.\ 2.10.1]{krylov2009controlled} which concerns Dynkin's formula, because that is the one we actually use in parts of the proof of Theorem \ref{thm:DPP_time}. To further simplify the notation, we adapt the statement to our needs and only consider one-dimensional It\^o diffusions. Of course the result holds in the $d$-dimensional case as well.
\begin{theorem}[{\cite[Thm.\ 2.10.1]{krylov2009controlled}}]
Let $(\Omega,\cF,(\cF_t)_{t\ge 0},\P)$ be a filtered probability space equipped with a one-dimensional Brownian motion $(W_t)_{t\ge 0}$. For $x_0\in\R$, let the process $(X_t)_{t\ge 0}$ be given by 
\begin{equation*}
X_t=x_0+\int_0^t b_r\ud r+\int_0^t\sigma_r\ud W_r,\quad t\ge 0,
\end{equation*}
where $b_r=b_r(\omega)$ and $\sigma_r=\sigma_r(\omega)$ are progressively measurable and real-valued. For an open set $Q\subset[0,\infty)\times \R$ and for $(t_0,x_0)\in Q$, denote
$\tau_Q=\inf\{s\ge 0:(t_0+s,X_s)\notin Q\}$.
Let $\tau$ be a stopping time such that $\tau\le \tau_Q$ and suppose there exist constants $K,\delta>0$ such that for all $(t,\omega)$ such that $t\le \tau(\omega)$ it holds
$|\sigma_t(\omega)|+|b_t(\omega)|\le K$ and $\sigma_t(\omega)\ge \delta$.
Denote 
$\cL_{t}=\frac12\sigma^2_t \partial_{xx} +b_t \partial_{x}$.

Then, for any $v\in W^{1,2;2}(Q)$,
\begin{equation*}
v(t_0,x_0)=\E\Big[v(t_0+\tau,X_\tau)-\int_0^{\tau}\big(\partial_t+\cL_r\big)v(t_0+r,X_r)\ud r\Big].
\end{equation*}
\end{theorem}

The next result is \cite[Cor.\ 2.4.8]{krylov2009controlled}, which we use in the proof of Theorem \ref{thm:PDE_q-orig}. The statement requires the introduction of some notation. Once again, we adapt Krylov's original statement to our needs and we take a simplified setting for one-dimensional controlled It\^o diffusions. For the general result we refer the reader to \cite{krylov2009controlled}.

We consider a weak formulation of a stochastic control problem. Let $Q\subset[0,\infty)\times\R$ and fix $(t_0,x_0)\in Q$. The class of admissible controls is denoted $\cB(t_0,x_0)$ and $\alpha\in\cB(t_0,x_0)$ is a tuple 
\begin{equation*}
\alpha=(\Omega^\alpha,\cF^\alpha,\P^\alpha,W_t^\alpha,\cF_t^\alpha,\sigma_t^\alpha,b_t^\alpha),
\end{equation*}
where $(\Omega^\alpha,\cF^\alpha,\P^\alpha)$ is a probability space, $(W_t^\alpha,\cF_t^\alpha)_{t\ge 0}$ is a $1$-dimensional Brownian motion on $(\Omega^\alpha,\cF^\alpha,\P^\alpha)$, processes $\sigma_t^\alpha$ and $b_t^\alpha$ are progressively measurable with respect to $(\cF_t^\alpha)_{t\ge 0}$ and there are constants $K,\delta>0$ such that $|\sigma_t^\alpha(\omega)|+|b_t^\alpha(\omega)|\le K$ and $\sigma_t^\alpha(\omega)\ge \delta$, for all $(t,\omega)\in[0,\infty)\times\Omega^\alpha$.
The dynamics associated to a control $\alpha\in\cB(t_0,x_0)$ reads as
\begin{equation*}
X_t^{\alpha,x_0}=x_0+\int_0^t\sigma_u^\alpha\ud W_u^\alpha+\int_0^t b_u^\alpha \ud u,\quad t\ge 0,
\end{equation*}
and we denote $\tau^\alpha=\tau^{\alpha,t_0,x_0}=\inf\{s\ge 0:(t_0+s,X_s^{\alpha,x_0})\notin Q\}$.

\begin{corollary}[{\cite[Cor.\ 2.4.8]{krylov2009controlled}}]\label{cor:krylov248}
Let $f\in L^{p}(Q)$ for some $p\in[2,\infty)$, with $f\le 0$ a.e.\ in $Q$. For all $(s,x)\in Q$ let
\begin{equation*}
\sup_{\alpha\in\cB(s,x)}\E^\alpha\Big[\int_0^{\tau^\alpha}f(s+t,X_t^{\alpha,x})\ud t\Big]=0.
\end{equation*}
Then $f=0$ a.e.\ in $Q$.
\end{corollary}
We find this result particularly nice and we outline the main idea of the proof. The interested reader can find fuller details in Section 2.4 of \cite{krylov2009controlled}.
\begin{proof}
We prove the result assuming that $f\in C^\infty_c(Q)$ and $f\le 0$ in $Q$. The general case with $f\in L^{p}(Q)$, $p\in[2,\infty)$, is obtained by approximation (cf.\ \cite[Thms.\ 2.4.6, 2.4.7]{krylov2009controlled} which also rely upon \cite[Thm.\ 2.2.4]{krylov2009controlled}). 

Since $f\le 0$ we have 
\begin{equation*}
\begin{aligned}
0&=\sup_{\alpha\in\cB(s,x)}\E^\alpha\Big[\int_0^{\tau^\alpha}f(s+u,X_u^{\alpha,x})\ud u\Big]\\
&\le\sup_{\alpha\in\cB(s,x)}\E^\alpha\Big[\int_0^{\tau^\alpha}\e^{-\lambda u}f(s+u,X_u^{\alpha,x})\ud u\Big]\eqqcolon z^{\lambda}(s,x;f)\le 0.
\end{aligned}
\end{equation*}
Then, $z^{\lambda}(s,x;f)=0$ and we now show that $\lambda z^\lambda(s,x;f)\to f(s,x)$ as $\lambda \to \infty$. That implies $f=0$ as needed. 
Applying It\^o's formula we have,
\begin{equation*}
f(s,x)=\E^\alpha\Big[\int_0^{t\wedge \tau^\alpha}\e^{-\lambda u}\big(\lambda f(s+u,X_u^{\alpha,x})-\cL_u^\alpha f(s+u,X_u^{\alpha,x})\big)\ud u+\e^{-\lambda(t\wedge\tau^\alpha)}f(s+t\wedge\tau^\alpha,X_{t\wedge\tau^\alpha}^{\alpha,x})\Big],
\end{equation*}
where
$\cL_u^\alpha =\frac{\partial}{\partial t}+\frac12(\sigma_u^\alpha)^2\partial_{xx}+b_u^\alpha\partial_x$. Since $\sigma_t$ and $b_t$ are bounded and $f$ has compact support we have 
\begin{equation*}
\big|\lambda f(s,x)-\cL_u^\alpha f(s,x)\big|\le h(t,x), 
\end{equation*}
for a bounded positive function $h$. Then, letting $t\to \infty$ and using again that $f$ has compact support in $Q$ we deduce by dominated convergence theorem
\begin{equation*}
f(s,x)=\E^\alpha\Big[\int_0^{\tau^\alpha}\e^{-\lambda u}\big(\lambda f(s+u,X_u^{\alpha,x})-\cL_u^\alpha f(s+u,X_u^{\alpha,x})\big)\ud u\Big].
\end{equation*}
Using the above representation of $f$ we now have
\begin{equation*}
\begin{aligned}
|\lambda z^\lambda(s,x;f)-f(s,x)|=&\Big|\sup_{\alpha\in\cB(s,x)}\E^\alpha\Big[\int_0^{\tau^\alpha}\e^{-\lambda u}\cL_u^\alpha f(s+u,X_u^{\alpha,x})\ud u\Big]\Big|\\
\le&\sup_{\alpha\in\cB(s,x)}\E^\alpha\Big[\int_0^{\tau^\alpha}\e^{-\lambda u}\big|\cL_u^\alpha f(s+u,X_u^{\alpha,x})\big|\ud u\Big]\\
\le&\sup_{\alpha\in\cB(s,x)}\E^\alpha\Big[\int_0^{\tau^\alpha}\e^{-\lambda u}h(s+u,X_u^{\alpha,x})\ud u\Big]=\frac{\|h\|_{L^{\infty}(Q)}}{\lambda}.
\end{aligned}
\end{equation*}
It follows that $\lambda z^{\lambda}(s,x;f)\to f(s,x)$ when $\lambda\to \infty$, as needed.
\end{proof}

The final result in this section is a lemma that we use in the proof of Theorem \ref{thm:smooth}. This is taken from \cite[Lem.\ 4.1]{krylov1981properties} and we need to introduce some notation. 
Let $\|\cdot\|_{d}$ be the euclidean norm of $\R^d$. Consider the parabolic operator
\begin{equation}\label{eq:kry_PDE}
\cL=\sum_{i,j=1}^da^{ij}(t,x)\frac{\partial}{\partial x_i\partial x_j}+\sum_{j=1}^db^j(t,x)\frac{\partial}{\partial x_j}-\frac{\partial}{\partial t},
\end{equation}
where the coefficients $a^{ij}$ and $b^j$ are assumed to be measurable for $i,j\in\{1,\ldots d\}$ and there is a constant $\mu>0$ such that
\begin{equation}\label{eq:kry_ass}
\begin{aligned}
\mu \|\lambda\|^2_d\le \sum_{i,j=1}^da^{ij}(t,x)\lambda_i\lambda_j\le \mu^{-1}\|\lambda\|^2_d\quad\text{and}\quad \| b(t,x)\|_d \le \mu^{-1},
\end{aligned}
\end{equation}
for all $(t,x)\in\R^{d+1}$ and $\lambda\in\R^d$. In \cite{krylov1981properties} the authors are interested in parabolic PDEs that are forward in time and for $z_0=(t_0,x_0)\in\R^{d+1}$ they consider sets of the form 
\begin{equation*}
C_{r}(z_0)=\{(t,x)\in\R\times \R^d:t_0-r^2<t<t_0,\|x-x_0\|_d<r\},
\end{equation*} 
whereas in our Theorem \ref{thm:smooth} we have a backward PDE and we redefined $C_r(z_0)$ accordingly. 
For $z_1=(t_1,x_1)$ and $z_2=(t_2,x_2)$ we denote the parabolic distance
\begin{equation*}
\ud(z_1,z_2)\coloneqq \|x_1-x_2\|_d+|t_1-t_2|^{\frac{1}{2}}.
\end{equation*}

\begin{lemma}[{\cite[Lem.\ 4.1]{krylov1981properties}}]\label{lem:krylov}
Let $r\le 1$ and $z_0\in\R^{d+1}$. Assume $u\in W^{1,2;d+1}(C_{2r}(z_0))$ satisfies $\cL u=0$, with $\cL$ as in \eqref{eq:kry_PDE} and satisfying \eqref{eq:kry_ass}. Then, there exist constants $N=N(\mu,d)>0$ and $\alpha=\alpha(\mu,d)\in(0,1)$ such that for any $z_1,z_2\in \overline{C_{r}(z_0)}$
\begin{equation}\label{eq:kry_estim}
|u(z_1)-u(z_2)|\le Nr^{-\alpha}\ud(z_1,z_2)^\alpha \sup_{z\in C_{2r}(z_0)}|u(z)|.
\end{equation}
If, moreover, $b\equiv 0$, then \eqref{eq:kry_estim} is true for $r\ge1$ as well.
\end{lemma}

\section*{Acknowledgement}
We would like to thank Bertrand Lods for many interesting discussions on this work and for pointing us to \cite{krylov1981properties} and related literature. S.\ Villeneuve warmly thanks Dylan Possama\"i for their numerous discussions over many years on dynamic contracting models with state constraints. We also thank Salvatore Federico for useful comments on the first draft of the paper. A.\ Bovo and T.\ De Angelis were partially supported by EU -- Next Generation EU -- PRIN2022 (2022BEMMLZ) and PRIN-PNRR2022 (P20224TM7Z). S.\ Villeneuve was partially supported by the LTI@UniTo Research Fellowship funded by Collegio Carlo Alberto (Turin). S.\ Villeneuve also acknowledges support from the FDR-SCOR {\it Risk Markets and Value Creation} chair and ANR (Grant ANR-17-EUR-0010).

\bibliographystyle{plain}
\bibliography{Bibliography}

\begin{thebibliography}{10}

\bibitem{abijaber2025gaussian}
E.~Abi~Jaber and S.~Villeneuve.
\newblock Gaussian agency problems with memory and linear contracts.
\newblock {\em Finance Stoch.}, 29:143--176, 2025.

\bibitem{aid2022optimal}
R.~A{\"\i}d, D.~Possama{\"\i}, and N.~Touzi.
\newblock Optimal electricity demand response contracting with responsiveness
  incentives.
\newblock {\em Math.\ Oper.\ Res.}, 47(3):2112--2137, 2022.

\bibitem{baldi2017stochastic}
P.~Baldi.
\newblock {\em Stochastic calculus}.
\newblock Universitext. Springer, 2017.

\bibitem{biais2010large}
B.~Biais, T.~Mariotti, J.-C. Rochet, and S.~Villeneuve.
\newblock Large risks, limited liability, and dynamic moral hazard.
\newblock {\em Econometrica}, 78(1):73--118, 2010.

\bibitem{bouchard2012dynamic}
B.~Bouchard and M.~Nutz.
\newblock Weak dynamic programming for generalized state constraints.
\newblock {\em SIAM J.\ Control Optim.}, 50(6):3344--3373, 2012.

\bibitem{bouchard2011weak}
B.~Bouchard and N.~Touzi.
\newblock Weak dynamic programming principle for viscosity solutions.
\newblock {\em SIAM J.\ Control Optim.}, 49(3):948--962, 2011.

\bibitem{brezis2010functional}
H.~Brezis.
\newblock {\em Functional Analysis, Sobolev Spaces and Partial Differential
  Equations}.
\newblock Universitext. Springer, New York, 2011.

\bibitem{briand2007one}
P.~Briand, J.-P. Lepeltier, and J.~San~Martin.
\newblock One-dimensional backward stochastic differential equations whose
  coefficient is monotonic in $y$ and non-{L}ipschitz in $z$.
\newblock {\em Bernoulli}, 13(1):80--91, 2007.

\bibitem{caffarelli1996viscosity}
L.~Caffarelli, M.G. Crandall, M.~Kocan, and A.~\'Swi\c{e}ch.
\newblock On viscosity solutions of fully nonlinear equations with measurable
  ingredients.
\newblock {\em Comm.\ Pure Appl.\ Math.}, 49(4):365--398, 1996.

\bibitem{claisse2016pseudo}
J.~Claisse, D.~Talay, and X.~Tan.
\newblock A pseudo-{M}arkov property for controlled diffusion processes.
\newblock {\em {SIAM} J.\ Control Optim.}, 54(2):1017--1029, 2016.

\bibitem{crandall1992user}
M.G. Crandall, H.~Ishii, and P.-L. Lions.
\newblock User's guide to viscosity solutions of second order partial
  differential equations.
\newblock {\em Bull.\ Amer.\ Math.\ Soc.\ (N.S.)}, 27(1):1--67, 1992.

\bibitem{crandall1999existence}
M.G. Crandall, M.~Kocan, P.-L. Lions, and A.~\'Swi\c{e}ch.
\newblock Existence results for boundary problems for uniformly elliptic and
  parabolic fully nonlinear equations.
\newblock {\em Electron.\ J.\ Differential Equations}, 1999:1--20, 1999.

\bibitem{crandall2000Lptheory}
M.G. Crandall, M.~Kocan, and A.~\'Swi\c{e}ch.
\newblock ${L}^p$-theory for fully nonlinear uniformly parabolic equations.
\newblock {\em Comm.\ Partial Differential Equations}, 25(11-12):1997--2053,
  2000.

\bibitem{cvitanic2018dynamic}
J.~Cvitani{\'c}, D.~Possama{\"\i}, and N.~Touzi.
\newblock Dynamic programming approach to principal--agent problems.
\newblock {\em Finance Stoch.}, 22:1--37, 2018.

\bibitem{dayanik2003optimal}
S.~Dayanik and I.~Karatzas.
\newblock On the optimal stopping problem for one-dimensional diffusions.
\newblock {\em Stochastic Process.\ Appl.}, 107(2):173--212, 2003.

\bibitem{degiorgi1956diff}
E.~De~Giorgi.
\newblock Sull'analiticit\`a delle estremali degli integrali multipli.
\newblock {\em Atti Accad. Naz. Lincei Rend. Cl. Sci. Fis. Mat. Natur.},
  8(20):438--441, 1956.

\bibitem{degiorgi1957diff}
E.~De~Giorgi.
\newblock Sulla differenziabilit\`a e l'analiticit\`a delle estremali degli
  integrali multipli regolari.
\newblock {\em Mem. Accad. Sci. Torino Cl. Sci. Fis. Mat. Natur.}, 3(3):25--43,
  1957.

\bibitem{demarzo2006optimal}
P.M.\ DeMarzo and Y.~Sannikov.
\newblock Optimal security design and dynamic capital structure in a
  continuous-time agency model.
\newblock {\em J.\ Finance}, 61(6):2681--2724, 2006.

\bibitem{edmans2016executive}
A.~Edmans and X.~Gabaix.
\newblock Executive compensation: a modern primer.
\newblock {\em J.\ Econom.\ Lit.}, 54(4):1232--1287, 2016.

\bibitem{eleuch2021optimal}
O.~El~Euch, T.~Mastrolia, M.~Rosenbaum, and N.~Touzi.
\newblock Optimal make-take fees for market making regulation.
\newblock {\em Math.\ Finance}, 31(1):109--148, 2021.

\bibitem{elie2021mean}
R.~{\'E}lie, E.~Hubert, T.~Mastrolia, and D.~Possama{\"\i}.
\newblock Mean-field moral hazard for optimal energy demand response
  management.
\newblock {\em Math.\ Finance}, 31(1):399--473, 2021.

\bibitem{elie2019tale}
R.~{\'E}lie, T.~Mastrolia, and D.~Possama{\"\i}.
\newblock A tale of a principal and many, many agents.
\newblock {\em Math.\ Oper.\ Res.}, 44(2):440--467, 2019.

\bibitem{evans10partial}
L.C. Evans.
\newblock {\em Partial differential equations}, volume~19 of {\em Graduate
  Studies in Mathematics}.
\newblock American Mathematical Society, Providence, RI, second edition, 2010.

\bibitem{evans2015measure}
L.C. Evans and R.F. Gariepy.
\newblock {\em Measure Theory and Fine Properties of Functions, Revised
  Edition}, volume~1 of {\em Applications of {M}athematics}.
\newblock Chapman and {H}all/{CRC}, 2015.

\bibitem{fabbri2017stochastic}
G.~Fabbri, F.~Gozzi, and A.~Swiech.
\newblock Stochastic optimal control in infinite dimension.
\newblock {\em Probability and Stochastic Modelling. Springer}, 2017.

\bibitem{fleming2012deterministic}
W.H. Fleming and R.W. Rishel.
\newblock {\em Deterministic and stochastic optimal control}, volume~1 of {\em
  Applications of {M}athematics}.
\newblock Springer New York, NY, 1975.

\bibitem{friedman2008partial}
A.~Friedman.
\newblock {\em Partial differential equations of parabolic type}.
\newblock Dover Publications, Mineola, NY, 2008.

\bibitem{gibbons1992optimal}
R.~Gibbons and K.J. Murphy.
\newblock Optimal incentive contracts in the presence of career concerns:
  Theory and evidence.
\newblock {\em J.\ Polit.\ Econ.}, 100(3):468--505, 1992.

\bibitem{holmstrom1987aggregation}
B.~Holmstr{\"o}m and P.~Milgrom.
\newblock Aggregation and linearity in the provision of intertemporal
  incentives.
\newblock {\em Econometrica}, 55(2):303--328, 1987.

\bibitem{innes1990limited}
R.D. Innes.
\newblock Limited liability and incentive contracting with ex--ante action
  choices.
\newblock {\em J.\ Econom.\ Theory}, 52(1):45--67, 1990.

\bibitem{karatzas2014brownian}
I.~Karatzas and S.E. Shreve.
\newblock {\em Brownian Motion and Stochastic Calculus}, volume 113 of {\em
  Graduate Texts in Mathematics}.
\newblock Springer New York, NY, 2 edition, 2014.

\bibitem{krvsek2026randomisation}
D.~Kr{\v{s}}ek and D.~Possama{\"\i}.
\newblock Randomisation with moral hazard: a path to existence of optimal
  contracts.
\newblock {\em Ann.\ Appl.\ Probab.}, 36(4):3206--3265, 2026.

\bibitem{krylov2008lectures}
N.V. Krylov.
\newblock {\em Lectures on elliptic and parabolic equations in {S}obolev
  spaces}, volume~96 of {\em Graduate Studies in Mathematics}.
\newblock American Mathematical Society, Providence, RI, 2008.

\bibitem{krylov2009controlled}
N.V. Krylov.
\newblock {\em Controlled diffusion processes}, volume~14 of {\em Stochastic
  Modelling and Applied Probability}.
\newblock Springer Berlin, 2009.
\newblock Translated from the 1977 Russian original by A.B. Aries, Reprint of
  the 1980 edition.

\bibitem{krylov2018sobolev}
N.V. Krylov.
\newblock {\em {S}obolev and viscosity solutions for fully nonlinear elliptic
  and parabolic equations}, volume 233 of {\em Mathematical surveys and
  monographs}.
\newblock American Mathematical Soc., 2018.

\bibitem{krylov1981properties}
N.V. Krylov and M.V. Safonov.
\newblock A certain property of solutions of parabolic equations with
  measurable coefficients.
\newblock {\em Math.\ {USSR}-Izv.}, 16(1):151--164, 1981.

\bibitem{lin2022random}
Y.~Lin, Z.~Ren, N.~Touzi, and J.~Yang.
\newblock Random horizon principal--agent problems.
\newblock {\em SIAM J.\ Control Optim.}, 60(1):355--384, 2022.

\bibitem{marinovic2019ceo}
I.~Marinovic and F.~Varas.
\newblock {CEO} horizon, optimal pay duration, and the escalation of
  short-termism.
\newblock {\em J.\ Finance}, 74(4):2011--2053, 2019.

\bibitem{mastrolia2018moral}
T.~Mastrolia and D.~Possama{\"\i}.
\newblock Moral hazard under ambiguity.
\newblock {\em J.\ Optim.\ Th.\ Appl.}, 179(2):452--500, 2018.

\bibitem{mirrlees1976optimal}
J.A. Mirrlees.
\newblock The optimal structure of incentives and authority within an
  organization.
\newblock {\em Bell J.\ Econ.}, 7(1):105--131, 1976.

\bibitem{nash1958continuity}
J.~Nash.
\newblock Continuity of solutions of parabolic and elliptic equations.
\newblock {\em Amer.\ J.\ Math}, 80(4):931--954, 1958.

\bibitem{niculescu2006convex}
C.~Niculescu and L.-E. Persson.
\newblock {\em Convex functions and their applications}, volume~23 of {\em CMS
  Books in Mathematics}.
\newblock Springer, 2006.

\bibitem{peskir2006optimal}
G.~Peskir and A.~Shiryaev.
\newblock {\em Optimal stopping and free-boundary problems}.
\newblock Lectures in Mathematics. ETH Z\"urich. Birkh\"{a}user Verlag, Basel,
  2006.

\bibitem{possamai2025golden}
D.~Possama{\"\i} and N.~Touzi.
\newblock Is there a golden parachute in {S}annikov's principal--agent problem?
\newblock {\em Math.\ Oper.\ Res.}, 50(2):1173--1203, 2025.

\bibitem{revuz2013continuous}
D.~Revuz and M.~Yor.
\newblock {\em Continuous martingales and Brownian motion}, volume 293 of {\em
  Comprehensive Studies in Mathematics}.
\newblock Springer Berlin, Heidelberg, 3 edition, 2013.

\bibitem{sannikov2008continuous}
Y.~Sannikov.
\newblock A continuous-time version of the {P}rincipal--{A}gent problem.
\newblock {\em The Review of Economic Studies}, 75(3):957--984, 2008.

\bibitem{sappington1983limited}
D.~Sappington.
\newblock Limited liability contracts between principal and agent.
\newblock {\em J.\ Econom.\ Theory}, 29(1):1--21, 1983.

\bibitem{taylor2023nonlinear}
M.E. Taylor.
\newblock {\em Partial Differential Equations III: Nonlinear Equations}, volume
  117 of {\em Applied Mathematical Sciences}.
\newblock Springer Cham, 3 edition, 2023.

\end{thebibliography}
\end{document}